\documentclass[a4paper,10pt]{article}

\usepackage[utf8]{inputenc}
\usepackage[T1]{fontenc}  
\usepackage[english]{babel}  
\usepackage{graphicx}       
\usepackage{subcaption}   
\usepackage{geometry}       
\usepackage{caption}       
\usepackage{amsfonts}       
\usepackage{float} 
\usepackage{xcolor}
\usepackage{xcolor}         
\usepackage{fancyhdr}       
\usepackage{graphicx}
\usepackage{subcaption}
\usepackage{amsmath,amssymb,amsfonts,amsthm,physics,bm,dsfont}
\usepackage{authblk}
\usepackage{tabularx}
\usepackage{dblfloatfix}
\usepackage{todonotes}
\usepackage{bigints}

\usepackage{tikz}
\usepackage[colorlinks=true,plainpages=true,urlcolor=blue,anchorcolor=red, citecolor=blue, linkcolor=blue]{hyperref}

\usepackage{url}            
\usepackage{booktabs}
\usepackage{nicefrac}       
\usepackage{microtype}      
\usepackage{lipsum}
\usepackage[english]{babel}
\usepackage{ragged2e}
\usepackage{epsfig}
\usepackage{moreverb}
\usepackage{fancybox}
\usepackage[sectionbib]{chapterbib}
\usepackage{cellspace}
\usepackage{subcaption}
\usepackage{caption}
\usepackage{siunitx}
\usepackage{adjustbox} 
\usepackage{soul}

\usepackage{booktabs}
\usepackage{multirow}
\usepackage{threeparttable}
\usepackage{enumitem}

\newtheorem{theorem}{Theorem}
\newtheorem{definition}[theorem]{Definition}

\newtheorem{condition}[theorem]{Condition}
\newtheorem{proposition}[theorem]{Proposition}
\newtheorem{lemma}[theorem]{Lemma}

\newtheorem{remark}[theorem]{Remark}

\newcommand{\id}{\mathds{1}}

\newcommand{\dx}{0pt}  
\newcommand{\dy}{3pt}  

\newcommand{\quart}[3]{%
	\ifthenelse{\equal{#2}{tl}}{
		\adjincludegraphics[width=#3\textwidth,clip,
		trim={0 {\dimexpr.5\height+\dy\relax} {\dimexpr.5\width-\dx\relax} 0}]{#1}%
	}{}%
	\ifthenelse{\equal{#2}{tr}}{
		\adjincludegraphics[width=#3\textwidth,clip,
		trim={{\dimexpr.5\width+\dx\relax} {\dimexpr.5\height+\dy\relax} 0 0}]{#1}%
	}{}%
	\ifthenelse{\equal{#2}{bl}}{
		\adjincludegraphics[width=#3\textwidth,clip,
		trim={0 0 {\dimexpr.5\width-\dx\relax} {\dimexpr.5\height-\dy\relax}}]{#1}%
	}{}%
	\ifthenelse{\equal{#2}{br}}{
		\adjincludegraphics[width=#3\textwidth,clip,
		trim={{\dimexpr.5\width+\dx\relax} 0 0 {\dimexpr.5\height-\dy\relax}}]{#1}%
	}{}%
}

\newcommand{\half}[3]{%
	\ifthenelse{\equal{#2}{t}}{
		\adjincludegraphics[width=#3\textwidth,clip,
		trim={0 {\dimexpr.5\height+\dy\relax} 0 0}]{#1}%
	}{}%
	\ifthenelse{\equal{#2}{r}}{
		\adjincludegraphics[width=#3\textwidth,clip,
		trim={{\dimexpr.5\width+\dx\relax} 0 0 0}]{#1}%
	}{}%
	\ifthenelse{\equal{#2}{b}}{
		\adjincludegraphics[width=#3\textwidth,clip,
		trim={0 0 {\dimexpr.5\height-\dy\relax} 0}]{#1}%
	}{}%
	\ifthenelse{\equal{#2}{l}}{
		\adjincludegraphics[width=#3\textwidth,clip,
		trim={0 0 {\dimexpr.5\width-\dx\relax}  0}]{#1}%
	}{}%
}

\newcommand{\gapx}{2pt} 
\newcommand{\gapy}{2pt} 

\renewcommand{\dx}{0pt}
\renewcommand{\dy}{0pt}

\renewcommand{\arraystretch}{0}  
\title{Arbitrary  order stationarity preserving   stabilized finite elements for multidimensional nonlinear hyperbolic problems. \\
Application to the  Euler equations with gravity}
\author[$\circ,\star$]{M. Ziggaf}
\author[+]{D. Torlo}
\author[$\star$]{M. Ricchiuto}
\affil[$\circ$]{LaR2A Laboratory, Dep. of Mathematics, Faculty of Sciences, Abdelmalek Essaadi Univ., B.P. 2121, Tetouan, Morocco}
\affil[$\star$]{Inria, Univ. Bordeaux, CNRS, Bordeaux INP, IMB, UMR 5251, F-33400 Talence, France}
\affil[+]{Dip. di Matematica G. Castelnuovo, Universit\`a di Roma La Sapienza, Roma, Italy}

\date{ }

\begin{document}

\maketitle

\begin{abstract}
We develop arbitrarily high-order, stationarity-preserving stabilized finite element methods for multidimensional nonlinear hyperbolic balance laws on Cartesian grids. 
We aim at approximating all the steady states of the problem at hand, including non-trivial genuinely multidimensional equilibria, with a level of accuracy higher than the nominal one of the underlying scheme.
We formalize  more precisely the meaning  of  stationarity preservation, providing some technical conditions for its realizability.
We then recast  the multidimensional global-flux quadrature of Barsukow et al. (Num. Meth. PDEs, 2025) as a local preprocessing of the physical fluxes that maps continuous polynomial vector fields 
to a local space with Raviart--Thomas-type structure.
Both the Galerkin and SUPG formulations are recast in this setting. The resulting methods  extend the stationarity-preserving finite-volume approach of Barsukow et al. (J. Comput. Phys., 2026) to high-order continuous finite elements and Barsukow et al. (Num. Meth. PDEs, 2025) to nonlinear balance laws. We analyze key properties of the proposed schemes, including local conservation and nodal superconvergence of the discrete steady kernel, and we discuss their relation to low-Mach-compliant discretizations.

We apply the framework to the compressible Euler equations with gravity. A simple source-term reformulation yields machine-precision preservation of isothermal hydrostatic equilibria. Extensive numerical benchmarks, including moving equilibrium, near-equilibrium, and instability-dominated regimes, demonstrate clear improvements in robustness and accuracy over standard SUPG and reference finite-volume methods.
\end{abstract}

\section{Generalities}

This paper focuses on the construction of  enhanced discretizations for multi-dimensional hyperbolic balance laws. 
We focus here on  the Euler equations for a perfect gas, including gravitational effects, although many ideas can be used for other models.
As a prototype example, we will thus consider the system on a spatial domain $\Omega \subset \mathbb{R}^d$ for $d\geq 1$ and for $t>0$:
\begin{equation}\label{Euler_Eq0}
\begin{cases}
	\partial_t \rho + \nabla \cdot (\rho \mathbf{v}) = 0, \\
	\partial_t (\rho \mathbf{v}) + \nabla \cdot (\rho \mathbf{v} \otimes \mathbf{v}) + \grad{p} = -\rho \grad{\phi}, \\
	\partial_t (\rho E) + \nabla \cdot (\rho H \mathbf{v}) = -\rho \mathbf{v} \cdot \grad{\phi},
\end{cases}
\end{equation}
where $\rho : \Omega \times \mathbb R^+ \to \mathbb R^+ $ denotes the fluid density, $\mathbf{v} = (  v_1,\dots,v_d): \Omega \times \mathbb R^+ \to \mathbb R^d$ the velocity vector, $E: \Omega \times \mathbb R^+ \to \mathbb R^+ $ and $H: \Omega \times \mathbb R^+ \to \mathbb R^+ $ are the specific total energy and enthalpy respectively,
defined as
\begin{equation}\label{Euler_Eq1}
E = e + \dfrac{1}{2}|\mathbf{v}|^2\;,\quad H = E + \dfrac{p}{\rho}\,.
\end{equation}
The above relations are closed by the equation of state. In this work, we will consider perfect gases for which
the pressure $p$, density $\rho$, and  internal energy $e$ are related by
\begin{equation}\label{Euler_Eq2}
 p = (\gamma - 1)\rho e 
 \end{equation}
with $\gamma$ the ratio of specific heats (e.g., $\gamma = 1.4$ for air). We  denoted by $\phi : \Omega \to \mathbb R$, the gravitational potential,
which in this work is assumed to be a given function of  $\mathbf{x} =( x^1,\dots,x^d)\in\mathbb{R}^d$, the spatial coordinate vector. This paper focuses on the two-dimensional case $d=2$. The system can be  written in compact form as 
\begin{equation}\label{Euler_Eq3}
	\partial_t \bm{W} + \nabla\cdot  \mathbf{F}  = \mathbf{S}(\bm{W};\mathbf{x}),
\end{equation}
with $\bm{W}\in \mathbb R^{s}$ the conserved variables, with $s=d+2$ denoting the size of the system, and  $\mathbf{F}\in \mathbb{R}^{s}\times\mathbb{R}^{d}$ the flux tensor and $\mathbf{S}\in \mathbb{R}^{s}$ the source term.  For  $d=2$  we have
$\mathbf{F}=(\mathbf{F}_1,\mathbf{F}_2)$ with
\begin{equation}\label{Euler_Eq4}
	\bm{W} =
	\begin{pmatrix}
		\rho \\[5pt]
		\rho v_1 \\[5pt]
		\rho v_2 \\[5pt]
		\rho E
	\end{pmatrix}, \qquad
	\mathbf{F}_1 (\bm{W}) =
	\begin{pmatrix}
		\rho v_1 \\[5pt]
		\rho v_1^2 + p \\[5pt]
		\rho v_1 v_2 \\[5pt]
		\rho v_1 H
	\end{pmatrix}, \qquad
	\mathbf{F}_2 (\bm{W}) =
	\begin{pmatrix}
		\rho v_2 \\[5pt]
		\rho v_1 v_2 \\[5pt]
		\rho v_2^2 + p \\[5pt]
		\rho v_2 H
	\end{pmatrix},
\end{equation}
and
\begin{equation}\label{Euler_Eq4b}
	\mathbf{S}(\bm W;\mathbf x)=
	\begin{pmatrix}
		0 \\[5pt]
		-\rho\,\partial_{x^1}\phi \\[5pt]
		-\rho\,\partial_{x^2}\phi \\[5pt]
		-\rho\,\mathbf v\cdot \nabla\phi
	\end{pmatrix}.
\end{equation}
The system is fully hyperbolic, and in particular  $\forall \vec n \in\mathbb{R}^d$ with $|\vec n|=1$,
the flux Jacobians
\begin{equation}\label{Euler_Eq5}
\mathbf{J}_{\vec n} = \sum_{j=1}^d \mathbf{J}_{j} n_j\in \mathbb R^{ s\times s},\quad \text{with }
\mathbf{J}_{j} := \dfrac{\partial\mathbf{F}_j}{\partial\bm{W}} \in \mathbb R^{ s\times s} \text{ for }j=1,\dots, d, 
\end{equation}
admit in  2d the  real eigenvalues $( \mathbf{v} \cdot \vec n - \sqrt{\gamma p/\rho},  \mathbf{v} \cdot \vec n,  \mathbf{v} \cdot \vec n,  \mathbf{v} \cdot \vec n +  \sqrt{\gamma p/\rho} )$, and a full set of linearly independent eigenvectors.  In the following, we will also sometimes denote by $\mathbf{J}$ the order-three-tensor
\begin{equation}\label{Euler_Eq6b}
\mathbf{J} = (\mathbf{J}_1,\dots,\mathbf{J}_d) \in \mathbb R^{d\times s \times s}.
\end{equation}

System~\eqref{Euler_Eq0} admits a rich ensemble  of  non-trivial stationary states that are governed by  the constraints
\begin{equation}\label{Euler_Eq6}
\begin{split}
	 \nabla \cdot (\rho \mathbf{v}) =  &0, \\
	 \nabla \cdot (\rho \mathbf{v} \otimes \mathbf{v}) + \grad{p} = &-\rho \grad{\phi}, \\
	  \nabla \cdot (\rho H \mathbf{v}) = &-\rho \mathbf{v} \cdot \grad{\phi}.
\end{split}
\end{equation}
These constraints can be written in many different forms, all equivalent to the one above in the smooth case.
In particular,  steady states satisfy
\begin{equation}\label{eq:stationary-state}
	\nabla \cdot (\rho \mathbf{v}) = 0, \qquad
	(\rho \mathbf{v} \cdot \grad)\mathbf{v} + \grad p = -\rho \grad\phi, \qquad
	\mathbf{v} \cdot \grad (H + \phi) = 0,
\end{equation}
which correspond respectively to mass conservation, momentum balance, and preservation of the total enthalpy along streamlines. 

An interesting family of solutions is given by  hydrostatic equilibria characterized by:
\begin{equation}
	\mathbf{v} \equiv 0, 
	\qquad \grad p = - \rho \grad \phi,
\end{equation}
which can be satisfied only when the following compatibility condition holds:
\begin{equation}
	\nabla \times (\rho \grad \phi) = 0.
\end{equation}

Other equilibria where the velocity is not zero are also relevant in many applications. For example, isentropic vortices are stationary states of the system, and are often used as test cases for numerical methods \cite{ricchiuto2021analytical,barsukow2025genuinely}. In these vortices, the velocity field is non-zero and has a rotational structure, while the pressure and density fields are arranged in such a way that the centrifugal forces are balanced by the pressure gradient, resulting in a steady state.

\subsection{Well balanced or stationarity preservation?}

Identifying and preserving stationary states at the discrete level is crucial in many physical applications, 
particularly when flows close to equilibrium are investigated, for which  spurious numerical errors can spoil the true dynamics. 
One needs, however, to be precise concerning the notion of stationarity preservation, 
as the literature is rich in contributions related to this topic. 

For balance laws in one space dimension, the notion of well balanced schemes is rooted in early works
by P.L. Roe \cite{Roe87},  Bermudez-Vazquez \cite{bv94}, and  LeRoux-Greenberg \cite{doi:10.1137/0733001}, among others. 
Since then the notion of discrete  preservation of non-trivial equilibria  has  been 
refined in the context of all numerical methods: finite  differences
\cite{Li2015,Li2020,Xu2025,Wang2022,Zhang2023,Ren2024}, finite volumes  \cite{Audusse2004,Castro2008,Noelle2007,Kurganov2020,Ciallella2022,Chertock2022,Ciallella2023,Zhang2025}, 
discontinuous Galerkin    \cite{Xing2014,Liu2022,Zhang2023,Zhang2025,Yang2021,mantri2024fully},  finite elements 
\cite{Azerad2017,Guermond2018,Behzadi2020, Knobloch2024,Micalizzi2024}, residual distribution   
\cite{Ricchiuto2009,Chang2023,Ricchiuto2015,Ricchiuto2011,Arpaia2018,Arpaia2020}, Active Flux  \cite{Barsukow2024,Abgrall2024,Liu2025}, and so on.
It is important to highlight the existence of two standpoints. 
On one hand, methods which are able to preserve \textit{exact} steady states in some projected form (cell average, pointwise values, moments, etc.).
On the other, the idea of preserving some discrete approximation of the steady state, obtained with some enhanced numerical method,
possibly with higher order of consistency accuracy w.r.t. the underlying approach, see \emph{inter alia} 
\cite{math9151799,GOMEZBUENO2021125820,GOMEZBUENO202318,Ciallella2023,Mantri2024,Kazolea2025,AbgLiuRic25},  based on  discrete minimization procedures, or on high order ODE solvers. 
In this case, one could speak of \textit{approximate} well-balanced, or \textit{discrete} well-balanced.\\

In multiple space dimensions,  several works have shown how to treat solutions at rest with hydrostatic  equilibrium pressure/potential,
see e.g. \cite{Audusse2004,THOMANN2020109723,CZ17,chertock2018well} for shallow water, and Euler with gravitation.
Some works have also considered the preservation of one-dimensional  moving solutions in a two-dimensional setting 
\cite{MICHELDANSAC2021105152,Ricchiuto2011,Ricchiuto2015,Ciallella2025,Mantri2024}.  
To study perturbations of \emph{a-priori given} multidimensional equilibria, the  so called  deviation method, or  reconstruction of fluctuations,
can be used. In this case,  the numerical unknown is the deviation w.r.t. the given equilibrium,  and the numerical discretization is corrected so that the   equilibrium   is preserved when the deviation is zero (see e.g.
\cite{doi:10.1137/060674879,doi:10.1137/18M1196704,GCD18,Ciallella2022,Dumbser2024,BERBERICH2021104858}).
In parallel, some studies have focused on the rotating shallow water equations, emphasizing the preservation of the linear geostrophic equilibrium under the influence of Coriolis forces (see, for example, \cite{audusse2025energy,delpino2025} and references therein), which represents a specific yet multidimensional equilibrium state.

This paper deals with a more challenging question. Following works done in one dimension, we want to design 
discretizations that are compatible with all possible genuinely multidimensional stationary states.
Even in the homogeneous case, the difficulty is the  form of the first equation in \eqref{eq:stationary-state}, which requires embedding the numerical method
with a solenoidal preserving condition. 
For hyperbolic systems, one  key of the  problem is numerical dissipation. 
For linear acoustics, which are a linearization of \eqref{Euler_Eq0},  stationary states are defined by a constant pressure, and 
the constraint $\nabla\cdot \mathbf{v} =0$.  
In this case,  previous works 
\cite{barsukow2019stationarity,barsukow2025structure} have clarified that numerical dissipation based on some discrete Laplacian of the main unknowns
is incompatible  with the preservation of the solenoidal condition, unless the velocity is constant. 
The numerical dissipation must in this case be based on a grad(div) vector Laplacian $\nabla ( \mu(\Delta x) \nabla\cdot \mathbf{v}).$ 
This structural condition is extremely important, as it also relates to
the low Mach behaviour of \eqref{Euler_Eq0}, as shown in \cite{jung2022steady,jung2024behavior}. 
In particular, schemes that have 
this structure are in principle well behaved in the low Mach limit. In \cite{jung2022steady,jung2024behavior,barsukow2019stationarity,barsukow2025structure} it is also shown that  
this condition and the preservation of a curl-type involution  for the  velocity are also related.
However, we still  need a   definition of  what is a ``\emph{stationarity preserving}''  scheme.


\subsubsection{Stationarity preservation: definition}
For  linear   acoustics, an interesting definition is provided in \cite{barsukow2019stationarity}:   a stationarity preserving scheme is 
one that admits   a rich set of discrete steady states including a  discretization of $\nabla \cdot \mathbf{v} = 0$, 
without additional constraints.     
The definition of rich set given in  \cite{barsukow2019stationarity} relies strongly on the analysis of the problem in Fourier space. 
A rigorous definition  in physical  space is provided in \cite{Barsukow_nodeconservative2025},   which is somewhat 
closer to the approximate well-balanced approach introduced before.

Here, we first introduce a more general definition of stationarity preserving methods and then we try to characterize it through some necessary conditions.

\begin{definition}[Stationarity preserving (SP) discretization]\label{def:SP_original}
  Consider an evolutionary hyperbolic problem \eqref{Euler_Eq3} admitting a rich family of equilibria $\mathcal{W}$ such that for all $\mathbb W \in \mathcal{W}$ we have that $\nabla \cdot \mathbb{F}(\mathbb W) = \mathbf{S}(\mathbb W)$.
  Consider a discrete set of $N_h$ nodes $\mathbf{x}_\alpha$ of the geometry, for $\alpha =1,\dots, N_h$, and an approximation of the solution 
  at these nodes $\mathbf{W}_\alpha  \approx \mathbf{W}(\mathbf{x}_\alpha)$ for $\alpha = 1,\dots, N_h$.
  Consider a discretization of \eqref{Euler_Eq3} evolving the array $\lbrace \mathbf{W}_\alpha \rbrace_{\alpha=1}^{N_h}$ according to
  \begin{equation}\label{eq:general_num_method}
  \sum_{\beta =1}^{N_h} M_{\alpha \beta} \dfrac{d\bm{W}_\beta }{dt} + \mathsf{R}_\alpha =0 \quad\text{plus boundary conditions}
  \end{equation}
  for some 
matrix  $M$ of entries $M_{\alpha \beta}\in \mathbb R$ and $\mathsf{R}_\alpha \in \mathbb R^{s}\,,\; \forall \alpha$, an evolution discrete operator.
The method \eqref{eq:general_num_method} is said to be \emph{stationarity preserving} (SP) if for all $\mathbb W \in \mathcal{W}$ there exists some discrete
approximation $\mathbf W_h$ such that $\mathbf{W}_h(\mathbf{x}_\alpha)=\mathbf{W}_\alpha$, 
 $\mathsf{R}_\alpha(\mathbf W_h) =0$ for all $\alpha$, and $\lim_{h\to 0} \lVert \mathbf W_h - \mathbb W \rVert = 0$.
\end{definition}

We highlight here that this definition is very close to the concept of fully approximately well-balanced schemes  (see \cite{Kazolea2025,mantri2024fully} and references therein). 
The main difference is the context of application. Typically, fully well-balanced schemes are designed in one-dimensional problems and are driven by boundary conditions, solutions that an ODE solver can approximate.
Here, we extend this concept also to multi-dimensional problems that are not necessarily driven by boundary conditions, for example, compact isentropic vortices, but not only.
It is also important  to note that of course the condition $\mathsf{R}_\alpha=0$ necessarily  involves some members of the  family of equilibria $\mathcal{W}$, however not necessarily
all of them. An example of a method not verifying the above condition is discussed later in the paper.

Clearly, the above   definition is very general and does not give any design criterion for numerical methods. 
We introduce a condition for a scheme that will characterize the schemes that are stationarity preserving. At the moment, the authors are not sure if this condition is necessary, nor if it is sufficient, but it seems a good starting point to characterize stationarity preserving schemes.  
 
\begin{condition}[Dual residuals  condition]\label{cond:SP} 
Necessary conditions  for Definition \ref{def:SP_original} to be realizable are the following 
\begin{enumerate}[label={C.\arabic*}]
\item \label{item:dual_residual_primal_residual} there exists another collection of geometrical entities identified in the points $\tilde {\mathbf{x}}_\gamma$, for $\gamma =1,\dots, \tilde{N}_h$ not necessarily coinciding
with the solution location $\mathbf{x}_\alpha$, and at these points $\tilde {\mathbf{x}}_\gamma$ are associated discrete approximations of the residual $\Psi_\gamma \approx (\nabla\cdot  \mathbf{F}  - \mathbf{S})_{\gamma}$ for $\gamma =1,\dots, \tilde{N}_h$, such that \begin{equation}\label{eq:cond_implication}
\Psi_\gamma=0  \;\forall\,\gamma =1,\dots, \tilde{N}_h \Rightarrow  \mathsf{R}_\alpha =0 \,\;\forall\, \alpha = 1, \dots, N_h;
\end{equation}
\item \label{item:SP_number_of_unknowns} $N_h \geq  \tilde{N}_h$, i.e., the total number of discrete unknowns is larger (or equal) than the number of   independent  relations  provided by the  local stationarity conditions $\Psi_\gamma=0 \;\forall\,\gamma =1,\dots, N_h.$
\end{enumerate}
 \end{condition}

Although it would be desirable to have a necessary and sufficient condition for stationarity preservation, it is much harder to get an equivalent condition to Definition~\ref{def:SP_original}. In the following, we will focus on Condition~\ref{cond:SP} and we will show that the proposed methods verify such a condition.

To clarify the previous condition, we provide a couple of examples of nodes $\mathbf{x}_\alpha$ and residuals $\Psi_\gamma$ for $\gamma =1,\dots, \tilde{N}_h$. 
In finite element, for a Cartesian grid, the $\mathbf{x}_\alpha$ can be the grid of points defined by the tensor product of the Gauss-Lobatto points in each direction, while the $\Psi_\gamma$ can be defined as the residuals of the PDE at the subcells defined by the Cartesian mesh within the grid of points inside each element.
In a finite volume setting, the $\mathbf{x}_\alpha$ can be the cell centers, while the $\Psi_\gamma$ can be defined as the residuals of the PDE at the cell interfaces or at corner points.

This is typically done when considering artificial diffusion terms, which are defined at the dual mesh level rather than at the solution points, e.g. at cells or subcells for FE and at interfaces for FV.\\

The above condition contains several important parts. Firstly, in general the evolution part of the scheme $\mathsf{R}_\alpha$ contains both a consistent approximation
of the residual $ (\nabla\cdot  \mathbf{F}  - \mathbf{S})_{\alpha}$, and some form of numerical dissipation. 
This is not true of course for centered schemes, to  which
the definition  above applies indeed with $\mathbf x_\alpha =\tilde{\mathbf x}_\alpha$ and $N_h=\tilde{N}_h$.  These schemes have however little interest for \eqref{Euler_Eq3} as soon as one considers explicit time stepping because of their instability.
So, we will rule out this option here. 

The second important point is that Condition~\ref{item:dual_residual_primal_residual} is a necessary one: there must exist some local discrete approximation
of the divergence (or divergence minus source), whose kernel belongs to kernels of \emph{both the consistent part of the scheme and of the numerical dissipation}. 
There are many methods that satisfy this necessary condition. Examples are  residual distribution (RD) schemes \cite{Abgrall2022,RD-ency,AR:17},
as well as certain types of so-called residual based compact  (RBC) schemes by Lerat and collaborators \cite{lerat,lerat2,Corre20071567,LERAT201231}. 
A very important point is that this necessary condition must hold at the fully discrete level.  This is even a stricter constraint.
For example, depending on the quadrature formulas, or on the definition of the discrete residuals used in the dissipation, 
not all  the formulations of RD and RBC schemes  verify this necessary condition. Similarly,  
stabilized finite elements, as e.g. the successful streamline upwind Petrov-Galerkin (SUPG)  method \cite{brooks82,hughes86,HuS:10}, do not verify this condition,
despite their fully residual character, and of the proper grad(div) structure of their numerical dissipation, 
as rigorously proved in \cite{barsukow2025structure} for acoustic equations.

Lastly, Condition~\ref{item:SP_number_of_unknowns} is a \emph{realizability condition}. It is necessary for a discrete stationary state satisfying all the constraints to actually exist.
It is the most difficult to satisfy in practice on general meshes. The Cartesian mesh setting is 
the easiest one to this end, although some exceptional configurations on non Cartesian meshes are possible, as shown in \cite{Barsukow_nodeconservative2025}.
To satisfy the Condition~\ref{item:SP_number_of_unknowns}, one should at least count the unknowns $N_h \cdot s$ and number of independent relations (plus the boundary  conditions constraints) and check whether there are solutions for $\mathsf{W}$ that are compatible with $\Psi_\gamma = 0$ for all $\gamma=1,\dots, \tilde{N}_h$, see e.g. \cite{barsukow2025structure,Barsukow_nodeconservative2025,barsukow2025genuinely,barsukow2025stationaritypreservingnodalfinite}.

\begin{remark}[Dual residuals implication for stationarity: necessary condition for SP]\label{ram:SP}
If the method \eqref{eq:general_num_method} verifies Condition~\ref{cond:SP}, then it has good chances of being stationarity preserving in the sense of Definition~\ref{def:SP_original}.
The difficult part in showing that Condition~\ref{cond:SP} is necessary for stationarity preservation is to pass from the residuals $\Psi_\gamma$ to the actual discrete solution $\mathbf{W}_h$. Indeed, for linear schemes, it is possible to directly put conditions on the solution variables, see \cite{barsukow2025structure}, but for nonlinear schemes this is less straightforward.
\end{remark}

\begin{remark}[SP and boundary conditions] \label{rem:SPBCs} \underline{Condition~\ref{cond:SP}  does not account for   boundary conditions}.
  Indeed, in multidimensional settings, many stationary states are not driven by boundary conditions, but rather by the balance of fluxes and sources in the interior of the domain, which are selected by the initial conditions.
In Condition~\ref{cond:SP}, we are thus considering states with constant  flow at infinity, or  periodic flows, or even local compactly supported states  over static background.
As an example, one can think of  a single  isolated vortex, or a set of isolated,  stationary  vortices. 

If one wants to include the effects of boundary conditions on the definition of the stationary states, the Condition~\ref{cond:SP} should be modified to include them in the right-hand sides of \eqref{eq:cond_implication}, as well as in the counting of the degrees of freedom in the second condition.
This is not the scope of this paper.
\end{remark}

\subsection{Multidimensional stationarity preserving/global flux quadrature}\label{sec:multid_GF}

In \cite{barsukow2025structure,barsukow2025stationaritypreservingnodalfinite,barsukow2025genuinely} some SP numerical methods have been introduced.
The key idea   proposed in these works to enforce the SP conditions 
is  to  systematically discretize  the divergence operator using  the following representation, in the 2D homogeneous case,  
\begin{equation}\label{eq:GF0}
\nabla\cdot  \mathbf{F} \equiv   \nabla\cdot \left( \partial_{x^2}\int^{x^2} \mathbf{F}_1(x^1,\eta)\mathrm{d}\eta,\, \partial_{x^1}\int^{x^1} \mathbf{F}_2(\xi,x^2)\mathrm{d}\xi \right) = \partial_{x^1 x^2} (\mathcal{F}_1 +\mathcal{F}_2)
\end{equation}
having set
\begin{equation}\label{eq:GF1}
\mathcal{F}_1:= \int^{x^2} \mathbf{F}_1(x^1,\eta)\mathrm{d}\eta\;,\quad \mathcal{F}_2:= \int^{x^1} \mathbf{F}_2(\xi,x^2)\mathrm{d}\xi,
\end{equation}
with a redefinition of the   flux vector  as  
\begin{equation}\label{eq:GF0a}
 \mathbf{F} = \left( \partial_{x^2}\mathcal{F}_1,\, \partial_{x^1} \mathcal{F}_2\right) .
\end{equation}
Though analytically these two formulations are equivalent, at the numerical level, we will introduce some errors in the integral and derivative process, which will make the two operators different. As we will see, in practice both the integration and the derivatives are performed using local 
(finite element) polynomials, providing essentially a local pre-processing of the fluxes. This representation has many interesting properties.
For example, in \cite{barsukow2025structure}, the integrals above are evaluated with some high-order quadrature formula. 

Now, the stationarity condition $\nabla \cdot \mathbf{F} = 0$ can be expressed in terms of the quantity
\begin{equation}\label{eq:GF2}
 \mathcal{G} := \mathcal{F}_1 +\mathcal{F}_2,
 \end{equation}
 and reads
 \begin{equation}\label{eq:GF3}
 \mathcal{G} = \Gamma_0+ \mathbf{f}(x^1) + \mathbf{g}(x^2), 
\end{equation}
for $\Gamma_0\in \mathbb R^{r}$, $\mathbf{f}, \mathbf{g}:\mathbb R \to \mathbb R^{r}$ arbitrary functions.  
It is important to see here that the univariate character of the functions $ \mathbf{f}$ and $ \mathbf{g}$ makes their control relatively easy via boundary conditions,
thus one can assume these functions to be zero at one boundary (constants being accounted for into $\Gamma_0$). Another interesting property is that, in the neighbourhood of a fixed location $\mathbf{x}_\star$,
we can consider the local residual  
  \begin{equation}\label{eq:GF4}
  \begin{aligned}
      \Psi_{\mathbf{x}_\star}(\mathbf{x})  := &
  \int\limits_{x^{2,\star}}^{x^2} \left( \mathbf{F}_1(x^1,\eta) -\mathbf{F}_1(x^{1,\star},\eta) \right) \mathrm{d}\eta +
  \int\limits_{x^{1,\star}}^{x^1} \left( \mathbf{F}_2(\xi,x^2) -\mathbf{F}_2(\xi,x^{2,\star}) \right)   \mathrm{d}\xi\\
  = &  \int\limits_{x^{1,\star}}^{x^1}  \int\limits_{x^{2,\star}}^{x^2}\nabla\cdot  \mathbf{F} \, \mathrm{d} \xi\, \mathrm{d} \eta \;.
  \end{aligned}
  \end{equation}
  Trivially, we can show that 
  $$
  \partial_{x^1x^2} \mathcal{G}   =   \partial_{x^1x^2}  \Psi_{\mathbf{x}_\star},
  $$
  so  that the representation \eqref{eq:GF0} is also equivalent to (in the neighbourhood of   $\mathbf{x}_\star$ fixed)
    \begin{equation}\label{eq:GF5}
    \nabla\cdot  \mathbf{F} \equiv   \partial_{x^1x^2}  \Psi_{\mathbf{x}_\star},
  \end{equation}
  and the stationarity condition \eqref{eq:GF3}  can be now expressed as 
      \begin{equation}\label{eq:GF6}
      \Psi_{\mathbf{x}_\star}(\mathbf{x}) =0 \,,\;\forall \,\mathbf{x}.
        \end{equation}
 Another interesting aspect is that given a vector  $ \mathbf{v}=(v_1,v_2)$, the solenoidal/stationarity condition  associated to this representation  is, up to appropriate boundary conditions, and up to a constant 
    \begin{equation}\label{eq:GF7}
\int^{x^2} v_1 \mathrm{d}x^2 =-  \int^{x^1} v_2 \mathrm{d}x^1.
        \end{equation}
Setting 
$$
\psi := \dfrac{1}{2}\left( \int^{x^2} v_1 \mathrm{d}x^2 -  \int^{x^1} v_2 \mathrm{d}x^1 \right) =\int^{x^2} v_1 \mathrm{d}x^2 =-  \int^{x^1} v_2 \mathrm{d}x^1
$$
one can easily show that when solenoidal/stationary  condition \eqref{eq:GF7} is verified, the  
$  \mathbf{v}$ can be written as 
\begin{equation}\label{eq:GF8}
  \mathbf{v} = ( \partial_{x^2}\psi, -\partial_{x^1}\psi) = \nabla^{\perp} \psi,
\end{equation}
which is the classical curl-potential condition for solenoidal fields. 

Finally, for non-homogenous problems, the generalization of \eqref{eq:GF0}  and \eqref{eq:GF5} can be written as 
    \begin{equation}\label{eq:GF9}
    \nabla\cdot  \mathbf{F}  - \mathbf{S} \equiv   \partial_{x^1x^2} (\mathcal{F}_1 + \mathcal{F}_2 + \mathcal{S})  
  \end{equation}
  with the definitions \eqref{eq:GF1}, and now setting
      \begin{equation}\label{eq:GF10}
      \mathcal{S} :=- \int^{x^1}\int^{x^2} \mathbf{S}\,\mathrm{d}x^1 \mathrm{d}x^2 .
        \end{equation}
As before, in the neighbourhood of a fixed location $\mathbf{x}_\star$ we can equivalently write 
      \begin{equation}\label{eq:GF11}
           \nabla\cdot  \mathbf{F}  - \mathbf{S}  \equiv  \partial_{x^1x^2}\Psi_{\mathbf{x}_\star}
        \end{equation}
        with now
          \begin{equation}\label{eq:GF12}
          \begin{aligned}
      \Psi_{\mathbf{x}_\star}(\mathbf{x})  := &
  \int\limits_{x^{2,\star}}^{x^2} \left( \mathbf{F}_1(\mathbf{x}) -\mathbf{F}_1(x^{1,\star},x^2) \right) \mathrm{d}x^2  +
  \int\limits_{x^{1,\star}}^{x^1} \left( \mathbf{F}_2(\mathbf{x}) -\mathbf{F}_2(x^1,x^{2,\star}) \right)   \mathrm{d}x^1-  \int\limits_{x^{1,\star}}^{x^1}  \int\limits_{x^{2,\star}}^{x^2}\mathbf{S}(\mathbf{x})\,\mathrm{d}x^1 \mathrm{d}x^2\\
  =   &\int\limits_{x^{1,\star}}^{x^1}  \int\limits_{x^{2,\star}}^{x^2} (\nabla\cdot  \mathbf{F}(\mathbf{x}) - \mathbf{S}(\mathbf{x}))\mathrm{d}x^1 \mathrm{d}x^2 \;.
  \end{aligned}
  \end{equation}
Again, the stationary conditions are simply expressed by \eqref{eq:GF3} and \eqref{eq:GF6}.       
Previous works \cite{barsukow2025structure,barsukow2025genuinely,barsukow2025stationaritypreservingnodalfinite} have demonstrated that
when this approach is combined with a stabilization compatible with the stationary divergence condition (see discussion in the previous section),
the resulting numerical methods show enormous accuracy enhancements for arbitrary stationary states,
as well as perturbations of these equilibria. In the nonlinear case, even at low order, the 
previous work \cite{barsukow2025genuinely} has shown great improvements in the approximation of hydrodynamic instabilities,
and well-behaved solutions in the limit of Mach numbers approaching zero. Other interesting properties
can be shown, as the existence of involutions, and  local   super-convergence dictated only by the quadrature formulas used in \eqref{eq:GF1}.

Due to the simultaneous treatment  of all the terms of the PDE, and to the introduction of the primitive of the source term, bearing similarities
with one-dimensional ideas discussed e.g. in \cite{mantri2024fully,Cheng2019}, this formulation has been referred to  as global flux quadrature in   \cite{barsukow2025structure}. 
The terminology  ``global''  may be  misleading.  It must be noted that representations of the source terms using derivatives of some equilibrium flux 
have also been referred to in the past as spatial \emph{localization}  of source terms (see e.g. \cite{Gosse-toscani03}). 
Indeed, the schemes obtained with this idea are fully local. As we will see, this can be proven 
simply using the residual based representation \eqref{eq:GF10}-\eqref{eq:GF11} on mesh cells \cite{barsukow2025genuinely}. 
The unified treatment of all terms in the PDE aligns closely with the concept of fully residual methods, which naturally connect to residual distribution schemes \cite{Abgrall2022,amr2025}.

\subsection{Contribution and structure of this paper}

In this paper, we provide the first very high order extension of the stationarity preserving (previously known as global flux) quadrature of   \cite{barsukow2025structure}  to  highly nonlinear systems of balance laws.
We consider to this end the Euler equations with gravitational potential. 
We provide a description of the method and we generalize
the consistency estimates provided in  \cite{barsukow2025structure} for general stationary states to the nonlinear case, and we provide a heuristic
justification of its low Mach number consistency. 
We also introduce a treatment of the gravity source allowing to retain exactly iso-thermal hydrostatic equilibria.
We perform a detailed benchmark study on stationary solutions with
a wide range of Mach numbers, and on hydrodynamic instabilities, as well as problems involving weak shocks. 
The numerical results show that  the stabilized SUPG finite element method with the global flux quadrature is free of mesh imprinting,   
and outperforms 
the classical one, for both smooth and non-smooth problems, for all Mach numbers,  and  for both stationary  problems and hydrodynamic instabilities.\\

The structure of the paper is the following. Section~\ref{sec:discretization} introduces the notation for the mesh and finite element spaces, the weak formulation of the equations,
the streamline upwind stabilization method, and the genuinely explicit Defect Correction time stepping strategy used in the paper.
In Section~\ref{sec:lack}, we  discuss the  stationarity preserving properties of the standard methods, recalling the negative results proven in our previous work, 
and discussing possible exceptions and their drawbacks. Section~\ref{sec:stationarityGF} is devoted to the  presentation of the stationarity preserving/global flux quadrature method for systems
of  non-linear balance laws. The discrete equations obtained are analyzed in Section~\ref{sec:properties} with respect to local conservation,   and super-convergence at steady state.
An extension to preserve   exactly hydrostatic solutions when solving the Euler equations with gravitational potential is given in Section~\ref{sec:hydro_equi}.
Finally, a thorough numerical benchmark study showing the enhancements brought by the new formulation is presented in Section~\ref{sec:results}. The paper is concluded in Section~\ref{sec:conclusion} by a short summary,
and, in Appendix~\ref{app:consistency}, with a super-convergence proof.

\section{Discretization framework: stabilized finite elements}\label{sec:discretization}

\subsection{Mesh, nodal unknowns, basis functions}

We consider discrete approximations based on tensor finite elements defined on 
Cartesian tessellations of the domain.  So, let $\Omega$ be  a rectangular  domain $\Omega := \Omega^{x^1} \times \Omega^{x^2} \subset \mathbb R^2$. We define 
   cells $E_{ij}:=   E^{1}_i \times E^{2}_j := [x^1_{i},x^1_{i+1}]\times [x^2_{j}, x^2_{j+1}]$ with $i=0,\dots, N_1-1$, $j=0,\dots,N_2-1$.
We  set   $|E^{1}_i|=h_{1,i}$, $|E_j^{2}|=h_{2,j}$  and $h:=\min_{i,j}( h_{1,i},h_{2,j})$. 
In this work, we assume   $h_{1,i}=h_1\;\forall\,i$  and $h_{2,j}= h_2\;\forall\,j$.  The computational domain is  $\Omega_{h}:=\cup_{i=0}^{N_1-1}\cup_{j=0}^{{N_2}-1} [x^1_{i},x^1_{i+1}]\times [x^2_{j},x^2_{j+1}]$.
To be fully precise,  since $\Omega$ is rectangular, $\Omega = \Omega_h$, so we will simply use $\Omega$ below. We will in particular write that  $\Omega =[x^1_0, x^1_0+L_1]\times  [x^2_0, x^2_0+L_2]$.\\

We use continuous nodal FEM spaces of polynomial degree $K$. In particular, 
in each one-dimensional cell $ E^{d}_\ell$ we introduce $K+1$ interpolation points
$\{x^{d}_{\ell,p}\}_{p=0,\dots,K} \in E^{d}_\ell$, with $x^d_{\ell,0} = x^d_{\ell}$ and $x^d_{\ell,K} = x^d_{\ell+1}$,
and  of course $x^{d}_{\ell,0} = x^{d}_{\ell-1,K}$ for $\ell=1,\dots,N_d-1$. In total, we will consider $N_h \equiv (N_1K+1)(N_2K+1)$ points of type $(x^1_{i,p},x^2_{j,k})$.
In this work, we only use  Gauss--Lobatto points in each direction. 
These points are used to define the Lagrangian basis functions $\{\varphi^d_{\ell,p}\}_{{p=0,\dots K}} $ spanning in each direction the local 
space of univariate polynomials $ \mathbb P^K(E^d_\ell)$ of degree at most K.  This allows us to define the tensor based approximation space 
\begin{equation}\label{eq:Vh}
		V^K_h(\Omega):=
		\left\lbrace q \in \mathcal{C}^0(\Omega): q|_E\in \mathbb Q^K\equiv\bigotimes\limits_{d=1,2}\mathbb     P^K(E^{d}),\;\forall E= \bigotimes\limits_{d=1,2}E^{d} \in \Omega\right\rbrace.
\end{equation}
Any function  $q_h \in V^K_h$ can be written as  
\begin{equation}\label{eq:2DFEM_functions}
	q_h(\mathbf{x}) =\sum_{i=0;j=0}^{N_1-1;N_2-1}q_h(\mathbf{x})\big|_{E_{ij}}\;,\quad
	q_h(\mathbf{x})\big|_{E_{ij}} :=  \sum_{p=0;k=0}^{K;K} \mathrm q^{ij}_{pk}\, \varphi^1_{i,p}(x^1)\varphi^2_{j,k}(x^2),
\end{equation}
where $ \mathrm  q^{ij}_{pk}$ are locally numbered nodal values  of $q$ constrained to continuous elements by $\mathrm q_{0k}^{ij} = \mathrm q_{Kk}^{i-1,j}$ for all $k=0,\dots,K$, $i=1,\dots,N_1-1$ and $j=0,\dots,N_2-1$ and $\mathrm q_{p0}^{ij} = \mathrm q_{pK}^{i,j-1}$ for all $p=0,\dots,K$, $j=1,\dots,N_2-1$ and $i=0,\dots,N_1-1$. 
We will use the roman font $\mathrm q$ for the degrees of freedom that refer to a function $q_h \in V^K_h$ in a discrete Finite Element space. 
In this work,  these values will correspond to the collocated values of $q$ at standard Gauss--Lobatto interpolation points. 

Lastly, note that a given  interpolation point $\mathbf{x}_{\alpha}=(x^1_{\alpha^1},x^2_{\alpha^2})\in E_{ij}$   will correspond to a tuple $(i,p;j,k)$ such that locally 
 $\mathbf{x}_{\alpha}:=( x^{1}_{i,p},\, x^{2}_{j,k} )$. The correspondence $\alpha\mapsto (i, p; j,k) $ is referred to as the local-to-global
reordering of the unknowns as well as their univariate versions $\alpha^1 \mapsto (i,p)$ and $\alpha^2 \mapsto (j,k)$. We will sometimes use the notation $(i(\alpha),p(\alpha); j(\alpha), k(\alpha))$ to make use of this  map. 
Moreover, we will denote with 
\begin{equation}
\varphi_\alpha(\mathbf{x}) :=\varphi^1_{\alpha^1}(x^1)\varphi^2_{\alpha^2}(x^2):=  \varphi^1_{i(\alpha),p(\alpha)}(x^1)\varphi^2_{j(\alpha),k(\alpha)}(x^2)
\end{equation}
for $\alpha =1,\dots, N_h$ the basis functions spanning $V^K_h$ in the global numbering, $\alpha^m = 1,\dots, (N_mK+1)$, for $m=1,2$, so that $\varphi_\alpha(\mathbf{x}_\beta) = \delta_{\alpha\beta}$ for all $\alpha,\beta =1,\dots, N_h$.
 

\subsection{Weak formulation: discrete divergence}

The  semi-discrete unstabilized continuous Galerkin weak form of the  generic system of balance laws  \eqref{Euler_Eq3} reads as  for every $\varphi \in V^K_h$
\begin{equation}\label{eq:weakG0}
\int\limits_{\Omega}\varphi \dfrac{d\mathbf{W}_h}{dt}\,d\mathbf{x} + 
\int\limits_{\Omega}\varphi \nabla\cdot  \mathbf{F}_h(\mathbf{W}_h)\,d\mathbf{x}=
\int\limits_{\Omega}\varphi \, \mathbf{S}_h(\mathbf{W}_h;\mathbf{x})\,d\mathbf{x}.
\end{equation}
For    $\mathbf{W}_h,\, \mathbf{F}_h,\,  \mathbf{S}_h \in (V^K_h)^s$   the integral above can be easily evaluated.
In particular, due to the tensorized polynomial approximation, they can all be 
expressed in terms  of classical one-dimensional mass matrix and derivative operators.
For example, the first integral involves the mass matrix $\mathsf{M}$ which can be classically expressed using local assembly and  the local to global renumbering, using as $\varphi = \varphi_\alpha$ for all $\alpha = 1,\dots, N_h$
 $$
 \int\limits_{\Omega}\varphi_\alpha \dfrac{d\mathbf{W}_h}{dt}\,d\mathbf{x} \equiv
  \sum\limits_{E_{ij}\in\Omega}  \sum\limits_{\beta \in E_{ij}}    [\mathsf{M}^{E_{ij}}]_{\alpha\beta}\dfrac{d\mathbf{W}_{\beta}}{dt}\;, 
  $$
  with, for any $\mathbf{x}_{\alpha},\mathbf{x}_{\beta}\in E_{ij}$,
  $$
   [\mathsf{M}^{E_{ij}}]_{\alpha\beta}:=
 \int_{x^1_{i}}^{x^1_{i+1}} \!\!\varphi^1_{i(\alpha),p(\alpha)}\!(x^1)\, \varphi^1_{i(\beta),p(\beta)}\!(x^1)\,\mathrm{d}x^1  \;
  \int_{x^2_{j}}^{x^2_{j+1}}\!\! \varphi^2_{j(\alpha),k(\alpha)}\!(x^2)\, \varphi^2_{j(\beta),k(\beta)}\!(x^2)\,\mathrm{d}x^2  \,, 
 $$ 
where the two integrals are the one dimensional mass matrices in the corresponding directions.
The above property can be compactly expressed using the Kronecker product between matrices $\otimes$ (see \cite{barsukow2025structure}  for more details):
$$
\mathsf{M}^{E_{ij}} = M_1^{E_i^{1}}\otimes  M_2^{E_j^{2}}.
$$
where 
$$
[M_m^{E_i^{m}}]_{\alpha^m,\beta^m} :=   \int_{x^m_{i}}^{x^m_{i+1}} \!\!\varphi^m_{i(\alpha^m),p(\alpha^m)}\!(x^m)\, \varphi^m_{i(\beta^m),p(\beta^m)}\!(x^m)\,\mathrm{d}x^m\;, \text{ for }x^m_{\alpha^m}, x^m_{\beta^m} \in E_i^{m}, \; \text{ for }m=1,2.
$$
At the global level, we can define
$$
\mathsf{M}=   M_1\otimes  M_2 
$$
having set the univariate global mass matrices as
$$
[M_m]_{\alpha^m,\beta^m} := \sum\limits_{E_i^{m} \in\Omega^{x^m}} \sum_{\alpha^m,\beta^m \in E^m_i} \left[ M_m^{E_i^{m}} \right]_{\alpha^m,\beta^m}\;,\quad m=1,2\,.
$$
The entries of $\mathsf{M}$ are
$$
[\mathsf{M}]_{\alpha, \beta}=   \sum\limits_{\substack{E_{ij}\in\Omega\\\alpha,\beta\in E_{ij}}}  [\mathsf{M}^{E_{ij}}]_{\alpha \beta} = [M_1\otimes  M_2]_{\alpha  \beta}. 
$$
Note that similar notations are used for all other matrices appearing in the paper.

Thus, we have 
$$
\int\limits_{\Omega}\varphi_\alpha \dfrac{d\mathbf{W}_h}{dt}\,d\mathbf{x} \equiv
\left[ M_1\otimes  M_2 \dfrac{d }{dt} \mathrm{W} \right]_{\alpha}
$$
with $ \mathrm{W} \in \mathbb R^{N_h \times s}$ is the array of nodal values  of the unknowns.

We can proceed similarly  for all integrals. In particular, we have
\begin{equation}\label{eq:localDx1}
 \int\limits_{\Omega}\varphi_\alpha  \partial_{x^1}   \mathbf{F}_1 \,d\mathbf{x}  \equiv \left[
  \sum\limits_{E_{ij}\in\Omega}   (D_1^{E_i^{1}}\otimes M_2^{E_j^{2}} ) \mathrm{F}_1^{E_{ij}} \right]_{\alpha}
\end{equation}
with
\begin{equation}\label{eq:localDx2}
\left[ D_1^{E^1_i} \right]_{\alpha^1,\beta^1} :=  \int_{x^1_{i}}^{x^1_{i+1}} \!\!\varphi^1_{\alpha^1}\!(x^1)\, (\varphi^1_{\beta^1}\!(x^1))'\,\mathrm{d}x^1 \,,\text{ for }x^1_{\alpha^1}, x^1_{\beta^1} \in E_i^{1}\,,
\end{equation}
the  integrated $x^1$ derivative operator,  and $\mathrm{F}_1^{E_{ij}}$ the array of nodal values in $E_{ij}$ of the flux in the $x^1$ direction. 
Again, after global assembly we have  
$$
\int\limits_{\Omega}\varphi_\alpha  \partial_{x^1}   \mathbf{F}_1 \,d\mathbf{x} \ \equiv  \left[ (D_1 \otimes M_2)\mathrm{F}_1 \right]_{\alpha}
$$
with now $\mathrm{F}_1$ is the globally numbered array of nodal values of the flux in the $x^1$ direction.
The definition for the $x^2$ derivative operator is identical. 

Finally, the Galerkin weak form can be written  as the system of  non-linear ODEs
\begin{equation}\label{eq:weakG1}
(M_1\otimes  M_2)  \dfrac{d }{dt} \mathrm{W} +(D_1\otimes  M_2)   \mathrm{F}_1  +(M_1\otimes  D_2)  \mathrm{F}_2=
(M_1\otimes  M_2)  \mathrm{S}\;.
\end{equation}
With this notation we can define local and global integrated divergence operators as
 \begin{equation}\label{eq:weakG2}
 \textsf{div}_h \mathbf{F}  :=  (D_1\otimes  M_2)   \mathrm{F}_1  +(M_1\otimes  D_2)  \mathrm{F}_2\;,\;\;\;
 \textsf{div}_h^{E_{ij}}\mathbf{F} := (D_1^{E_i^{1}}\otimes M_2^{E_j^{2}}) \mathrm{F}_1^{E_{ij}} + (M_1^{E_i^{1}}\otimes D_2^{E_j^{2}}) \mathrm{F}_2^{E_{ij}}\,.
 \end{equation}

\subsection{Stabilization via streamline upwinding}\label{sec:stabilization}

As discussed in the introduction, the numerical treatment  of hyperbolic problems requires stabilization,
introducing   some form of artificial dissipation. 
When equilibrium solutions are sought, the use of classical Laplacian based operators
is a bad choice as the operator $\Delta\mathbf{u}$  is incompatible with the stationarity condition  $\nabla\cdot\mathbf{F}=\mathbf{S}$.
As shown in \cite{barsukow2025structure}, the classical streamline upwind (SU) term  \cite{brooks82,hughes86}   introduces
a stabilization with a structure that is naturally compatible with the stationarity condition. To recall the main idea, in absence of source terms
this method can be obtained from  the  weak form of the modified problem
\begin{equation}\label{eq:mod_eq}
\partial_t\mathbf{W} + \nabla\cdot\mathbf{F} =  \nabla\cdot\left( \mathbf{J} \,\tau\,  \nabla\cdot\mathbf{F}  \right),
\end{equation}
where we recall that  $\mathbf{J}$ is the order three tensor whose components are the Jacobians  of the fluxes (cf. equation \eqref{Euler_Eq6b}), 
and $\tau$ is in general a matrix depending on both $\mathbf{W}$, and the mesh. At the continuous level, the introduced stabilization
at the right hand side can be  readily shown to introduce entropy dissipation by expressing it in terms of the gradient of the entropy variables.
The interested reader can refer to e.g. \cite{hfm86,Barth98} for this aspect. 
In this work, we are interested in the behaviour of the method
when the numerical dissipation is evaluated by direct interpolation of the fluxes in the form \eqref{eq:mod_eq}, without a change into entropy variables. 
Naturally, in absence of sources, \eqref{eq:mod_eq} is automatically compatible with stationary states, i.e., if $\nabla\cdot\mathbf{F} = \mathbf{0}$, then the stabilization vanishes.
This is one of the reasons for considering the SUPG method, which already has some inherent multi-dimensional flavour.
A detailed analysis of the structure of the stabilization in the case of acoustics can be found in \cite{barsukow2025structure}. 

In presence of sources and for time dependent problems, one needs to account for both the time variation and the forcing terms to
obtain a fully consistent method  verifying an   orthogonality condition. This   requires    using the modified equation model
\begin{equation}\label{eq:mod_eq1}
\partial_t\mathbf{W} + \nabla\cdot\mathbf{F} =  \mathbf{S} + \nabla\cdot\left( \mathbf{J} \,\tau\,  (\partial_t\mathbf{W}+ \nabla\cdot\mathbf{F}  -\mathbf{S})\right).
\end{equation}
Stability is more complex to show in this case. Results  for linear problems are discussed in \cite{BOCHEV20042301,burman1,barsukow2025stationaritypreservingnodalfinite}.
Also in this case, \eqref{eq:mod_eq1} is in principle automatically compatible with all stationary states, i.e., when $\partial_t\mathbf{W}+ \nabla\cdot\mathbf{F}  -\mathbf{S} = \mathbf{0}$ then the stabilization vanishes.
Now, the variational form of \eqref{eq:mod_eq1} reads: for every $\varphi \in V^K_h$
\begin{equation}\label{eq:weakSU0}
\begin{split}
\int\limits_{\Omega}(\varphi  +   \nabla \varphi\cdot \mathbf{J} \,\tau\,) \dfrac{d\mathbf{W}_h}{dt}\,d\mathbf{x}  + &
\int\limits_{\Omega}\varphi \nabla\cdot  \mathbf{F}_h(\mathbf{W}_h)\,d\mathbf{x} \\ +& 
\int\limits_{\Omega} \nabla \varphi\cdot \mathbf{J} \,\tau\, \nabla\cdot  \mathbf{F}_h(\mathbf{W}_h)\,d\mathbf{x} =
\int\limits_{\Omega}(\varphi  +  \nabla \varphi\cdot \mathbf{J} \,\tau \,)    \mathbf{S}_h(\mathbf{W}_h;\mathbf{x})\,d\mathbf{x},
\end{split}
\end{equation}
where, as noted above, the flux Jacobian and $\tau$ are   also evaluated using the same     polynomial expansion in $V_h^K$.  In practice, this means that
both  are  evaluated by quadrature points. In the present work, we use the same points for both interpolation and quadrature, as in the so-called spectral element formulation
\cite{sem-book}. 
In this case, 
if $ \mathbf{J} \,\tau$ are also expanded on the $V_h^K$ basis 
one can  express    the elemental contributions of the 
stabilization integrals as 
\begin{equation}\label{eq:su_local_1}
\int_{E_{ij}}  \nabla \varphi\cdot \mathbf{J} \,\tau\, \dfrac{d\mathbf{W}_h}{dt}\,d\mathbf{x} \equiv 
\left\{ \left((D_1^{E_i^{1}})^{\mathrm{T}}  \otimes M_2^{E_j^{2}} \right) \textrm{diag}(\mathbf{J}_1 \,\tau) 
+ \left(M_1^{E_i^{1}}\otimes  (D_2^{E_j^{2}})^{\mathrm{T}}\right) \textrm{diag}(\mathbf{J}_2 \,\tau)\right\} \dfrac{d }{dt} \mathrm{W}^{E_{ij}}
\end{equation}
and  an analog expression for the source, while
\begin{equation}\label{eq:su_local_2a}
\begin{aligned}
\int\limits_{E_{ij}} \nabla \varphi\cdot \mathbf{J} \,\tau\, \nabla\cdot  \mathbf{F}_h(\mathbf{W}_h)\,d\mathbf{x}\equiv  & 
\left\{ \left( (D_1^{E_i^{1}})^{\mathrm{T}}  (M_1^{E_i^{1}})^{-1} \otimes \id_2^{E_j^{2}} \right)   \textrm{diag}(\mathbf{J}_1 \,\tau)  \left(D_1^{E_i^{1}}  \otimes M_2^{E_j^{2}} \right)  \right\} \mathrm{F}_1^{E_{ij}}\\[-7pt]
 + & \left\{ \left((D_1^{E_i^{1}})^{\mathrm{T}}(M_1^{E_i^{1}})^{-1}\otimes \id_2^{E_j^{2}} \right)     \textrm{diag}(\mathbf{J}_1 \,\tau)   \left(M_1^{E_i^{1}} \otimes D_2^{E_j^{2}}\right)  \right\} \mathrm{F}_2^{E_{ij}} \\
 + & \left\{  \left(\id_1^{E_i^{1}}    \otimes (D_2^{E_j^{2}})^{\mathrm{T}}(M_2^{E_j^{2}})^{-1}  \right)   \textrm{diag}(\mathbf{J}_2 \,\tau) \left( D_1^{E_i^{1}} \otimes M_2^{E_j^{2}}\right)  \right\} \mathrm{F}_1^{E_{ij}}\\[2pt]
  + &  \left\{ \left(\id_1^{E_i^{1}}   \otimes (D_2^{E_j^{2}})^{\mathrm{T}}(M_2^{E_j^{2}})^{-1} \right)   \textrm{diag}(\mathbf{J}_2 \,\tau) \left(M_1^{E_i^{1}}\otimes D_2^{E_j^{2}}\right)    \right\} \mathrm{F}_2^{E_{ij}},
 \end{aligned}
\end{equation}
with $\id_1^{E^1_i}$ and $\id_2^{E^2_j}$ being the identity matrices on the $x^1$ element $E_i^{1}$ and $x^2$ element $E_j^{2}$, respectively, and where $^{\mathrm{T}}$ denotes the transpose operator, and $\textrm{diag}(\mathbf{J}_m \,\tau)$ contain diagonal matrix blocks with  the values of the argument sampled at
quadrature points. 
The usual expressions  for the second order weak derivatives are obtained locally on each element 
if   $\textrm{diag}(\mathbf{J}_m \,\tau)$ are taken constant per element, e.g. in $x^1$ $(D_1^{E_i})^{\mathrm{T}} (M_1^{E_i})^{-1} D_1^{E_i}$, and the second order weak derivative matrices are obtained globally
if this term is  constant in the domain. Note that refactoring terms and using \eqref{eq:weakG2}, the last expression can be more compactly written as 
\begin{equation}\label{eq:su_local_2}
	\begin{split}
		\int\limits_{\Omega} \nabla \varphi\cdot \mathbf{J} \,\tau\, \nabla\cdot  \mathbf{F}_h(\mathbf{W}_h)\,d\mathbf{x}\equiv  &
		\sum_{E_{ij}\in\Omega}\left\{ \left( (D_1^{E_i^{1}})^{\mathrm{T}}  (M_1^{E_i^{1}})^{-1} \otimes \id_2^{E_j^{2}} \right)   \textrm{diag}(\mathbf{J}_1 \,\tau)  \right\}\textsf{div}_h^{E_{ij}}\mathbf{F}
		\\+ &\sum_{E_{ij}\in\Omega}\left\{ \left(\id_1^{E_i^{1}}   \otimes (D_2^{E_j^{2}})^{\mathrm{T}}(M_2^{E_j^{2}})^{-1} \right)   \textrm{diag}(\mathbf{J}_2 \,\tau)  \right\} \textsf{div}_h^{E_{ij}}\mathbf{F}.
	\end{split}
\end{equation}
 We write the semi-discrete method \eqref{eq:weakSU0} as 
\begin{equation}\label{eq:weakG1-SU}
M \dfrac{d }{dt} \mathrm{W} +
 \textsf{div}_h \mathbf{F}    + \mathsf{st}_h(\mathrm{F}_1,\mathrm{F}_2)=
M  \mathrm{S},
\end{equation}
where
$$
M = M_1\otimes  M_2   + M^{\textsf{SU}}
$$
with $M^{\textsf{SU}}$ implicitly defined by \eqref{eq:weakSU0} and \eqref{eq:su_local_1}, and with $\mathsf{st}_h(\mathrm{F}_1,\mathrm{F}_2)$
given by \eqref{eq:su_local_2a} (or equivalently \eqref{eq:su_local_2}). Note that $M^{\textsf{SU}}=M^{\textsf{SU}}( \mathrm{W})$, and so $M=M(\mathrm{W})$, 
depend on $\mathrm{W}$ through the flux Jacobian $\mathbf{J}$ used in the stabilization.

%
%
%
%
%
%

%
%
%
%

\subsection{Genuinely explicit time stepping via deferred/defect correction}\label{sec:DeC}

In the spectral element setting, the use of the same interpolation points for the numerical quadrature makes the mass matrix $M_1\otimes  M_2$  fully diagonal.
The contribution $M^{\textsf{SU}}$ from the stabilization coming from \eqref{eq:su_local_1} leads however to block entries that are not diagonal.
The resulting matrix needs to be inverted to compute new solution values. 
To obtain a genuinely  explicit  high-order time discretization we use  a  Defect  Correction (DeC) (or Deferred Correction) method \cite{dutt2000spectral,minion2003semi},  
based on the construction   introduced in \cite{RA10} for residual distribution methods, and  generalized in    \cite{abgrall2017high,micalizzi2022new}.
We proceed as follows.


We consider time slabs $[t_n,t_{n+1}]$, and set   $\Delta t= t_{n+1}- t_n $.   Within the each slab,
we  introduce  $M+1$ stages in time $[t_n,t_{n+1}]$, defined by the temporal collocation points $t_n=t_n^0< \cdots < t^m_n \cdots  < t^M_n=t_{n+1}$.
We denote by $\mathrm{W}^{n,m}$ the stage value $m$, with $\mathrm{W}^{n,0}= \mathrm{W}^n\approx \mathrm{W}(t_n)$ and  $\mathrm{W}^{n,M}= \mathrm{W}^{n+1}\approx \mathrm{W}(t_{n+1})$. With an abuse of notation, we will drop the index $n$ as we will consider only the time step $[t_n, t_{n+1}]$.  
Let $\{\gamma^m(t)\}_{m=0}^{M}$ denote the   degree-$M$ Lagrange  polynomial bases
associated to $\{t^m\}_{m=0}^{M}$, and set   
\begin{equation}\label{eq:stageLagrange}
\mathrm{W}(t) = \sum_{m=0}^{M} \gamma_m(t) \mathrm{W}^m\;,\quad
\theta^m_r:=\frac{1}{\Delta t}\int_{t^0}^{t^m} \gamma_r(t) \dd t.
\end{equation}
For the stabilized method \eqref{eq:weakG1-SU},  we now introduce the array of high order operators $\mathcal{L}^{H}$ of components for $m=1,\dots,M$
\begin{equation}
\begin{split}
	\mathcal{L}^{H,m}( \mathsf{W}  ) :=   &\mathsf{M}   (\mathrm{W}^m -  \mathrm{W}^0 )
	+ \Delta t \sum_{r=0}^M \theta_r^m \mathsf{M}^{\textsf{SU}}(\mathrm{W}^r) \dfrac{d\mathrm{W}(t^r)}{dt} \\
	+& \Delta t \sum_{r=0}^M \theta_r^m \{ \textsf{div}_h \mathbf{F}(\mathrm{W}^r)    + \mathsf{st}_h(\mathrm{W}^r)  -  M(\mathrm{W}^r )
	 \mathrm{S}(\mathrm{W}^r)\}  ,
	 \end{split}
\end{equation}
where we recall that $\mathsf{M} =  M_1\otimes M_2$ is diagonal, while $\mathsf{M}^\text{SU}(\mathsf{W})$ and $\mathsf{M}(\mathsf{W})$ are not. We have denoted  with $\mathsf{W} = (\mathrm{W}^{0}, \dots, \mathrm{W}^{M}) $ the array  of  all the stage values,  
and    $d \mathrm{W}(t)/dt$  is evaluated using the Lagrange interpolation \eqref{eq:stageLagrange}. 
The  time scheme obtained  solving the non-linear system  $\mathcal{L}^{H,m}=0,\;\forall m$ 
in each temporal  slab, is the    implicit full table Runge-Kutta continuous collocation method associated  to the specific choice of stage values. The coefficients $\theta^m_r:=\frac{1}{\Delta t}\int_{t^0}^{t^m} \gamma_r(t) \dd t$ are the usual coefficients of the Butcher tableau of the method.
When $\{t^m\}_{m=0}^{M}$ are the Gauss-Lobatto points in  $[t_n,t_{n+1}]$ one obtains the so called LobattoIIIA method \cite{Hairer1993}. In the DeC approach,
the  above operator is only used   as a high order limit of explicit iterations. These iterations are driven by the  low order (first order)  operator  
\begin{equation}
\begin{split}
	\mathcal{L}^{L,m}(\mathsf{W}  ) :=   \mathsf{M}   (\mathrm{W}^m -  \mathrm{W}^0 )
	+& \Delta t \beta^m\ \{ \textsf{div}_h \mathbf{F}(\mathrm{W}^0)    + \mathsf{st}_h(\mathrm{W}^0)  -  M(\mathrm{W}^0 )
	 \mathrm{S}(\mathrm{W}^0)\},
	 \end{split}
\end{equation}
with $\beta^m := \frac{t^m_n-t^0_n}{\Delta t}$. Note that the low order operator is fully explicit, as it only depends on the stage value $\mathrm{W}^0$ at the beginning of the time step.

Now, we define an iterative process to approximate the solution of $\mathcal{L}^H=0$. If $\mathsf{W}^{(k)} $ is the array  of    the stage values at the explicit iteration $k$, the DeC method  computes new values of the solution as 
\begin{equation}
	\label{eq:DeC1}
	\begin{cases}
   	   \mathsf{W}^{(0)} =( \mathrm{W}^{n}, \dots \mathrm{W}^{n},   \dots \mathrm{W}^{n} ),\\
		\mathcal{L}^{L}(\mathsf{W}^{(k)} )=\mathcal{L}^{L}(\mathsf{W}^{(k-1)})-\mathcal{L}^{H}(\mathsf{W}^{(k-1)}), \qquad k=1,\dots, \kappa,\\
		\mathrm{W}^{n+1} := \mathrm{W}^{M,(\kappa)} .
	\end{cases}
\end{equation}
Replacing all the expressions at $k$ iteration and $m$ stage, we obtain the  explicit-Euler-like stage iterations:
\begin{equation}
	\label{eq:DeC2}
	\begin{split}
	     \mathrm{W}^{m,(k)}  =            \mathrm{W}^n
	- &\Delta t \sum_{r=0}^M \theta_r^m \mathsf{M}^{-1}  M^{\textsf{SU}}(\mathrm{W}^{r,(k-1)}) \dfrac{d\mathrm{W^{(k-1)}}(t^r)}{dt} \\
	-& \Delta t \sum_{r=0}^M \theta_r^m \mathsf{M}^{-1}  \{ \textsf{div}_h(\mathrm{W}^{r,(k-1)})    + \mathsf{st}_h(\mathrm{W}^{r,(k-1)})  - M(\mathrm{W}^{r,(k-1)} )\mathrm{S}(\mathrm{W}^{r,(k-1)})\}
	 \end{split}
\end{equation}
 with $\mathsf{M}$ diagonal. 
This method has essentially the structure of a multi-stage Runge-Kutta method, and allows to avoid the expensive inversion of the   mass matrix entries introduced by the SU term. 
A detailed Fourier stability  analysis of DeC time stepping with    stabilized  finite element discretizations 
is provided in \cite{michel2021spectral} on one space dimension, and  in   \cite{michel2023spectral} on structured triangulations in two dimensions.   


%

\section{Lack of stationarity preserving in the classical setting}\label{sec:lack}

 In absence of sources, stationary states of the Galerkin operator are defined by (up to boundary conditions)
 \begin{equation}\label{eq:weakG3}
 \textsf{div}_h \mathbf{F}=0 .
 \end{equation}
 
 For the stabilized method to be stationarity preserving according to Condition~\ref{cond:SP},  we need to  first verify the realizability of first point. 
 We thus need to look at the residual approximations provided by the Galerkin approximation. Then, we look at whether these residuals set to zero could force the stabilization to vanish as well.
 There are only two  possibilities for this condition to be verified: 
  \begin{enumerate}[label={P.\arabic*}]
 	\item \label{item:SUPG_div_global_implies_stab_zero} either \eqref{eq:weakG3} implies that also the stabilization terms vanish;
  \item \label{item:cell_div_zero} or  $\mathsf{div}_h^E\mathbf{F} =0\;\forall \,E$ (cf. \eqref{eq:weakG2}), in which case  \eqref{eq:weakG3} is verified and  the stabilization terms vanish. 
  \end{enumerate}
Condition~\ref{item:SUPG_div_global_implies_stab_zero} is met only when the solution is constant, which does not include all possible equilibria. This  can be proven rigorously for linear systems,
and  is  summarized   by the following  results.

\begin{lemma}[Local and global second derivative operators] \label{lem:nosp} For any $m=1,\dots,d$:
$$
\sum_{\ell}     (D_m^{E_\ell^{m}})^{\mathrm{T}}  (M_m^{E_\ell^{m}})^{-1}   D_m^{E_\ell^{m}}     \ne
(D_m)^T (M_m)^{-1} D_m\;. 
$$
\begin{proof}
See \cite[Proposition 3]{barsukow2025structure}.
\end{proof}
\end{lemma}
Inspecting the entries  of \eqref{eq:su_local_2a}, Lemma~\ref{lem:nosp} shows   that, 
when   $\mathbf{J}\,\tau$ is constant   $\forall\, E_{i,j}\in\Omega$,    in the evaluation of second derivatives
the embedding of first derivative operators  only applies to the elemental entries, but is not true at the global  level.
The above result can be also generalized by including the $\mathbf{J}\,\tau$ term appearing in \eqref{eq:su_local_2a}.
We then have the following equivalence.

\begin{proposition}[Second  and first derivative kernel equivalence conditions]  \label{prop:equi} 
The following properties are equivalent in the $x^1$ direction
\begin{enumerate}
\item $ \sum_{E_{ij}}\Big\{ (D_1^{E_i^{1}})^{\mathrm{T}}  (M_1^{E_i^{1}})^{-1}      \otimes \id_2^{E_j^{2}} \Big\}\textsf{div}_h^{E_{ij}}\mathbf{v} = 0$   for any $\mathbf{v}\in\mathbb{R}^d $ such that $\textsf{div}_h\mathbf{v}=0$ ;
\item  $ \sum_{i}     (D_1^{E_i^{1}})^{\mathrm{T}}  (M_1^{E_i^{1}})^{-1}   D_1^{E_i^{1}}     = (D_1)^T (M_1)^{-1} D_1$ .
\end{enumerate}
and similarly for all other directions. 
\begin{proof}
See \cite{barsukow2025structure}, Proposition 2.
\end{proof}
\end{proposition} 
We can again inspect the entries of the stabilization \eqref{eq:su_local_2a}.  Already for linear systems, or when 
$\mathbf{J}\,\tau$ can be taken constant,  the equivalence states that for a member of the kernel of the divergence
to also nullify the stabilization, then the second derivatives must be evaluated using embedded first derivatives.

\emph{Combining the equivalence  of Proposition \ref{prop:equi} with Lemma \ref{lem:nosp}, we deduce that the SUPG method cannot be stationarity preserving
in the sense of Condition~\ref{cond:SP}.}\\

The only remaining possibility is \ref{item:cell_div_zero}, namely that $\textsf{div}^E_h \mathbf{F} =0$ should vanish for all elements.
This is a quite restrictive requirement.  Let us  for simplicity  consider $N_1=N_2=N$. Due to  the continuity of the approximation space, 
given  a vector $\mathbf{v}_h\in (V_h^K)^2$, the requirement $\textsf{div}^E_h \mathbf{v}=0\;\forall \,E$ is a system of size $(K+1)^2 N^2$, i.e.,
the total number of elements times the number of  degrees of freedom per element.
This system must be satisfied by  the total number of degrees of freedom describing  $\mathbf{v}_h$ given by  $2 ( N K +1 )^2 $. One can show that  $2 ( N K +1 )^2-1 $ of these constraints are linearly independent\footnote{We have  experimentally found   that up to $K=10$  there is only one relation which is not independent}.
Then, the only possibility  to have a stationary state is that the total number of unknowns should be at equal or larger than the number of constraints, namely 
$$
2 ( N K +1 )^2 \ge (K+1)^2 N^2-1 \approx (K+1)^2 N^2 \Rightarrow \sqrt{2}( NK+1) \gtrapprox N(K+1)   
$$
which for  sufficiently large $N$ is only possible if $K \gtrapprox 1+\sqrt{2} \approx 2.4$.
This means that for  $K = \{ 1,2\} $ the above constraints cannot be satisfied. 
Starting with $K=3$ 
in principle some solution could  be found.   The question is what would be the properties of such  discrete stationary state. 
	A possible way to  answer is to look into 
	families of vectorial  finite elements spaces that are compatible with a solenoidal condition, but not anymore in $V_h^K$. The best known family of elements providing the closest divergence conformal full vectorial approximation is the usual Raviart-Thomas $\mathbb{RT}_{K-1}$ element, defined by $\mathbb{RT}_{K-1}=\mathbb{P}_{K}\otimes d\mathbb{P}_{K-1}\oplus d\mathbb{P}_{K-1}\otimes\mathbb{P}_{K}$, where $d\mathbb P$ denotes polynomials which are discontinuous at the interfaces. Strictly speaking,   these spaces are not into the $(V_h^K)^2$ space relevant for this work, but we can consider them for the sake of argument
	as their restriction to an element does belong to the restriction of  $(V_h^K)^2$, which is $(\mathbb{Q}^K)^2$.
Note, however, that on these elements optimal interpolation error estimates are of order $h^{K}$ in $L^2$ norm (see e.g. \cite[Chapter 1]{brezzifortin}), 
which is less than the well known $h^{K+1}$ interpolation error estimates valid for non affine quadrilateral Lagrange elements (see e.g. \cite[Chapter 1]{eg04}).
Perhaps, for $K$ high enough, continuous conformal solutions could be found, but 
the approximation properties of classical conformal elements suggest that, should we be able to construct such approximation, 
we are likely pay the conformity with  null divergence with a reduction in interpolation accuracy  \cite{brezzifortin,Arbogast16}.
In this work, we consider another approach.

%

\section{Stationarity preserving  quadrature via global flux potentials }\label{sec:stationarityGF}

To construct stationarity preserving methods, we make use of the pre-processing of the flux vector discussed in  section~\ref{sec:multid_GF} 
 to modify the quadrature of the divergence operator both in the Galerkin as well as  in the stabilization term.
Details on the construction of the discrete equations are provided in this section.

\subsection{Approximation for the flux potentials}

We need to provide the explicit definition of the discrete counterpart of the potentials $(\mathcal{F}_1,\mathcal{F}_2)$ appearing in the 
stationarity preserving global flux formulation  \eqref{eq:GF0}. Following \cite{barsukow2025structure,barsukow2025stationaritypreservingnodalfinite}, 
we will assume that $(\mathcal{F}_1,\mathcal{F}_2)$ are approximated in $V_h^K(\Omega)$ as the physical fluxes, despite
being formally defined as derivatives of the latter. 
In particular, we set  (see \eqref{eq:Vh} and \eqref{eq:2DFEM_functions})
\begin{equation}\label{eq:bigF1}
\mathcal{F}_{1,h}(\mathbf{x}):= \sum_{i=0;j=0}^{N_1-1;N_2-1}\mathcal{F}_{1,h}^{E_{ij}}(\mathbf{x}) \;,\quad
\mathcal{F}_{2,h}(\mathbf{x}):=\sum_{i=0;j=0}^{N_1-1;N_2-1}\mathcal{F}_{2,h}^{E_{ij}}(\mathbf{x})
\end{equation}
with
\begin{equation}\label{eq:bigF2}
\mathcal{F}_{1,h}^{E_{ij}}(\mathbf{x}) =  \sum_{p=0;k=0}^{K;K} (\mathcal{F}_1)^{ij}_{pk}\, \varphi^1_{i,p}(x^1)\varphi^2_{j,k}(x^2)\;,\quad
\mathcal{F}_{2,h}^{E_{ij}}(\mathbf{x}) = \sum_{p=0;k=0}^{K;K}  (\mathcal{F}_2)^{ij}_{pk}\, \varphi^1_{i,p}(x^1)\varphi^2_{j,k}(x^2).
\end{equation}
The first key aspect of the new method is the computation of the nodal values $(\mathcal{F}_1)^{ij}_{pk}$ and $(\mathcal{F}_2)^{ij}_{pk}$.
This is achieved by applying their definition element-wise  and in a row-by-row / line-by-line fashion, namely we set $(\mathcal{F}_1)^{0j}_{0k} :=0$ for all $j,k$ and $(\mathcal{F}_2)^{i0}_{p0} :=0$ for all $i,p$, and then we compute the remaining values by integrating the fluxes $\mathbf{F}_1$ and $\mathbf{F}_2$ along the horizontal and vertical directions respectively iteratively on the cells, as follows
\begin{equation}\label{eq:bigF3}
\begin{split}
(\mathcal{F}_1)^{ij}_{pk} = (\mathcal{F}_1)^{ij}_{p0} + \int_{x^2_{0}}^{x^2_{k}} \mathbf{F}_{1,h}(t,x^1_{p},x^2) \,\mathrm{d}x^2,\\
(\mathcal{F}_2)^{ij}_{pk} = (\mathcal{F}_2)^{ij}_{0k} + \int_{x^1_{0}}^{x^1_{p}} \mathbf{F}_{2,h}(t,x^1,x^2_{k}) \,\mathrm{d}x^1.
\end{split}
\end{equation}
We use the properties of $ V_h^K$ a second time to impose that the point values obtained from \eqref{eq:bigF3} should be continuous when passing from one element to another.
So, we require that 
\begin{equation}\label{eq:bigF6}
(\mathcal{F}_1)^{ij}_{p0} =(\mathcal{F}_1)^{i j-1}_{pK}\;,\quad
(\mathcal{F}_2)^{ij}_{0k} =(\mathcal{F}_2)^{i-1 j}_{Kk}.
\end{equation}
To obtain the final discrete values, we now use twice the properties of the space $ V_h^K$. 
First,  we use  explicitly the ansatz   $\mathbf{F}_h\in V_h^K$, and  recast the above formulas as 
\begin{equation}\label{eq:bigF4}
\begin{split}
(\mathcal{F}_1)^{ij}_{pk} = (\mathcal{F}_1)^{ij}_{p0} + \sum_{\ell=0}^K\left(\int_{x^2_{0}}^{x^2_{k}} \varphi^2_\ell(x^2)  \,\mathrm{d}x^2\right) ( \mathrm{F}_1)^{ij}_{p\ell} = (\mathcal{F}_1)^{ij}_{p0}  +  
[(\id_1^{E_i^{1}} \otimes \mathcal{I}^{E^{2}_{j}}    )\mathrm{F}^{ij}_1]_{pk},\\
(\mathcal{F}_2)^{ij}_{pk} = (\mathcal{F}_2)^{ij}_{0k} + \sum_{\ell=0}^K \left(\int_{x^1_{0}}^{x^1_{p}} \varphi^1_\ell(x^1)   \,\mathrm{d}x^1\right) ( \mathrm{F}_2)^{ij}_{\ell k}= (\mathcal{F}_2)^{ij}_{0k}  + [ (\mathcal{I}^{E^{1}_{i}}\otimes \id_2^{E_j^{2}} ) \mathrm{F}^{ij}_2 ]_{pk}.
\end{split}
\end{equation}
On the reference element $[0,1]$, we define the  integration tables
\begin{equation}\label{eq:bigF5}
\mathcal{I}_{p k} := \int_{0}^{\xi_p} \varphi_k(\xi)\,d\xi 
\end{equation}
so that  we can formally define $\mathcal{I}_m^{E_{i}^{m}}= h_m \mathcal{I}$ for integration in the direction of  $x_m$,   $m = 1,\dots,d$.  
A very important remark is that \emph{by definition}  the table $\mathcal{I}$ coincides with the Butcher tableau of the fully implicit continuous 
collocation Runge-Kutta method associated to the collocation points $\{\xi_\ell\}_{\ell=0}^K$. When the Gauss-Lobatto points are used,
we recover again the LobattoIIIA method defined by \eqref{eq:stageLagrange}, in particular $\mathcal{I}_{m r}=\theta_r^m$, with $m,r = 0,\dots, K$ (see \cite[Chapter II.7]{Hairer1993}).
However, we prefer using separate notations, to differentiate integration in  space, using the $\mathcal{I}$ tables, from integration in time, using \eqref{eq:stageLagrange}.

With these definitions we have that \emph{by construction} the flux potentials verify the following property.
\begin{proposition}[Flux potentials and LobattoIIIA integration]\label{prop:Lobatto_integration}  Let $x^1_{\pi}$ and $x^2_{\kappa}$ be two fixed horizontal and vertical
coordinates. The nodal values of the flux potentials $(\mathcal{F}_1,\mathcal{F}_2)$ obtained from
\eqref{eq:bigF6}-\eqref{eq:bigF4}-\eqref{eq:bigF5} are equivalent to the integration of the ODEs
\begin{equation}\label{eq:bigF7}
\dfrac{d\mathcal{F}_{1,h}(t^n,x^1_{\pi},x^2)}{dx^2} = \mathbf{F}_1(t^n,x^1_{\pi},x^2) \;,\quad
\dfrac{d\mathcal{F}_{2,h}(t^n,x_{1},x^2_{\kappa})}{dx^1} = \mathbf{F}_2(t^n,x_{1},x^2_{\kappa})
\end{equation}
with the  LobattoIIIA collocation method (see \cite[Chapter II.7]{Hairer1993}).
\end{proposition}

Finally, concerning the source term, on a given element $E_{ij}$ we proceed similarly and   set
\begin{equation}\label{eq:bigF8}
\begin{split}
\mathcal{S}^{ij}_{pk} = \mathcal{S}^{ij}_{00}  - \sum_{m,  \ell=0}^K 
\left(\int_{x^1_{0}}^{x^1_{p}} \varphi^1_m(x^1)  \,\mathrm{d}x^1 \right) 
\left(\int_{x^2_{0}}^{x^2_{k}} \varphi^2_\ell(x^2)  \,dx^1\right) 
( \mathrm{S}^{ij})_{m \ell} = \mathcal{S}^{ij}_{00} -[(\mathcal{I}_1^{E_i^{1}} \otimes \mathcal{I}_2^{E_j^{2}}    )\mathrm{S}^{ij}]_{pk}.
\end{split}
\end{equation}

\begin{remark}
With the construction \eqref{eq:bigF6}-\eqref{eq:bigF4}-\eqref{eq:bigF5}, $\mathcal{F}_{1,h}$ and $\mathcal{F}_{2,h}$ are well defined functions in $V_h^K$, contrary to what would have happened by simply integrating the fluxes $\mathbf{F}_1$ in $x^2$, 
which would have lead to a function $\mathcal{F}_1 \in \mathbb P^{K} \otimes \mathbb P^{K+1}$ in each element and similarly to a $\mathcal{F}_2 \in \mathbb P^{K+1} \otimes \mathbb P^{K}$. 
This will be the key of the success of the method, as the resulting integrated fluxes (and source) $\mathcal{F}_{1,h}, \mathcal{F}_{2,h}, \mathcal{S}_h	 \in V_h^K$ can be balanced degree of freedom by degree of freedom consistently, 
and, as a result, locally we have for the flux vector $(\partial_y \mathcal{F}_{1,h}, \partial_x\mathcal{F}_{2,h}) \in  \mathbb P^{K} \otimes \mathbb dP^{K-1} \oplus  \mathbb dP^{K-1} \otimes \mathbb P^{K} $,
which is  essentially a Raviart-Thomas space. This local projection allows to satisfy the divergence conformal character of the method, with a classical collocated space.
\end{remark}

\subsection{Modification of the Galerkin variational form}

In this section, we revise the weak Galerkin form \eqref{eq:weakG0} in light of  the quadrature ansatz 
\eqref{eq:GF9}, complemented with  the discrete definitions  \eqref{eq:GF0a} and \eqref{eq:GF10}. We will focus on the modification
of the  divergence and source term, since no modification is done on the treatment of  the time derivative.
We start with the divergence operator,  given by \cite{barsukow2025structure,barsukow2025stationaritypreservingnodalfinite}
\begin{equation}\label{eq:GFdiv1}
\int\limits_\Omega\varphi_{\alpha}\left( \partial_{x^1} \partial_{x^2} \mathcal{F}_{1,h} + \partial_{x^2}\partial_{x^1} \mathcal{F}_{2,h}\right)\,d\mathbf{x}\equiv 
\left[( D_1\otimes D_2) \mathcal{F}_1 \right]_{\alpha} +\left[ (D_1\otimes D_2 )\mathcal{F}_2 \right]_{\alpha} = \left[( D_1\otimes D_2) (\mathcal{F}_1 +\mathcal{F}_2)\right]_{\alpha}  \,,
\end{equation}
where by abuse of notation $\mathcal{F}_1$ and $\mathcal{F}_2$ above denote the arrays with the nodal values of the potentials.
Several properties of the above discrete divergence can be revealed by  writing  the entries associated to each single element,  as in \eqref{eq:localDx1}, and using the explicit  relation between $\mathcal{F}_m$ and $\mathrm{F}_m$.
These properties are summarized in the     following result.
\begin{proposition}[Objectivity] The discrete divergence \eqref{eq:GFdiv1} can be equivalently written as 
\begin{equation}\label{eq:GFdiv2}
\int\limits_\Omega\varphi_{\alpha}\left( \partial_{x^1} \partial_{x^2} \mathcal{F}_{1,h} + \partial_{x^2}\partial_{x^1} \mathcal{F}_{2,h}\right)\,d\mathbf{x}\equiv 
\left[\sum_{E_{ij}\in\Omega} ( D_1^{E_i^{1}}\otimes D_2^{E_j^{2}} \mathcal{I}^{E_j^{2}}_2 ) \mathrm{F}_1^{E_{ij}} +  ( D_1^{E_i^{1}}   \mathcal{I}^{E_j^{1}}_1 \otimes D_2^{E_j^{2}} ) \mathrm{F}_2^{E_{ij}}  \right]_{\alpha}
\end{equation}
and the resulting discrete divergence
\begin{enumerate}
\item  is independent of the integration constants  $(\mathcal{F}_1)^{ij}_{p0}$ and $(\mathcal{F}_2)^{ij}_{0k}$;
\item  is invariant w.r.t. the  direction of integration in the definition of the potentials. 
\end{enumerate}
\begin{proof}
Expression  \eqref{eq:GFdiv2} can be obtained  by formally replacing the local definitions of the integrated potentials \eqref{eq:bigF4}.
In doing so, property 1.  is  given by showing that the bilinear interpolation of $(\mathcal{F}_1)^{ij}_{p0}$ and $(\mathcal{F}_2)^{ij}_{0k}$ leads to univariate polynomials,
thus disappearing when evaluating the mixed derivative  on the left hand side   \eqref{eq:GFdiv2}. Property 2. is a consequence of the relation
$$
\int_{x^2_{0}}^{x^2_{k}}\mathbf{F}_1\,\mathrm{d}x^2 -   \int_{x^2_{K}}^{x^2_{k}}\mathbf{F}_1\,\mathrm{d}x^2  = \int_{x^2_{0}}^{x^2_{K}}\mathbf{F}_1\,\mathrm{d}x^2.
$$
Since the right hand side  is independent on $k$, its interpolation is constant so the derivative  of the potentials on the left hand side, each using integration in opposite directions,
provides the same value.  For  all remaining details we refer to \cite[Proposition 6 and Proposition 9]{barsukow2025structure}, as well as to \cite[Section 4.1]{barsukow2025stationaritypreservingnodalfinite}.
\end{proof}
\end{proposition}
As done before, we can introduce  the global and local divergence operator arrays (denoted here with uppercase to distinguish them from the standard FEM formulation)
\begin{equation}\label{eq:GFdiv3}\begin{aligned}
\textsf{DIV}_h \mathbf{F}:=  & ( D_1\otimes D_2) \mathcal{F}_1   + (D_1\otimes D_2 )\mathcal{F}_2  \;,\\
\textsf{DIV}_h^{E_{ij}}  \mathbf{F}:=  &( D_1^{E_i^{1}}\otimes D_2^{E_j^{2}} \mathcal{I}^{E_j^{2}}_2 ) \mathrm{F}_1^{E_{ij}} +  ( D_1^{E_i^{1}}   \mathcal{I}^{E_j^{1}}_1 \otimes D_2^{E_j^{2}} ) \mathrm{F}_2^{E_{ij}}  .
\end{aligned}
\end{equation}
Very interestingly, the only difference between $ \textsf{DIV}_h^{E_{ij}}  \mathbf{F}$ and $\textsf{div}_h^{E_{ij}}  \mathbf{F}$ in \eqref{eq:weakG2} is that 
the mass matrices $M_m^{E_i^{m}}$ have been replaced by the operator $D_m^{E_i^{m}} \mathcal{I}^{E_i^{m}}_m$, still consistent with a mass matrix. 

Finally, in presence of source terms, \eqref{eq:GFdiv1} needs to be replaced by 
\begin{equation}\label{eq:GFdiv4}
\int\limits_\Omega\varphi_{\alpha}\left( \partial_{x^1} \partial_{x^2} \mathcal{F}_{1,h} + \partial_{x^2}\partial_{x^1} \mathcal{F}_{2,h} + \partial_{x^1} \partial_{x^2}\mathcal{S}_h\right)\,d\mathbf{x}\equiv 
\left[( D_1\otimes D_2) \mathcal{F}_1 \right]_{\alpha} +\left[ (D_1\otimes D_2 )\mathcal{F}_2 \right]_{\alpha} +\left[ (D_1\otimes D_2 )\mathcal{S} \right]_{\alpha} \,.
\end{equation}
In this case, using also \eqref{eq:bigF8}, one is led to introduce the arrays of \emph{residuals} 
\begin{equation}\label{eq:GFdiv5}\begin{aligned}
\mathsf{R}_h  :=  &\; ( D_1\otimes D_2) \mathcal{F}_1   + (D_1\otimes D_2 )\mathcal{F}_2 + (D_1\otimes D_2 )\mathcal{S}   \;,\\
\mathsf{R}_h^{E_{ij}}   :=  &\;\textsf{DIV}_h^{E_{ij}}  \mathbf{F}-   ( D_1^{E_i^{1}}  \mathcal{I}^{E_j^{1}}_1\otimes D_2^{E_j^{2}} \mathcal{I}^{E_j^{2}}_2 ) \mathrm{S}.
\end{aligned}
\end{equation}

\subsection{SU stabilization in GFQ formulation}

We will proceed as before adding to the variational form the stabilization term arising from the model equation \eqref{eq:mod_eq1}. 
In this case, we  use the integrated potentials to evaluate   the divergence operator on the right hand side. 
The resulting scheme reads 
\begin{equation} \label{eq:weakG1-SUGF}
\mathsf{M}\dfrac{d}{dt}W + \mathsf{R}_h  + \mathsf{ST}_h(\mathrm{F}_1,\mathrm{F}_2,\mathrm{S}) =0,
\end{equation}
where $\mathsf{M}$ is exactly the same non-diagonal mass matrix appearing in \eqref{eq:weakG1-SU}, the Galerkin residual  $\mathsf{R}_h$ is defined by \eqref{eq:GFdiv5},
and now the stabilization term  (denoted here with uppercase) is evaluated as follows 
 \begin{equation} \label{eq:weakG2-SUGF}
[ \mathsf{ST}_h(\mathrm{F}_1,\mathrm{F}_2,\mathrm{S})]_{\alpha}  :=  \sum_{E_{ij}\in\Omega}\int\limits_{E_{ij}} \nabla \varphi_{\alpha} \cdot\mathbf{J}\tau \left( \partial_{x^1}\partial_{x^2}\mathcal{F}_{1,h}
+\partial_{x^2}\partial_{x^1}\mathcal{F}_{2,h} + \partial_{x^1}\partial_{x^2}\mathcal{S}_{h} \right)\,d\mathbf{x}.
 \end{equation}
Replacing the definitions of the potentials \eqref{eq:bigF4} and \eqref{eq:bigF8}, and following the same developments of Section~\ref{sec:stabilization}, we can show that  the streamline upwind term can be written as 
 \begin{equation} \label{eq:weakG3-SUGF}
\begin{aligned}
\left[ \mathsf{ST}_h(\mathrm{F}_1,\mathrm{F}_2,\mathrm{S})\right]_{\alpha}  := \Bigg[ \sum_{E_{ij}\in\Omega} 
&\left\{ \left( (D_1^{E_i^{1}})^{\mathrm{T}}  (M_1^{E_i^{1}})^{-1} \otimes \id_2^{E_j^{2}} \right)   \textrm{diag}(\mathbf{J}_1 \,\tau)  \right\} 
\mathsf{R}_h^{E_{ij}}\\
+&\left\{ \left( \id_1^{E_i^{1}}  \otimes (D_2^{E_j^{2}})^{\mathrm{T}}  (M_2^{E_j^{2}})^{-1}   \right)   \textrm{diag}(\mathbf{J}_2 \,\tau)  \right\} 
\mathsf{R}_h^{E_{ij}}
\Bigg] _{\alpha}
\end{aligned}
 \end{equation}
with $\mathsf{R}_h^{E_{ij}}$ the array of local residuals defined in \eqref{eq:GFdiv5}.  As for the Galerkin term,   the main difference compared to the $\mathsf{st}_h$ term appearing in
\eqref{eq:weakG1-SU} is related to the modification of the mass matrices appearing in the $\textsf{DIV}_h^{E_{ij}}  \mathbf{F}$  term  present in 
$\mathsf{R}_h^{E_{ij}}$ (cf. equations \eqref{eq:GFdiv5} and \eqref{eq:GFdiv3}).  

We will denote the schemes defined by \eqref{eq:weakG1-SUGF} and \eqref{eq:weakG2-SUGF} as \emph{SUPG-GFQ} methods.

\subsection{Time evolution using DeC}

The integration in time of \eqref{eq:weakG1-SUGF}  is performed following \emph{exactly} the steps underlying the Defect Correction method   discussed in Section~\ref{sec:DeC}. The only difference is that now
we set 
\begin{equation}
\begin{split}
	\mathcal{L}^{H,m}( \mathsf{W}  ) :=   &\mathsf{M}   (\mathrm{W}^m -  \mathrm{W}^0 )
	+ \Delta t \sum_{r=0}^M \theta_r^m M^{\textsf{SU}}(\mathrm{W}^r) \dfrac{d\mathrm{W}(t^r)}{dt} 
	+ \Delta t \sum_{r=0}^M \theta_r^m \{ \textsf{R}_h(\mathrm{W}^r)    + \mathsf{ST}_h(\mathrm{W}^r) \},
	 \end{split}
\end{equation}
for the  high order  operator, while the low order operator is defined as 
\begin{equation}
\begin{split}
	\mathcal{L}^{L,m}(\mathsf{W}  ) :=   \mathsf{M}   (\mathrm{W}^m -  \mathrm{W}^0 )
	+& \Delta t \ \{ \textsf{R}_h(\mathrm{W}^0)    + \mathsf{ST}_h(\mathrm{W}^0)\}.  
	 \end{split}
\end{equation}
This results in the   explicit multi-stage iterations for all $k=1,\dots, \kappa$ and $m=1,\dots,M$:
\begin{equation}
	\label{eq:DeC-GF}
	\begin{split}
	     \mathrm{W}^{m,(k)}  =            \mathrm{W}^n
	- &\Delta t \sum_{r=0}^M \theta_r^m \mathsf{M}^{-1}  M^{\textsf{SU}}(\mathrm{W}^{r,(k-1)}) \dfrac{d\mathrm{W^{(k-1)}}(t^r)}{dt} \\
	-& \Delta t \sum_{r=0}^M \theta_r^m \mathsf{M}^{-1}  \{ \textsf{R}_h(\mathrm{W}^{r,(k-1)})    + \mathsf{ST}_h(\mathrm{W}^{r,(k-1)}) \}
	 \end{split}
\end{equation}
 with $\mathsf{M}$ diagonal. As already remarked, despite of the non-diagonal mass matrix contributions of the SU term,
this method has the structure of an explicit multi-stage Runge-Kutta scheme.

\section{Properties  of the  methods}\label{sec:properties}
In the following section, we summarize the main properties of the presented method.

\subsection{Local conservation and Lax Wendroff estimates}\label{sec:local_conservation}

The   preprocessing introduced leads to a locally discontinuous approximation of the flux components, evaluated as derivatives of the potentials.
 In this section, we   show that, despite this, and without the use of any numerical flux, the new method 
 is locally conservative, and  verifies  a Lax-Wendroff theorem. 
We consider first  the case in which there is no source term, then give some simple arguments to include   uniformly bounded sources. 
Similar results are known for the standard SUPG, see    e.g.   \cite{AR:17,Abgrall2022,amr2025}.

The main tool we will use is the residual distribution framework \cite{AR:17}, in which the scheme will be recast through a fluctuation form. Following  then the   general formulation analyzed in the review paper \cite{amr2025},
we will prove that the local fluctuations sum up to a contour integral of the flux. 

Both schemes, the SUPG \eqref{eq:DeC2} and the SUPG-GFQ \eqref{eq:DeC-GF}, can be written for a collocation node $\alpha$ as 
\begin{equation}\label{eq:DeC-RD}
|\omega_{\alpha} |\mathrm{W}^{m,(k)}_{\alpha} = |\omega_{\alpha}| \mathrm{W}^n_{\alpha} - \Delta t \sum_{r=0}^M \sum_{E_{ij}\in \Omega} [\Phi_{\alpha}^{ij}]^{r,(k-1)}
\end{equation}
having set $|\omega_{\alpha}|=\mathsf{M}_{\alpha\alpha}$, which is the area of a nodal cell of side lengths  proportional to Gauss-Lobatto   quadrature  weights. 
In particular, for the stationarity preserving method we  have
\begin{equation}\label{eq:fluct_su_gf}
\Phi_{\alpha}^{ij} = \int_{E_{ij}} \Big( \varphi_{\alpha} ( \partial_{x^1} \partial_{x^2} \mathcal{F}_{1,h} + \partial_{x^2}\partial_{x^1} \mathcal{F}_{2,h}) + \nabla\varphi_{\alpha} \cdot\mathbf{J}\,\tau\,\dfrac{d W_h }{dt} +   \nabla\varphi_{\alpha} \cdot\mathbf{J}\,\tau\, ( \partial_{x^1} \partial_{x^2} \mathcal{F}_{1,h} + \partial_{x^2}\partial_{x^1} \mathcal{F}_{2,h})\Big)d\mathbf{x} .
\end{equation}
Following \cite{amr2025}, we now need to show that 
\begin{equation}\label{eq:local_conservation}
\sum_{\alpha\in E_{ij}} \Phi_{\alpha}^{ij} = \oint_{\partial E_{ij}} \! \!\mathbf{F}_h\cdot \hat{n}\, dS.
\end{equation}
We first note that for any function $f$ the interpolation properties of the shape functions imply  
\begin{equation}\label{eq:sum_nabla_varphi}
\sum_{\alpha\in E_{ij}}  \varphi_{\alpha}(\mathbf{x}) \equiv 1\;\forall\, \mathbf{x} \Rightarrow \sum_{\alpha\in E_{ij}} \int_{E_{ij}}  \nabla\varphi_{\alpha} \, f\,d\mathbf{x} = \bigintss_{E_{ij}}  \nabla \left\{\sum_{\alpha\in E_{ij}}  \varphi_{\alpha} \right\}   f\,d\mathbf{x}  =0.
\end{equation}
So, the stabilization terms vanish in \eqref{eq:local_conservation}. Concerning the remaining term we proceed as follows.  First, we note that in \eqref{eq:fluct_su_gf} we can remove from  $\mathcal{F}_{1,h}$ and $\mathcal{F}_{2,h}$  any univariate polynomial due to the presence of the mixed derivatives.
 Recalling \eqref{eq:GF4}, we can use \eqref{eq:sum_nabla_varphi} to readily show that
 \begin{equation}\label{eq:Phi_gf}
 \begin{aligned}
  \partial_{x^1} \partial_{x^2} \mathcal{F}_{1,h} + \partial_{x^2}\partial_{x^1} \mathcal{F}_{2,h} =   & \; \partial_{x^1} \partial_{x^2}\Psi^{ij}_h\;, \text{ for }\mathbf{x} \in E_{ij},  \\
[  \Psi^{ij}_h]^{ij}_{pq} := &\;[\mathcal{F}_{1,h} - \mathcal{F}_{1,h}(x^1_{i,0},x^2) + \mathcal{F}_{2,h} - \mathcal{F}_{2,h}(x^1, x^2_{j,0})]_{pq} =  \int_{x^2_{j,0}}^{x^2_{j,q}} \int_{x^1_{i,0}}^{x^1_{i,p}} \nabla\cdot\mathbf{F}_h\,d\mathbf{x}
  \end{aligned}
 \end{equation}
 having used Schwarz's theorem for the polynomials  $\mathcal{F}_{1,h}$ and $\mathcal{F}_{2,h}$, and where the last  equality is readily shown using the definition of the potentials  \eqref{eq:bigF2}  and \eqref{eq:bigF3},
and noting that 
$$
[\mathcal{F}_{1,h} - \mathcal{F}_{1,h}(x^1_{i,0},x^2)]_{pq}=
 \int_{x^2_{j,0}}^{x^2_{j,q}}[  \mathbf{F}_h(x^1_{i,p},x^2) -\mathbf{F}_h(x^1_{i,0},x^2)]\,\mathrm{d}x^2 =  \int_{x^2_{j,0}}^{x^2_{j,q}}\int_{x^1_{i,0}}^{x^1_{i,p}} \partial_x \mathbf{F}_h(x^1,x^2) \,d\mathbf{x},
$$
and similarly for the   term $ \mathcal{F}_{2,h} - \mathcal{F}_{2,h}(x^1, x^2_{j,0})$. Note that all integrals above are exact by construction, including in \eqref{eq:Phi_gf}. In particular,  the array $\Psi^{ij}_h$ contains the  integrals of the   divergence of the  interpolated flux over the sub-cells 
$\{[x^1_{i,0},x^1_{i,p}] \times [ x^2_{j,0},x^2_{j,q} ]\}_{p,q=0,\dots,K}$. Using again $\sum_\alpha\varphi_\alpha=1$, the exactness of the Gauss-Lobatto formulas, and \eqref{eq:Phi_gf}, we can now write:
\begin{equation}\label{eq:proofPhi}
\sum_{\alpha\in E_{ij}} \Phi_{\alpha}^{ij}  =   \sum_{\alpha\in E_{ij}}\int_{E_{ij}}  \partial_{x^1} \partial_{x^2}\Psi^{ij}_h\ \,d\mathbf{x} = [  \Psi^{ij}_h]_{KK} -  [  \Psi^{ij}_h]_{0K} - (  [\Psi^{ij}_h]_{K0} -  [\Psi^{ij}_h]_{00}) = [  \Psi^{ij}_h]_{KK}.
\end{equation}
By definition $[  \Psi^{ij}_h]_{KK}$ is the integral of the  divergence of  the interpolated flux over the whole element  $E_{ij}$. As before, the  exactness of the Gauss-Lobatto formulas allow to apply Gauss' theorem, and lead to \eqref{eq:local_conservation}.  Using this fact and  the continuity of the flux approximation $\mathbf{F}_h$, and under some relatively classical continuity and boundedness assumptions,
the scheme  verifies  a Lax-Wendroff result  (see \cite[Section \S4.1]{amr2025}).  Moreover, one can exhibit the existence of consistent numerical fluxes  associated to the nodal cell  $\omega_\alpha$,
which depend on  the $\Phi_{\alpha}^{ij}$.
The interested reader can refer to \cite{amr2025,AR:17,abgrall2020notion,Abgrall2022,barsukow2025genuinely,GRD26} for more  on these aspects.

Concerning the source term,    for bounded data and bounded source (so no discontinuous potentials), we can proceed classically (see e.g. \cite{amr2025,SHI20183}),
and study the pseudo weak form of the method. Note that even with these hypotheses   the use of  the $\partial_{x^1 x^2}  \mathcal{S}_h$  term in the discretisation makes it less clear that we indeed have consistency with weak solutions.
To clarify this,   we observe that under boundedness assumptions of $\mathbf{S}_h$, we also have 
 boundedness of $\partial_{x^1x^2}\mathcal{S}_{h}$. In particular, since $\partial_{x^1}\varphi^1_{p}(x^1) \partial_{x^2}\varphi^2_{q}(x^1) \sim   h^{-2}$, we can readily show that
$$
\left\lVert \partial_{x^1x^2}\mathcal{S}_{h}\right\rVert= \left\lVert \sum_{E_{ij}\in\Omega} \sum_{p,q=0}^K\partial_{x^1}\varphi^1_{p}(x^1) \partial_{x^2}\varphi^2_{q}(x^1) \int_{x^1_{i,0}}^{x^1_{i,p}}  \int_{x^2_{j,0}}^{x^2_{j,q}}  \mathbf{S}_h\,d\mathbf{x} \right\rVert
\le (C_1 h^{-1}) (C_2 h^{-1}) h^2 \left\lVert \mathbf{S}_h   \right\rVert=C_1 C_2 \left\lVert \mathbf{S}_h   \right\rVert
$$
for some bounded constants  $C_1$ and $C_2$.  We now  study, for any given smooth    function $\psi$, the quantity 
 $$
\sum_{\alpha}\psi_{\alpha}\left\{
\sum_{E_{ij}\in\Omega} \int_{E_{ij}}  \left(\varphi_{\alpha} \partial_{x^1x^2}\mathcal{S}_{h}+
   \nabla\varphi_{\alpha} \cdot\mathbf{J}\,\tau\, \partial_{x^1x^2}\mathcal{S}_{h}\right)\,d\mathbf{x} \right\}.
$$
where $\psi_{\alpha}=\psi(\mathbf{x}_{\alpha})$. We can readily show for the stabilization term that
$$
\sum_{\alpha}\psi_{\alpha}\sum_{E_{ij}\in\Omega} \int_{E_{ij}} \nabla\varphi_{\alpha} \cdot\mathbf{J}\,\tau\, \partial_{x^1x^2}\mathcal{S}_{h}\,d\mathbf{x}
= \sum_{E_{ij}\in\Omega} \sum_{\alpha} (\psi_{\alpha} - \bar\psi_{E_{ij}}) \int_{E_{ij}} \nabla\varphi_{\alpha} \cdot\mathbf{J}\,\tau\,  \partial_{x^1x^2}\mathcal{S}_{h} \,d\mathbf{x}=\mathcal{O}(h)
$$
with $\bar\psi_{E_{ij}}$ the average of $\psi$ over $E_{ij}$. The last equality  uses the smoothness of $\psi$ implying  $\psi_{\alpha} - \bar\psi_{E_{ij}}   =\mathcal{O}(h) $,
the  boundedness of $\mathbf{S}_h$ and $\partial_{x^1x^2}\mathcal{S}_{h}$,   and the fact that $\nabla\varphi_{\alpha} \cdot\mathbf{J}\,\tau$ is also bounded by definition of $\tau$.
Note that the identity $\sum_{\alpha}\nabla\varphi_\alpha=0$ has been used  to remove the $\bar\psi_{E_{ij}} $ term.
Now, we consider  
$$
\begin{aligned}
&\sum_{\alpha}\sum_{E_{ij}\in\Omega} \int_{E_{ij}} \psi_{\alpha}\varphi_{\alpha} \partial_{x^1x^2}\mathcal{S}_{h}\,d\mathbf{x} =
\sum_{E_{ij}\in\Omega} \int_{E_{ij}} \bar\psi_{E_{ij}} \partial_{x^1x^2}\mathcal{S}_{h}\,d\mathbf{x} 
+  \sum_{E_{ij}\in\Omega} \int_{E_{ij}} (\psi_h - \bar\psi_{E_{ij}}   ) \partial_{x^1x^2}\mathcal{S}_{h}\,d\mathbf{x}\\
= & \sum_{E_{ij}\in\Omega} \bar\psi_{E_{ij}} \int_{E_{ij}}\mathbf{S}_h\,\mathrm{d}x +   \sum_{E_{ij}\in\Omega} \int_{E_{ij}} (\psi_h - \bar\psi_{E_{ij}})   \partial_{x^1x^2}\mathcal{S}_{h}\,d\mathbf{x}\\
= &    \int_{\Omega} \psi_ h\mathbf{S}_h\,\mathrm{d}x 
 + \sum_{E_{ij}\in\Omega} \int_{E_{ij}} (\psi_h - \bar\psi_{E_{ij}})   (\partial_{x^1x^2}\mathcal{S}_{h}-\mathbf{S}_h)\,d\mathbf{x}   =
 \int_{\Omega} \psi_ h\mathbf{S}_h\,\mathrm{d}x  + \mathcal{O}(h).
\end{aligned}
$$
The first term in the second line is obtained proceeding as in  \eqref{eq:proofPhi}, and the last  estimate follows from the boundedness of $\partial_{x^1x^2}\mathcal{S}_{h}-\mathbf{S}_h$,
and the smoothness of $\psi$. The rest of the analysis is classical, see     \cite{herbin23,amr2025}.

\subsection{Discrete kernel and stationarity preservation}
We show that the new method has a rich set of discrete stationary states. 
 The main result is   the following.

\begin{proposition}\label{prop:equilibria}
For any arbitrary  univariate functions $\mathbf{f}(x^1)$ and $\mathbf{g}(x^2)$, with discrete interpolated counterparts $\mathsf{f}$ and $\mathsf{g}$,
the general class of functions $\mathbf{W}_h$ such that their integrated fluxes and source can be written as 
\begin{equation}\label{eq:equilibria}
	\begin{split}
	 \mathcal{F}_{1,h}(\mathbf{W}_h)(\mathbf{x}) + \mathcal{F}_{2,h}(\mathbf{W}_h)(\mathbf{x}) + \mathcal{S}_h(\mathbf{W}_h)(\mathbf{x}) = &\; \mathbf{f}_h(x^1)+\mathbf{g}_h(x^2) , \text { or discretely}\\[5pt]
		(\mathcal{F}_{1})^{ij}_{pk} + (\mathcal{F}_{2})^{ij}_{pk} + \mathcal{S}^{ij}_{pk} =& \; \mathsf{f}^i_p + \mathsf{g}^j_k,
	\end{split}
\end{equation} 
describe discrete stationary solutions of \eqref{eq:weakG1-SUGF} and \eqref{eq:DeC-GF}.
\end{proposition}
\begin{proof}
The proof is provided in \cite{barsukow2025structure,barsukow2025stationaritypreservingnodalfinite}. The main
idea is to  show that, in every element $E_{ij}$,  $\textsf{DIV}_h^{E_{ij}}$ vanishes if   applied to a function of the form $\mathbf{f}(x^1)+\mathbf{g}(x^2) \in (V_h^k)^s$, as in \cite[Proposition 5]{barsukow2025structure}, and that the stabilization term also vanishes on such functions, as in \cite[Proposition  10]{barsukow2025structure}. The extension to non-homogenous systems involves showing the same for  $\mathsf{R}_h^{E_{ij}}$, see \cite[Proposition 5]{barsukow2025stationaritypreservingnodalfinite} and 
\cite[Proposition 8]{barsukow2025stationaritypreservingnodalfinite}. 
\end{proof}

The above result shows that the flux pre-processing associated to  \eqref{eq:GF0a} and \eqref{eq:GF10} allow us to go beyond the negative result of section~\ref{sec:lack}, and to have a rich set of discrete stationary states.   Note that, as the proposition says, \emph{any}  functions   $\mathbf{f}(x)$ and $\mathbf{g}(y)$ define a family of the kernel.
Physically relevant solutions, however, are only obtained for
 \begin{equation}\label{eq:equilibria-phys}
 \begin{aligned}
 &\mathbf{g}_h(x^2) =  \mathcal{F}_{1,h}(\mathbf{W}_h)(x^1 =x^1_0,x^2)\;, \quad   \mathbf{f}_h(x^1) =\mathcal{F}_{2,h}(\mathbf{W}_h)(x^1, x^2 =x^2_0)\\
\Rightarrow & \int_{x^2_0}^{x^2}[\mathbf{F}_{1,h}-\mathbf{F}_{1,h}(x^1 =x^1_0,x^2) ] \mathrm{d}x^2 + \int_{x^1_0}^{x^1}[\mathbf{F}_{2,h}-\mathbf{F}_{2,h}(x^1,x^2=x^2_0) ] \mathrm{d}x^1=
 \int_{x^1_0}^{x^1} \int_{x^2_0}^{x^2}\mathbf{S}_hd\mathbf{x}
\end{aligned}
 \end{equation} 
having used the definition of the potentials to obtain the second relation.
This allows us to define a  particular case of the above family of equilibria that is defined by the proposition below.  
\begin{proposition}\label{prop:equilibria_local}
The class of functions $\mathbf{W}_h$ belonging to the kernel of Proposition \ref{prop:equilibria} 
with $\mathbf f$ and $\mathbf g$ given by \eqref{eq:equilibria-phys}, also verifies    $\Psi_h+\mathcal{S}_h=0$  $\forall E_{ij}\in \Omega$,
with $\Psi_h$  the  array of integrated divergences  \eqref{eq:Phi_gf}.
\end{proposition}
\begin{proof}
 The result is shown iteratively. 
 Within this first element $E_{00}$ we integrate on each subcell $[x^1_{0,0},x^1_{0,p}]\times [x^2_{0,0},x^2_{0,q}]$ with increasing $p$ and $q$ from $0$ to $K$.
 This allows us to obtain $\Psi_h+\mathcal{S}_h=0$  for $E_{00}$. We then  increase $i$ and $j$ by 1 in each direction, and use
 the result on $E_{00}$ to prove the same on $E_{10}$ and $E_{01}$. We continue iterating on $i$ and $j$ until all $E_{ij}$ have been covered.
\end{proof}
To link this result to the notion of stationarity preservation introduced in Condition~\ref{cond:SP}, we remark that the condition   $\Psi_h+\mathcal{S}_h=0$  
contains some trivial entries associated to the degrees of freedom on the element boundaries $x^1=x^1_{i,0}$ and $x^2=x^2_{j,0}$. A simple count shows that
this leaves $K^2$ conditions. As a consequence,  on a square mesh with $N_1 = N_2$, 
the system $\Psi_h+\mathcal{S}_h=0$   $\forall E_{ij}\in \Omega$ provides $K^2N^2$ constraints times the size of the system $s$, 
for  the $(KN+1)^2\times s = K^2N^2\times s + ( 2KN+1 )\times s$ unknowns. Differently from the case discussed in section~\ref{sec:lack}, here we always have 
a greater number of unknowns than the number of constraints. The method thus verifies  the necessary Condition~\ref{cond:SP} for stationarity preservation.
 
Proposition~\ref{prop:equilibria_local} is not enough to guarantee that the proposed scheme falls into Definition~\ref{def:SP_original}, as we cannot state that for every stationary solution of the PDE there exists a discrete stationary solution of the scheme. This would require the solution of a nonlinearly constrained optimization problem as 
\begin{equation}
\min_{\mathbf{W}_h} \|\mathbf{W}_h - \mathbb{W}\|^2 \quad \text{subject to } \mathsf{R}_\alpha(\mathbf{W}_h) = 0, \quad \forall \alpha.
\end{equation}
To guarantee that the solution of such problem exists and converges to the exact solution $\mathbb{W}$ with the order of accuracy of the scheme, we would need to prove that the constraints are regular, and that a solution exists,
and potentially  that it is unique e.g. for certain given   initial guess data.  This aspect  is not studied in the present work. Our results  only guarantee
the consistency with rich discrete stationary states, and study the properties of these states below.

\subsection{Pointwise super-convergence estimate at equilibrium}
Previous works on global flux finite element approximations based on collocated Gauss-Lobatto elements have all shown  that the kernel
of the scheme defines data inheriting the consistency properties of the LobattoIIIA  collocation method (see \cite{mantri2024fully,barsukow2025structure,barsukow2025stationaritypreservingnodalfinite} and \cite{Hairer1993}).
We prove here a similar result for the projection of exact data onto the kernel of the scheme. To this end, we recall that the definition of the potentials
verify the characterization of Proposition~\ref{prop:Lobatto_integration}. We can thus use the consistency error of the  LobattoIIIA method for each potential.
As we recall, this consistency is of order $h^{K+2}$ for internal degrees of freedom (in each one-dimensional direction), and order $h^{2K}$ for the endpoints of each one-dimensional element.
Assume now to have an exact smooth stationary solution, and consider for example   the first in \eqref{eq:bigF7}. If $\mathbf{F}_1^e$ is the exact flux, 
we can   write at a generic \emph{grid node} $x^{2,\star}$ (see \cite[Chapter 7]{Hairer1993})
\begin{equation}\label{eq:L3A_estimate1}
\mathcal{F}_{1,h}(t^n,x^1_{{\pi}}, x^{2,\star}) -\mathcal{F}_{1,h}^e(t^n,x^1_{{\pi}}, x^{2,\star}) = \mathcal{O}(h^{K+2}).
\end{equation}
The exact value of the potential is trivially given by the exact integral of the flux, this we have, by definition of the LobattoIIIA method   
\begin{equation}\label{eq:L3A_estimate2}
\int\limits_{x^2_0}^{x^{2,\star}}(\mathbf{F}_{1,h}^e   -\mathbf{F}_1^e)d x^2 = \mathcal{O}(h^{K+2}).
\end{equation}
Since $x^{2,\star}$ is finite, we deduce that $\mathbf{F}_{1,h}^e  -\mathbf{F}_1^e$ is also of an order $\mathcal{O}(h^{K+2})$, and similarly we can proceed for the flux in the $x^2$ direction,
using the second in \eqref{eq:bigF7}, and for the source using integrals in alternate directions (assuming we are given exact and smooth values of $\partial_{x^1}\mathbf{S}^e$ and $\partial_{x^2}\mathbf{S}^e$).
This allows us to infer   the properties of the  flux projection associated to the LobattoIIIA method, and to prove the following  estimate. 
\begin{proposition}[Kernel consistency  estimate]\label{prop:consistency_equilibria} Let $\mathbf{W}^e$ be  a smooth enough stationary solution of \eqref{Euler_Eq3}, and let  $\mathbf{F}^e$, and 
$\mathbf{S}^e$ be  the associated  exact flux tensor and source term.  Consider now the steady limit of the   method obtained from \eqref{eq:DeC-RD}:
$$
\sum_{E_{ij}\in\Omega}\Phi_{\alpha}^{ij}=0 ,\qquad \forall \alpha.
$$
Let $\psi$ be an arbitrary smooth test function with compact support  $\psi \in C^1_0(\Omega)$, and consider the following truncation error at steady state:
\begin{equation}\label{eq:TE0}
\mathsf{E}_h = \sum_{\alpha}\psi_\alpha \sum_{E_{ij}\in\Omega}\Phi_{\alpha}^{ij}(\mathbf{W}^e).
\end{equation}
If  $\Phi_{\alpha}^{ij}$ are  the fluctuations of the stationarity preserving method   \eqref{eq:fluct_su_gf}, then the following truncation error estimate holds:
\begin{equation}
\|\mathsf{E}_h\| \le C h^{K+1+t}  
\end{equation}
with $t=0$ for polynomial degree $K=1$ and $t=1$ for any polynomial degree $K\ge 2$.
\begin{proof}
See appendix \ref{app:consistency}
\end{proof}
\end{proposition}
The above consistency definition follows a classical  characterization used for residual distribution \cite{RD-ency,AR:17,amr2025}. Note that, despite being evaluated at stationary state,
the fluctuations $\Phi_{\alpha}^{ij}$ do not vanish, as we do not evaluate them on the discrete kernel, but formally replacing sampled values of the exact solution in the scheme.
The remainder is   classically used as a measure of the truncation error, and in the case of the stationarity preserving method, and thanks to the embedding of the LobattoIIIA method,
leads to a super-convergence result for all $K\ge 2$.

\begin{remark}[Kernel consistency with the low Mach limit] 
	It has been suggested in~\cite{barsukow2019stationarity} that numerical methods for the Euler equations whose linearization 
	(i.e., the corresponding method for linear acoustics, see~\cite{barsukow2025structure}) preserves stationarity are inherently low Mach number compliant. 
	A nonlinear stationarity-preserving method naturally shares this property, and several experimental confirmations of this behavior can be found in \cite{barsukow2018low}.
	Indeed, the numerical dissipation brought by upwind terms is what in both cases prevents the preservation of stationarity solutions.
	Centered schemes, on the other hand, are known to have no numerical dissipation and can preserve stationarity and be low Mach compliant.
	
	For the SUPG-GFQ,  besides the super-convergence, the stationary kernel has another remarkable property, which is that 
	it is compatible with centered approximations at each degree of freedom at the equilibrium.
	This can be seen looking at states verifying the local equilibria characterization 
	of Proposition~\ref{prop:equilibria_local}. There, the condition $\Psi_h+\mathcal{S}_h=0$ corresponds to $K^2$ constraints for every element $E_{ij}$.
	Without loss of generality, we can consider these conditions to be related to the integrals over the $K^2$ subcells $E_{ij}^{pq}\equiv [ x_{i,p}^1, x_{i,p+1}^1 ]\times[x_{j,q}^2, x_{j,q+1}^2]$ for $p,q\in\{0, K-1\}$.
	In other words, the kernel of Proposition~\ref{prop:equilibria_local} verifies
	$$
	\int_{x_{i,p}^1}^{ x_{i,p+1}^1 } \int_{x_{j,q}^2}^{ x_{j,q+1}^2} (\nabla\cdot\mathbf{F}_h -\mathbf{S}_h)\,d\mathbf{x} =0 \;\forall \,p,q\in\{0, K-1\}.
	$$
	
	There are interesting linear combinations of these (sub-)elemental conditions that are equivalent to the original ones and that coincide with classical central approximations of the divergence at the grid nodes.
	For example, for $\mathbb Q^1$ elements, when applied to a variable $p$, the gradients residuals can be written as a function of the centred approximation:
	$$
	\begin{aligned}
		\left. \dfrac{d}{dx}\right\rvert_{ij}  p = & \dfrac{p_{i+1,j+1} + 2 p_{i+1,j} + p_{i+1,j-1} }{8\,h} - \dfrac{p_{i-1,j+1} + 2 p_{i-1,j} + p_{i-1,j-1} }{8\,h},\\
		\left. \dfrac{d}{dy}\right\rvert_{ij}  p = & \dfrac{p_{i+1,j+1} + 2 p_{i,j+1} + p_{i-1,j+1} }{8\,h} - \dfrac{p_{i+1,j-1} + 2 p_{i,j-1} + p_{i-1,j-1} }{8\,h}.
	\end{aligned}
	$$
	The interesting property is that for low Mach stationary states, the kernel behaves 
	as the kernel of centered approximations, which are known to provide the correct asymptotic limit. 
	This heuristic argument will be verified in the numerical computations proposed in the results section.
	The full theoretical analysis is still ongoing and will be part of future work.  
\end{remark}

\subsection{A correction to preserve exactly hydrostatic  states}\label{sec:hydro_equi}

As in previous works on global flux quadrature for shallow water models \cite{Ciallella2023,mantri2024fully,Micalizzi2024,Kazolea2025}, we provide a correction to preserve \emph{exactly} static equilibria.
The  idea is to exploit the fact that  for \eqref{Euler_Eq0} we need to provide a definition of the nodal   gradient of the  potential $\phi$. This can be done in a way compatible with the stationary kernel.
To do this, we adapt an idea due to  \cite{CZ17}. The reference proposes schemes that can be well balanced for both isothermal and  isoentropic equilibria, however not simultaneously:
one has to choose a priori which equilibrium to preserve. For simplicity, we use here the correction for the isothermal case. The modification for the isoentropic can be deduced easily following \cite{CZ17}.

Given reference values $(\bar\rho,\bar p)$, isothermal hydrostatic equilibria are determined by the relation
\begin{equation}\label{eq:isothermal_equi}
\dfrac{p}{\bar p} =\dfrac{\rho}{\bar\rho}.
\end{equation}
For states at rest, the momentum equation can be written as 
\begin{equation}
0 = \nabla p + \rho \nabla \phi =  \rho  \nabla\kappa \Rightarrow  \kappa := \bar p\ln (\rho)   + \bar\rho \phi .	
\end{equation}
Hydrostatic equilibria are characterized by $\kappa \equiv \kappa_0$. To integrate the source, we write $\rho$ in function of $\kappa$ and $\phi$: 
\begin{equation}\label{eq:rho_from_kappa}
	\rho =  e^{\kappa/\bar p}\, e^{-\bar\rho \phi/\bar p}.
\end{equation}
Noting that $\nabla e^{-\bar{\rho} \phi/\bar p} = - \frac{\bar \rho}{\bar p} e^{-\bar\rho \phi/\bar p} \nabla \phi$ and using \eqref{eq:rho_from_kappa}, we can rewrite the integration of the source equivalently as follows
$$
\int \rho \nabla \phi  = -  \int  \frac{\bar p}{\bar \rho} e^{\kappa / \bar p}    \nabla   e^{-\bar \rho \phi/\bar p}.
$$
In particular, we set for the nodal source in each element $E_{ij}$:
\begin{equation}\label{eq:wb}
\left( \rho \nabla \phi \right)_h(\mathbf x):=  -  \frac{\bar p}{\bar \rho} \sum_{\alpha,\beta}    \varphi_\alpha(\mathbf x)       e^{\kappa_\alpha /\bar p} \nabla \varphi_\beta(\mathbf x)   e^{-\bar \rho \phi_\beta/\bar p},
\end{equation}
which at hydrostatic equilibrium, i.e., when $\kappa_\alpha = \kappa_0$ for all $\alpha$, since $\sum_{\alpha} \varphi(\mathbf x)\equiv 1$, reduces to
$$
\left( \rho \nabla \phi \right)_h(\mathbf x) =  -  \frac{\bar p}{\bar \rho} e^{\kappa_0/\bar p}  \sum_{\beta}        \nabla \varphi_\beta(\mathbf x)   e^{-\bar \rho \phi_\beta/\bar p}.
$$
For a hydrostatic equilibrium \eqref{eq:isothermal_equi},  we have now, using \eqref{eq:rho_from_kappa}, that
$$
\nabla p_h (\mathbf x)=  \sum_\beta\nabla\varphi_\beta(\mathbf x) p_\beta =   \sum_\beta\nabla\varphi_\beta(\mathbf x)   \bar p \dfrac{\rho_\beta}{\bar \rho} 
=    \sum_\beta\nabla\varphi_\beta (\mathbf x)   \frac{\bar p}{\bar \rho}   e^{\kappa_0/\bar p}  e^{-\bar\rho \phi_\beta/\bar p} =- \left( \rho \nabla \phi \right)_h(\mathbf x).
$$
One can easily check that the last  expression is identical to \eqref{eq:wb} thus leading to 
$\Psi_h+\mathcal{S}_h=0$ and exact balancing for this particular case.

Concerning the choice of the   reference state,  one can a priori use a local one since the ratio $\bar p/\bar \rho$ is constant at equilibrium. We
have used in practice elemental reference values when integrating the source, and in particular $\bar p^{E_{ij}} = p^{ij}_{0,0}$ and $\bar \rho^{E_{ij}} = \rho^{ij}_{0,0}$ (values of the first local degree of freedom).
This modification will be denoted in the following section as GFQ (WB) and we will compare some simulations without this modification denoting it with GFQ (non-WB).

\section{Numerical results}\label{sec:results}
We consider a quite large suite of numerical benchmarks involving solutions out of equilibrium, stationary states, and complex 
instabilities developing close to stationary equilibria. In most of the problems considered, the compatibility of the numerical method with stationary states,
and possibly low Mach number flows, is very relevant. Comparisons are performed between the standard and the stationarity preserving variant of SUPG.
Where relevant   we also include comparisons with the HLLC \cite{Toro1994,Batten1997} method with MUSCL-type reconstruction \cite{vanLeer1979} to include a method which
has a  more standard, Laplacian type, form of the numerical  dissipation. 

\subsection{Moving isentropic vortex}

To test the order of accuracy of the method, we consider the exact solution of a moving \textit{isentropic vortex} proposed by Shu et al.~\cite{ShuVortex1998}.
We use a standard set up, on the square  domain  $[0,10]^2$, we consider a velocity field defined as 
\begin{equation}\label{eq:vel_vortex}
	\begin{split}
		u(x,y,t) &= u_\infty - \frac{\varepsilon}{2\pi}\,\exp\Bigl(\frac{1}{2}(1 - r^2)\Bigr)\,(y - y_0(t)), \\
		v(x,y,t) &= v_\infty + \frac{\varepsilon}{2\pi}\,\exp\Bigl(\frac{1}{2}(1 - r^2)\Bigr)\,(x - x_0(t)),
	\end{split}
\end{equation}
with background velocity $(u_\infty, v_\infty) = (1.0,\,1.0)$, and $(x_0,y_0)$ the vortex center initially set  at $(5,5)$.
The pressure and density are set to  $\rho(r) = T(r)^{1/(\gamma - 1)}$ and  $p(r) = T(r)^{\gamma / (\gamma - 1)}$, with $T$  the temperature 
\begin{equation}\label{eq:temp_vortex}
	T(r) = 1 - \frac{(\gamma - 1)\,\varepsilon^2}{8\,\gamma\,\pi^2}\,\exp(1 - r^2),
\end{equation}
with    adiabatic constant  $\gamma=1.4$, and  with $r$ the distance from the vortex center.
With this initialization, the vortex  translates at constant speed. We run simulations until  the final time $t_{\mathrm{end}}=2$s.

\begin{figure}[H]
	\centering
	\begin{minipage}[c]{0.49\textwidth}
		\centering
		\includegraphics[trim=5cm 0 0.75cm 0,clip,width=\textwidth]{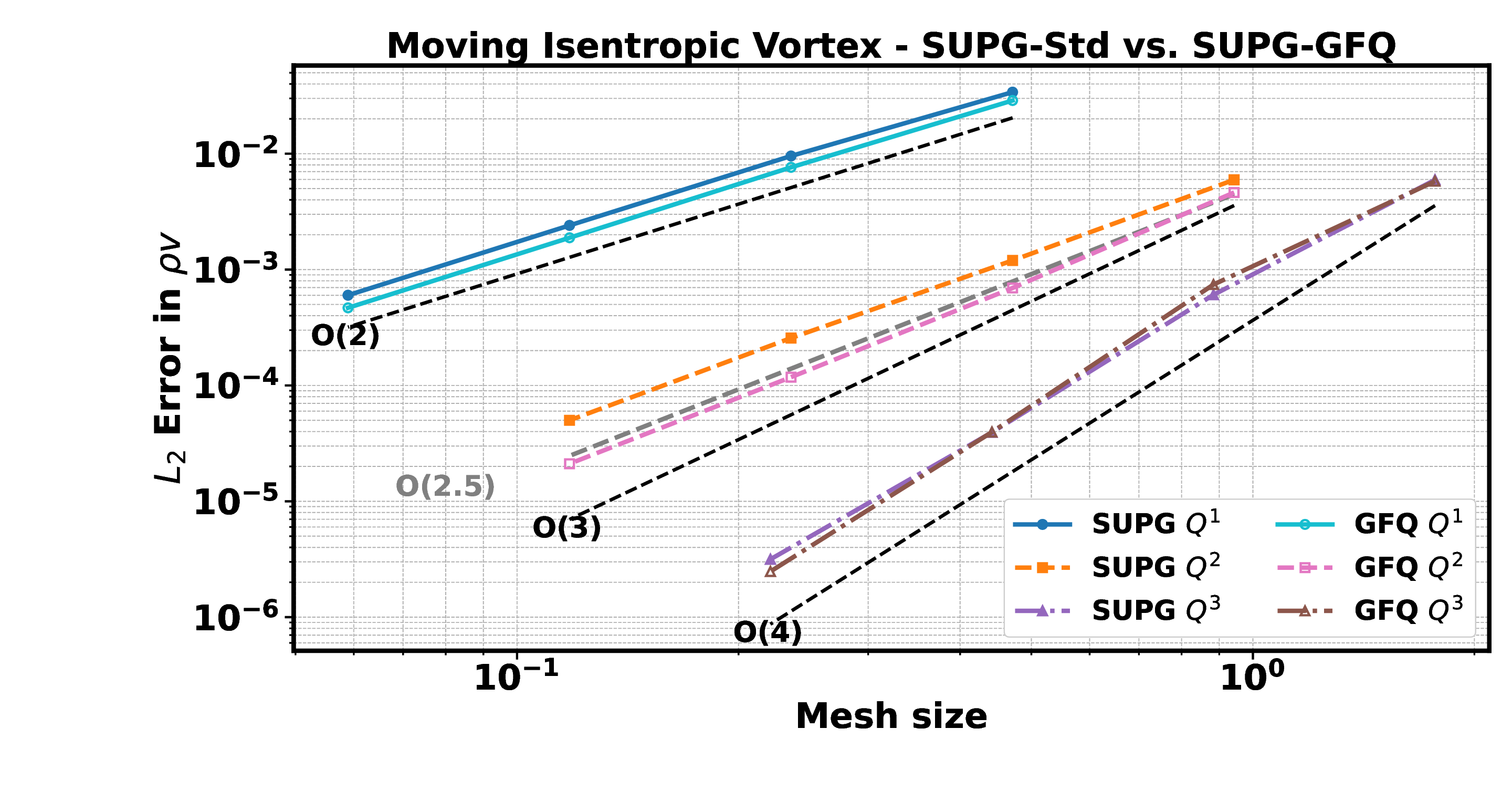} 
	\end{minipage}
	\hfill
	\begin{minipage}[c]{0.49\textwidth}
		\centering
		\includegraphics[trim=5cm 0 0.75cm 0,clip,width=\textwidth]{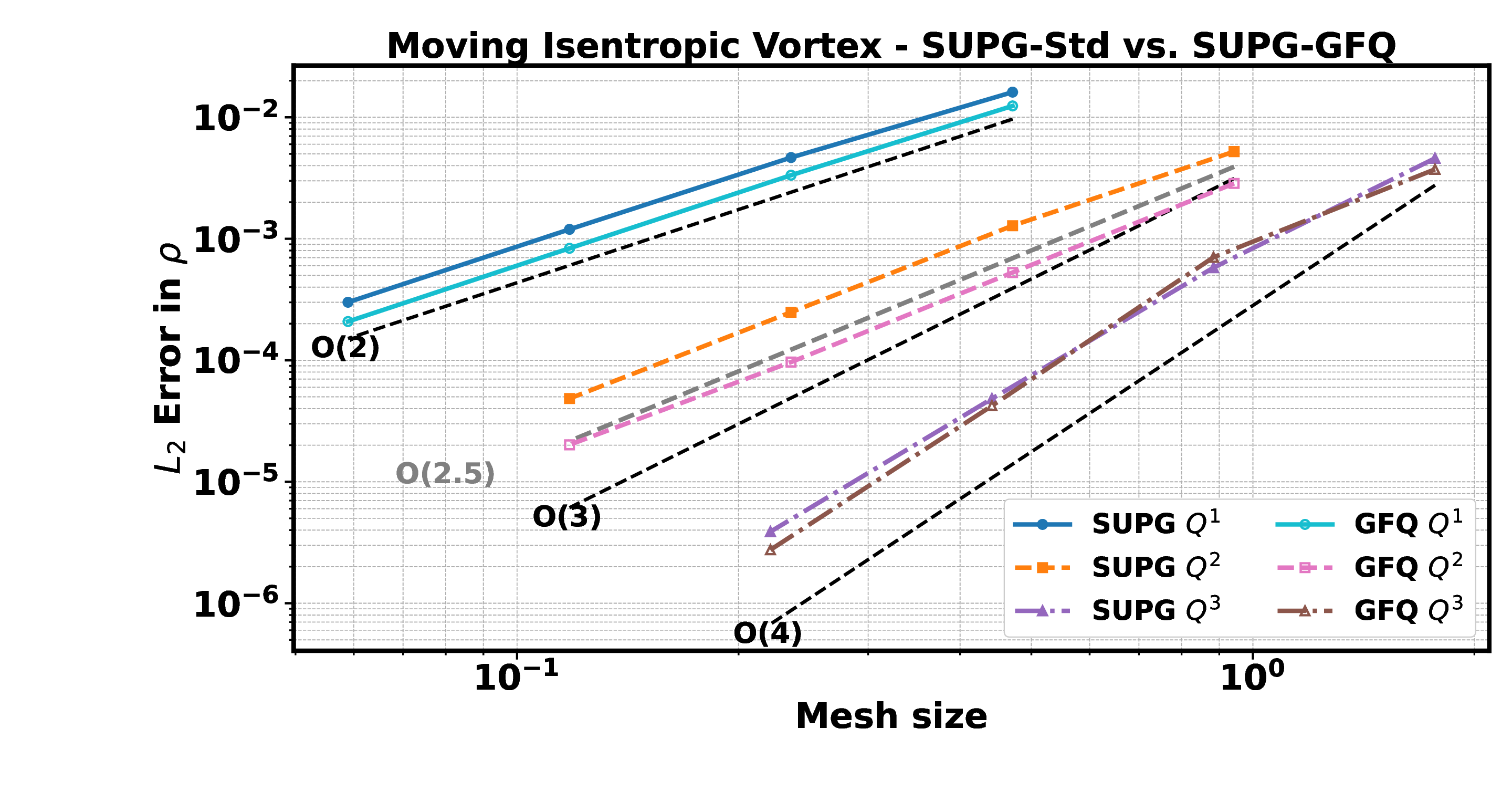} 
	\end{minipage}
	\caption{Euler equations: $L^2$ error on  momentum and density  for the moving isentropic vortex. Comparison between the SUPG-Std and SUPG-GFQ methods. Left: $\rho v$. Right: $\rho $.}
	\label{fig:movingvortexeuler}
\end{figure}

In Figure~\ref{fig:movingvortexeuler} we report the error convergence  of the standard SUPG (SUPG-Std) and the SUPG with global flux quadrature (SUPG-GFQ) for the moving isentropic vortex test case. 
The results are presented for $\mathbb{Q}^1$, $\mathbb{Q}^2$, and $\mathbb{Q}^3$ elements. Both methods show convergence towards the exact solution as the mesh is refined with the expected order of accuracy $K+1$. For $\mathbb Q^2$ elements there is a slight decrease of the order of accuracy as already observed in many other works, inter alia \cite{michel2023spectral,barsukow2025structure}. 
What is worth noticing is that the SUPG-GFQ method shows some   accuracy improvements (in error magnitude) compared to the standard SUPG already for  this moving
vortex case which is not a stationary one. In particular,  the errors of the SUPG-GFQ method are consistently smaller by a factor between 1.3 and 2. 

\subsection{Steady isentropic vortex}
We consider the previous isentropic vortex \eqref{eq:vel_vortex}-\eqref{eq:temp_vortex}-but with zero background velocity,  
providing a stationary solution of the Euler equations. We use this configuration to test the ability of the method to preserve steady states,
and check the super-convergence property. We run the initial value problem starting from the exact stationary state, until the 
final time $t_{\mathrm{end}}=1$s.

\begin{figure}
	\begin{minipage}{0.845\textwidth}
		\centering
		\begin{subfigure}{0.328\textwidth}
			\centering
			\includegraphics[width=\textwidth]{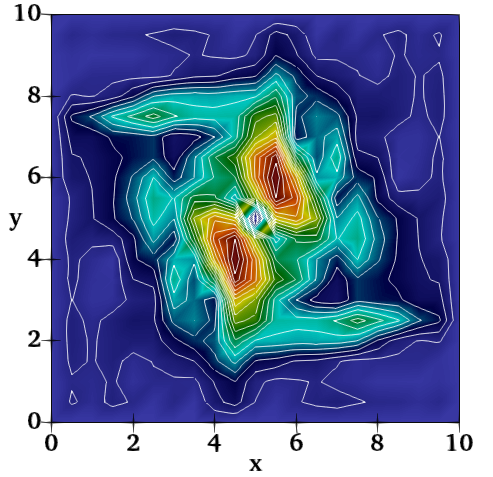}
			\caption{SUPG-Std}
			\label{fig:st_vort_SUPG_15}
		\end{subfigure}
		\hfill
		\begin{subfigure}{0.328\textwidth}
			\centering
			\includegraphics[width=\textwidth]{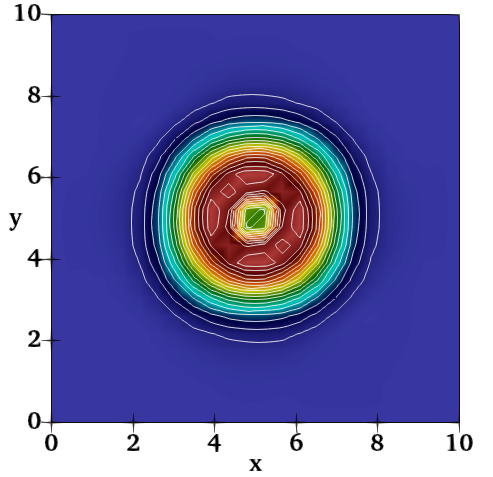}
			\caption{SUPG-GFQ}
			\label{fig:st_vort_GF_15}
		\end{subfigure}
		\hfill
		\begin{subfigure}{0.328\textwidth}
			\centering
			\includegraphics[width=\textwidth]{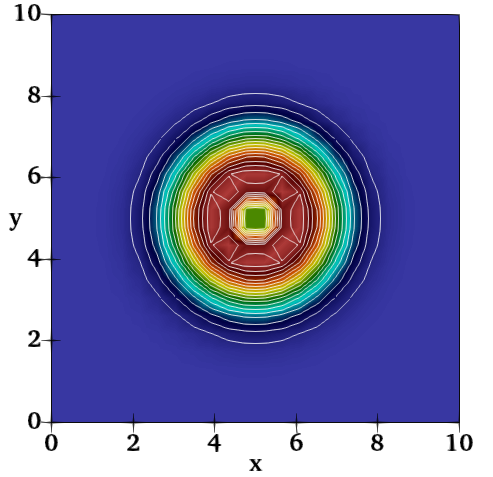}
			\caption{Exact}
			\label{fig:st_vort_exact_15}
		\end{subfigure}
		\centering
		
		\begin{subfigure}{0.328\textwidth}
			\centering
			\includegraphics[width=\textwidth]{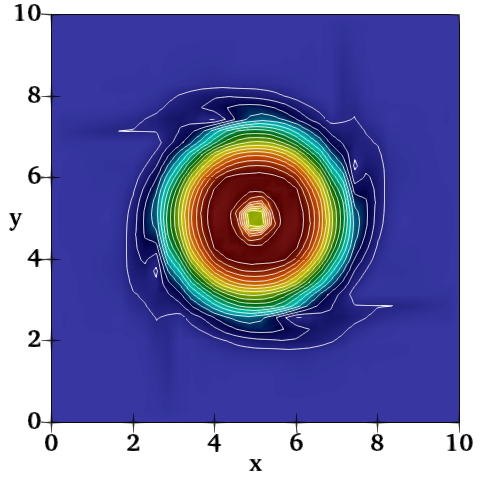}
			\caption{SUPG-Std}
			\label{fig:st_vort_SUPG_35}
		\end{subfigure}
		\hfill
		\begin{subfigure}{0.328\textwidth}
			\centering
			\includegraphics[width=\textwidth]{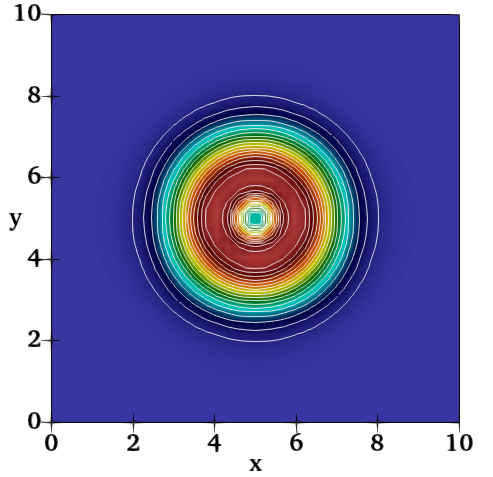}
			\caption{SUPG-GFQ}
			\label{fig:st_vort_GF_35}
		\end{subfigure}
		\hfill
		\begin{subfigure}{0.328\textwidth}
			\centering
			\includegraphics[width=\textwidth]{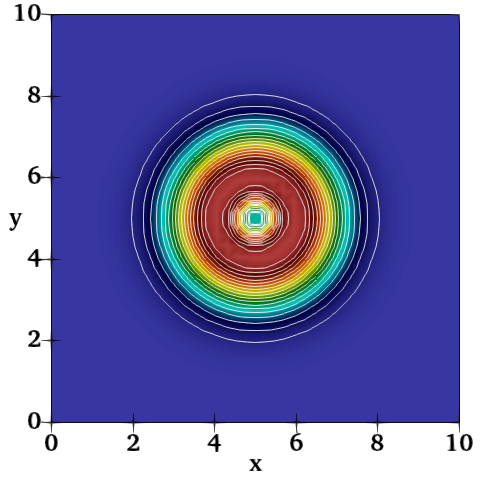}
			\caption{Exact}
			\label{fig:st_vort_exact_35}
		\end{subfigure}
		
	\end{minipage}
	\hfill
	\begin{minipage}{0.145\textwidth}
		\centering
		\includegraphics[width=\textwidth]{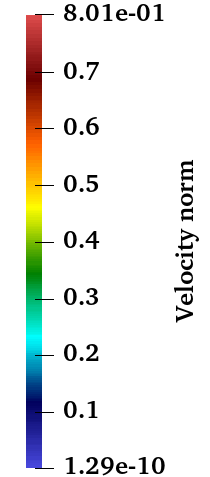}
	\end{minipage}
	
	\caption{
		Isentropic steady vortex with Mach = $0.7$ for the Euler equations at \( t_{\mathrm{end}} = 50\,\mathrm{s} \). 
		\textbf{Top row:} results on a \( 15 \times 15 \) mesh; 
		\textbf{bottom row:} results on a \( 35 \times 35 \) mesh, both with $\mathbb{Q}^1$ elements. 
		\textbf{Left:} standard SUPG. \textbf{Middle:} SUPG–GFQ. \textbf{Right:} exact solution.
	}\label{fig:steady_vortex}
\end{figure}

First, in Figure~\ref{fig:steady_vortex} we plot the velocity norm of the solution of the SUPG-Std and SUPG-GFQ methods for the steady isentropic vortex problem at different mesh resolutions for $\mathbb Q^1$ elements. It is evident that the solutions of the two schemes are strongly different. The SUPG-Std method shows a significant deterioration of the vortex structure, and even at the finer level, there are spurious structures showing up during the simulation. On the other hand, the SUPG-GFQ method is able to preserve the vortex structure even on the coarser mesh in close agreement with the exact solution.

 \begin{table}[htbp]
	\small
	\centering
	\caption{Convergence of the \textbf{steady isentropic vortex} at $t_{\mathrm{end}} = 1\,\mathrm{s}$.
		$L^2$ errors and experimental orders of accuracy (EOA) for SUPG-Std and SUPG-GFQ
		using $\mathbb{Q}^1$, $\mathbb{Q}^2$, and $\mathbb{Q}^3$ elements.}
	\label{tab:st_vortex_conv}
	\begin{tabular}{c |c |c|cc|cc|cc|cc}
		\hline\\[3pt]
		Element & Method & $N_x=N_y$
		& \multicolumn{2}{c|}{$\rho$}
		& \multicolumn{2}{c|}{$\rho u$}
		& \multicolumn{2}{c|}{$\rho v$}
		& \multicolumn{2}{c}{$\rho E$} \\[3pt]
		\hline\\[3pt]
		& & 
		& $L^2$ error & EOA
		& $L^2$ error & EOA
		& $L^2$ error & EOA
		& $L^2$ error & EOA \\[5pt]
		\hline\\[3pt]
		
		\multirow{8}{*}{$\mathbb{Q}^1$}
		& \multirow{4}{*}{SUPG-Std}
		&   30  & 2.49E-03 & --   & 3.30E-02 & --   & 3.27E-02 & --   & 2.92E-03 & --  \\[5pt]
		& & 60  & 7.09E-04 & 1.81 & 8.19E-03 & 2.01 & 8.11E-03 & 2.01 & 7.44E-04 & 1.97 \\[5pt]
		& & 120 & 1.80E-04 & 1.98 & 2.03E-03 & 2.01 & 2.02E-03 & 2.01 & 1.84E-04 & 2.02 \\[5pt]
		& & 240 & 4.50E-05 & 2.00 & 5.04E-04 & 2.01 & 5.03E-04 & 2.00 & 4.56E-05 & 2.01 \\[5pt]
		\cline{2-11}\\[3pt]
		& \multirow{4}{*}{SUPG-GFQ}
		&   30  & 1.06E-03 & --   & 1.47E-02 & --   & 1.47E-02 & --   & 1.26E-03 & --  \\[5pt]
		& & 60  & 3.45E-04 & 1.62 & 3.93E-03 & 1.90 & 3.93E-03 & 1.90 & 3.57E-04 & 1.82 \\[5pt]
		& & 120 & 8.95E-05 & 1.95 & 9.98E-04 & 1.98 & 9.98E-04 & 1.98 & 9.04E-05 & 1.98 \\[5pt]
		& & 240 & 2.25E-05 & 1.99 & 2.51E-04 & 1.99 & 2.51E-04 & 1.99 & 2.26E-05 & 2.00 \\[5pt]
		\hline\\[3pt]
		
		\multirow{8}{*}{$\mathbb{Q}^2$}
		& \multirow{4}{*}{SUPG-Std}
		& 15     & 1.69E-03 & --   & 2.58E-02 & --   & 2.57E-02 & --   & 1.82E-03 & --   \\[5pt]
		& & 30   & 2.89E-04 & 2.55 & 5.90E-03 & 2.13 & 5.82E-03 & 2.14 & 3.23E-04 & 2.49 \\[5pt]
		& & 60   & 4.62E-05 & 2.65 & 1.14E-03 & 2.37 & 1.13E-03 & 2.37 & 5.91E-05 & 2.45\\[5pt]
		& & 120  & 6.70E-06 & 2.79 & 2.06E-04 & 2.47 & 2.05E-04 & 2.46 & 1.01E-05 & 2.55\\[5pt]
		\cline{2-11}\\[3pt]
		& \multirow{4}{*}{SUPG-GFQ}
		&   15   & 8.51E-04 & --   & 7.36E-03 & --   & 7.39E-03 & --   & 6.56E-04 & --   \\[5pt]
		& & 30   & 5.00E-05 & 4.09 & 4.68E-04 & 3.97 & 4.69E-04 & 3.98 & 3.96E-05 & 4.05 \\[5pt]
		& & 60   & 3.11E-06 & 4.01 & 2.86E-05 & 4.03 & 2.87E-05 & 4.03 & 2.43E-06 & 4.02 \\[5pt]
		& & 120  & 1.91E-07 & 4.03 & 1.73E-06 & 4.05 & 1.73E-06 & 4.05 & 1.50E-07 & 4.02 \\[5pt]
		\hline\\[3pt]
		
		\multirow{8}{*}{$\mathbb{Q}^3$}
		& \multirow{4}{*}{SUPG-Std}
		&   8   & 3.58E-03 & --   & 5.49E-02 & --   & 5.50E-02 & --   & 4.49E-03 & --   \\[5pt]
		& & 16  & 3.24E-04 & 3.47 & 4.03E-03 & 3.77 & 4.06E-03 & 3.76 & 2.89E-04 & 3.96 \\[5pt]
		& & 32  & 1.39E-05 & 4.54 & 2.39E-04 & 4.08 & 2.35E-04 & 4.11 & 1.15E-05 & 4.65 \\[5pt]
		& & 64  & 6.69E-07 & 4.38 & 1.34E-05 & 4.15 & 1.33E-05 & 4.14 & 6.12E-07 & 4.24 \\[5pt]
		\cline{2-11}\\[3pt]
		& \multirow{4}{*}{SUPG-GFQ}
		&   8   & 1.07E-03 & --   & 7.26E-03 & --   & 7.56E-03 & --   & 1.04E-03 & --   \\[5pt]
		& & 16  & 4.58E-05 & 4.54 & 4.97E-04 & 3.87 & 4.97E-04 & 3.93 & 3.70E-05 & 4.81 \\[5pt]
		& & 32  & 1.42E-06 & 5.01 & 2.40E-05 & 4.37 & 2.40E-05 & 4.37 & 1.27E-06 & 4.86 \\[5pt]
		& & 64  & 4.84E-08 & 4.87 & 8.60E-07 & 4.80 & 8.60E-07 & 4.80 & 3.87E-08 & 5.04 \\[5pt]
		\hline
	\end{tabular}
\end{table}

In Table~\ref{tab:st_vortex_conv}, we report the $L^2$ errors and the experimental order of accuracy for the steady isentropic vortex test case. The results show that both methods converge towards the exact solution as the mesh is refined, but the SUPG-Std with the classical order of accuracy $K+1$, while the SUPG-GFQ with an extra order of accuracy for $K=2,3$, where it achieves super-convergence rates. 
This indicates that the SUPG-GFQ method is more effective in preserving the steady vortex structure and provides improved accuracy over the standard SUPG method.
\subsection{Low Mach Behavior: Steady isentropic vortex}
\renewcommand{\arraystretch}{0.93}
\setlength{\tabcolsep}{2pt}

\begin{table}
	\small
	\centering
	\caption{Convergence of the \textbf{steady isentropic vortex} with different Mach = $10^{-2},\,10^{-4},\,10^{-6}$ at $t_{\mathrm{end}} = 50\,\mathrm{s}$. 
		$L^2$ errors of the density and $\rho u$ and experimental order of accuracy (EOA)  
		using $\mathbb{Q}^1$, $\mathbb{Q}^2$, and $\mathbb{Q}^3$ elements for SUPG-Std and SUPG-GFQ, and a 2nd-order FV HLLC scheme.}\label{tab:low_mach_conv}
	\begin{tabular}{c|c|cc|cc|cc|cc|cc|cc}
		\hline
		{Mach}   & & \multicolumn{4}{c|}{Mach = $10^{-2}$} & \multicolumn{4}{c|}{Mach = $10^{-4}$} & \multicolumn{4}{c}{Mach = $10^{-6}$} \\ \hline
		{Variable} &  
		& \multicolumn{2}{c|}{$\rho$} 
		& \multicolumn{2}{c|}{$\rho u$} 
		& \multicolumn{2}{c|}{$\rho$} 
		& \multicolumn{2}{c|}{$\rho u$} 
		& \multicolumn{2}{c|}{$\rho$}  
		& \multicolumn{2}{c}{$\rho u$} \\
		\hline
		Element& $N_x$ & $L^2$ error & EOA &$L^2$ error & EOA &$L^2$ error & EOA &$L^2$ error & EOA & $L^2$ error & EOA & $L^2$ error & EOA \\
		\hline
		
		\multicolumn{14}{c}{\textbf{FV HLLC}} \\
		\hline
		2nd or.   	& 20    & 2.65E-05 & --   & 5.49E-01 & --   & 2.46E-07 & --   & 5.55E-01 & --   & 2.46E-09 & --   & 5.55E-01 & --   \\
		& 40    & 9.65E-06 & 1.46 & 2.17E-01 & 1.34 & 8.92E-08 & 1.46 & 2.30E-01 & 1.27 & 8.93E-10 & 1.46 & 2.30E-01 & 1.27 \\
		& 80    & 2.31E-06 & 2.06 & 4.62E-02 & 2.23 & 2.45E-08 & 1.86 & 5.69E-02 & 2.01 & 2.47E-10 & 1.86 & 5.71E-02 & 2.01 \\
		& 160 	& 5.30E-07 & 2.12 & 6.06E-03 & 2.93 & 4.96E-09 & 2.31 & 1.15E-02 & 2.31 & 5.03E-11 & 2.29 & 1.16E-02 & 2.30 \\ \hline
		
		\multicolumn{14}{c}{\textbf{SUPG-Std}} \\
		\hline
		$\mathbb{Q}^1$ & 20 & 5.51E-06 & --   & 1.49E-01 & --   & 4.34E-08 & --   & 1.46E-01 & --   & 4.34E-10 & --   & 1.46E-01 & --    \\
		& 40 & 9.06E-07 & 2.61 & 3.00E-02 & 2.31 & 6.66E-09 & 2.70 & 2.55E-02 & 2.52 & 6.66E-11 & 2.70 & 2.55E-02 & 2.52  \\
		& 80 & 1.34E-07 & 2.75 & 5.69E-03 & 2.40 & 8.62E-10 & 2.95 & 3.54E-03 & 2.85 & 8.62E-12 & 2.95 & 3.54E-03 & 2.85  \\
		& 160& 2.17E-08 & 2.63 & 1.23E-03 & 2.21 & 1.08E-10 & 3.00 & 5.09E-04 & 2.80 & 1.08E-12 & 3.00 & 5.09E-04 & 2.80  \\
		\hline
		$\mathbb{Q}^2$ & 10 & 5.36E-06 & --   & 9.85E-02 & --   & 4.70E-08  & --   & 9.73E-02 & --   & 4.58E-10 & --   & 9.45E-02 & --   \\
		& 20 & 1.26E-06 & 2.09 & 2.34E-02 & 2.07 & 1.16E-08  & 2.01 & 2.26E-02 & 2.11 & 1.18E-10 & 1.96 & 2.24E-02 & 2.08 \\
		& 40 & 2.94E-07 & 2.10 & 5.52E-03 & 2.08 & 2.79E-09  & 2.06 & 5.26E-03 & 2.10 & 2.90E-11 & 2.02 & 5.31E-03 & 2.08 \\
		& 80 & 6.17E-08 & 2.25 & 1.30E-03 & 2.09 & 5.66E-10  & 2.30 & 1.18E-03 & 2.16 & 6.13E-12 & 2.24 & 1.20E-03 & 2.14 \\
		\hline
		$\mathbb{Q}^3$ & 6  & 3.36E-06 & --   & 2.84E-02 & --   & 4.97E-09 & --   & 1.18E-02 & --   & 5.07E-11 & --   & 1.16E-02 & --   \\
		& 12 & 9.75E-07 & 1.78 & 8.15E-03 & 1.80 & 3.58E-09 & 0.48 & 7.82E-03 & 0.59 & 3.68E-11 & 0.46 & 7.74E-03 & 0.58 \\
		& 24 & 7.69E-08 & 3.67 & 9.17E-04 & 3.15 & 2.99E-10 & 3.58 & 8.69E-04 & 3.17 & 2.94E-12 & 3.65 & 8.37E-04 & 3.21 \\
		& 48 & 5.25E-09 & 3.87 & 5.80E-05 & 3.99 & 2.59E-11 & 3.53 & 7.90E-05 & 3.45 & 2.53E-13 & 3.54 & 7.40E-05 & 3.50 \\
		\hline
		\multicolumn{14}{c}{\textbf{SUPG-GFQ}} \\
		\hline
		$\mathbb{Q}^1$ & 20 & 3.01E-07 & --   & 2.65E-02 & --   & 1.72E-10 & --   & 9.43E-03 & --   & 1.71E-12 & --     & 9.43E-03 & --   \\
		& 40 & 8.78E-08 & 1.77 & 8.49E-03 & 1.64 & 4.69E-12 & 5.20 & 2.29E-03 & 2.04 & 4.79E-15 & 8.48   & 2.29E-03 & 2.04 \\
		& 80 & 2.27E-08 & 1.95 & 2.24E-03 & 1.92 & 1.16E-12 & 2.01 & 5.68E-04 & 2.01 & 1.11E-14 & -1.21  & 5.68E-04 & 2.01 \\
		& 160& 5.73E-09 & 1.99 & 5.67E-04 & 1.98 & 3.88E-13 & 1.59 & 1.42E-04 & 2.00 & 3.11E-14 & -1.49  & 1.42E-04 & 2.00 \\
		\hline
		$\mathbb{Q}^2$ & 10 & 4.99E-07  & --   & 3.56E-02 & --   & 4.82E-11 & --    & 3.77E-03 & --   & 2.06E-14 & --    & 3.75E-03 & --    \\
		& 20 & 4.99E-08  & 3.32 & 2.84E-03 & 3.65 & 1.94E-12 & 4.63  & 2.59E-04 & 3.86 & 7.22E-14 & -1.81 & 2.57E-04 & 3.86  \\
		& 40 & 3.17E-09  & 3.98 & 1.96E-04 & 3.86 & 4.47E-13 & 2.12  & 1.60E-05 & 3.99 & 2.92E-13 & -2.02 & 1.60E-05 & 3.99  \\
		& 80 & 1.93E-10  & 4.04 & 1.24E-05 & 3.98 & 1.99E-12 & -2.16 & 1.00E-06 & 3.99 & 1.11E-12 & -1.92 & 6.00E-06 & 1.37  \\
		\hline
		$\mathbb{Q}^3$ & 6  & 6.54E-07  & --   & 4.04E-02 & --   & 4.80E-11 & --    & 4.32E-03 & --    & 5.68E-14 & --    & 4.30E-03 & --    \\
		& 12 & 4.53E-08  & 3.85 & 6.84E-04 & 5.88 & 2.09E-12 & 4.52  & 1.42E-04 & 4.93  & 2.32E-13 & -2.03 & 1.41E-04 & 4.93  \\
		& 24 & 1.93E-09  & 4.55 & 1.03E-04 & 2.74 & 1.08E-12 & 0.95  & 5.02E-06 & 4.82  & 8.74E-13 & -1.91 & 5.02E-06 & 4.77  \\
		& 48 & 6.12E-11  & 4.98 & 4.00E-06 & 4.76 & 4.38E-12 & -2.02 & 2.42E-07 & 4.38  & 3.33E-12 & -1.93 & 8.88E-06 & -0.78 \\
		\hline
	\end{tabular}
	
\end{table}
We investigate the  behavior of the proposed SUPG-GFQ method for the steady isentropic vortex test case at low Mach numbers. 
The maximum Mach is related to the vortex strength $\varepsilon$ as 
$$
\varepsilon = 2\pi M \sqrt{\frac{\gamma}{1 + \frac{1}{2}(\gamma - 1)M^2}}$$
\noindent
The previous value  of $\varepsilon=5$ corresponds to a maximum Mach of about 0.7.  
%
%
%
We consider now long time runs, and compare the errors for   $t_{\mathrm{end}}=50$s.
In Table~\ref{tab:low_mach_conv}, we present the  $L^2$ errors and experimental orders of accuracy for the density $\rho$ and the momentum $\rho u$ at Mach numbers of $10^{-2}$, $10^{-4}$, and $10^{-6}$.
We compare the results with those obtained using the standard SUPG method, the SUPG-GFQ and a second-order finite volume HLLC scheme. 
Note that  on a Cartesian mesh the  total number of unknowns of the finite volume method is roughly equivalent
to the one of the second order  continuous finite element, as one evolves cell averages, the other nodal values.   
On an $N\times N$ mesh, the  ratio of the number of unknowns of  the two methods  is easily shown to be of order $\sim  1 + \mathcal{O}(1/ N)$.

Observing the errors for second order methods ($\mathbb Q^1$ FEM and FV HLLC), we see that the error of the SUPG-Std is already systematically smaller than the FV HLLC one, of around a factor 4. The SUPG-GFQ is again much better than the SUPG-Std, we can reach errors even 10 times smaller than the SUPG-Std. In particular, for Mach = $10^{-6}$ the density error of the SUPG-GFQ hits machine precision for very coarse meshes.
For $\mathbb Q^1$, we see second order of accuracy for all methods, but for $\mathbb Q^2$ and $\mathbb Q^3$ the SUPG-GFQ shows the theoretical super-convergence rates. Indeed, when significantly higher than machine precision, the SUPG-GFQ shows sharp convergence rates of order $K+2$.
In Figure~\ref{fig:vortexeuler_all}, we report the  relative  $x^2$-momentum errors   to visually show the accuracy enhancements 
of the SUPG-GFQ compared to the SUPG-Std and the FV HLLC, not only in terms of slopes but also error magnitude.  Note that where the  slopes become flat the absolute error
has reached values close to machine accuracy.  This is also true
for  the momentum errors 
in the case Mach = $10^{-6}$, which stall around higher values. Indeed,  
 these are relative errors  divided by the norm of the exact momentum which, 
 in the current initialization, scales with the Mach number.
Thus, the absolute value of the errors   are in this case roughly 6 orders of magnitude  below
the values of the plot.

\begin{figure}
	\centering
	\begin{minipage}{0.49\textwidth}
		\centering
		Second order methods\\
		Mach = $10^{-2}$\\	
		\includegraphics[width=\textwidth, trim={142 50 0 0},clip]{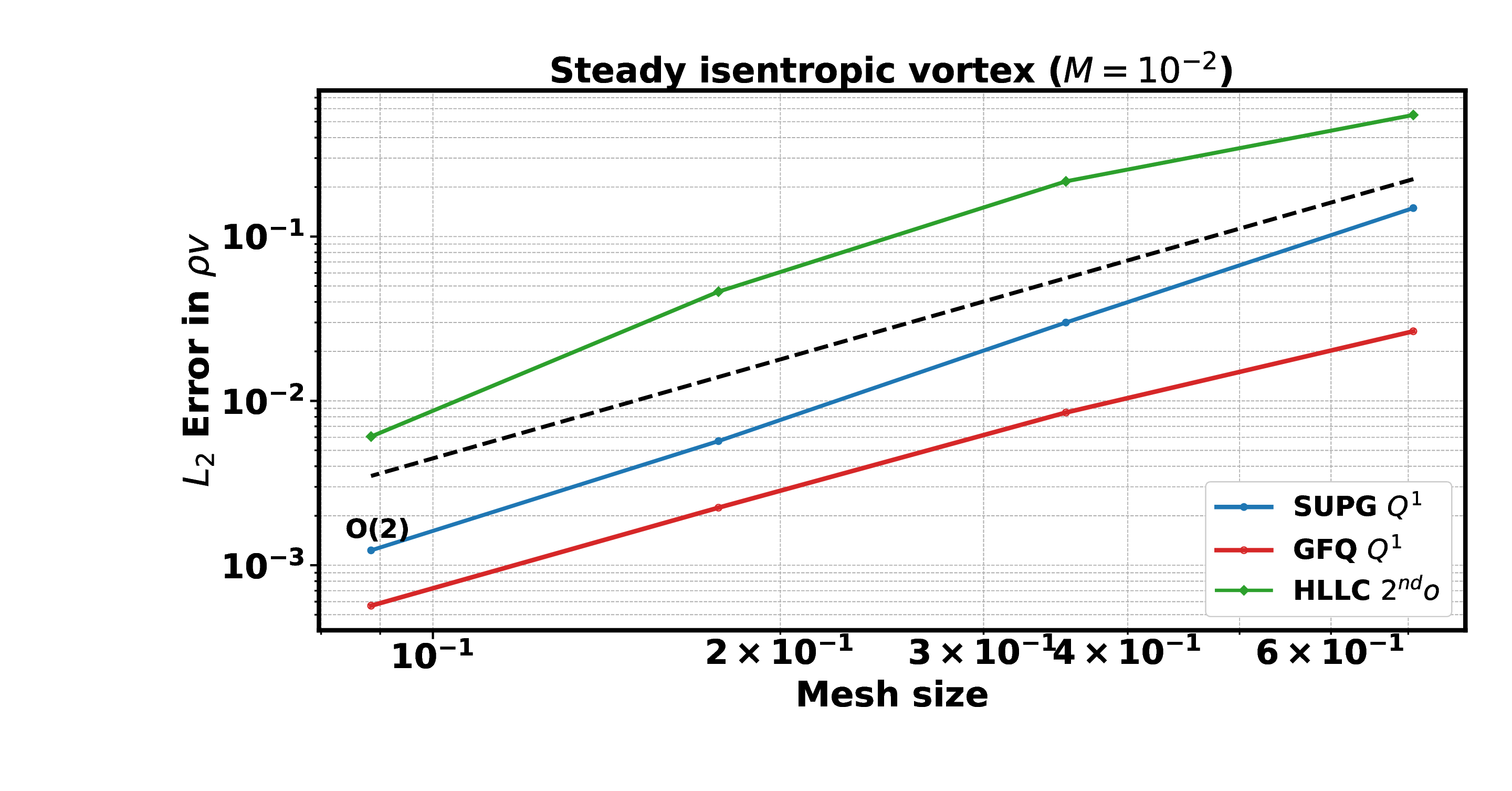}\\
		Mach = $10^{-4}$\\
		\includegraphics[width=\textwidth, trim={142 50 0 0},clip]{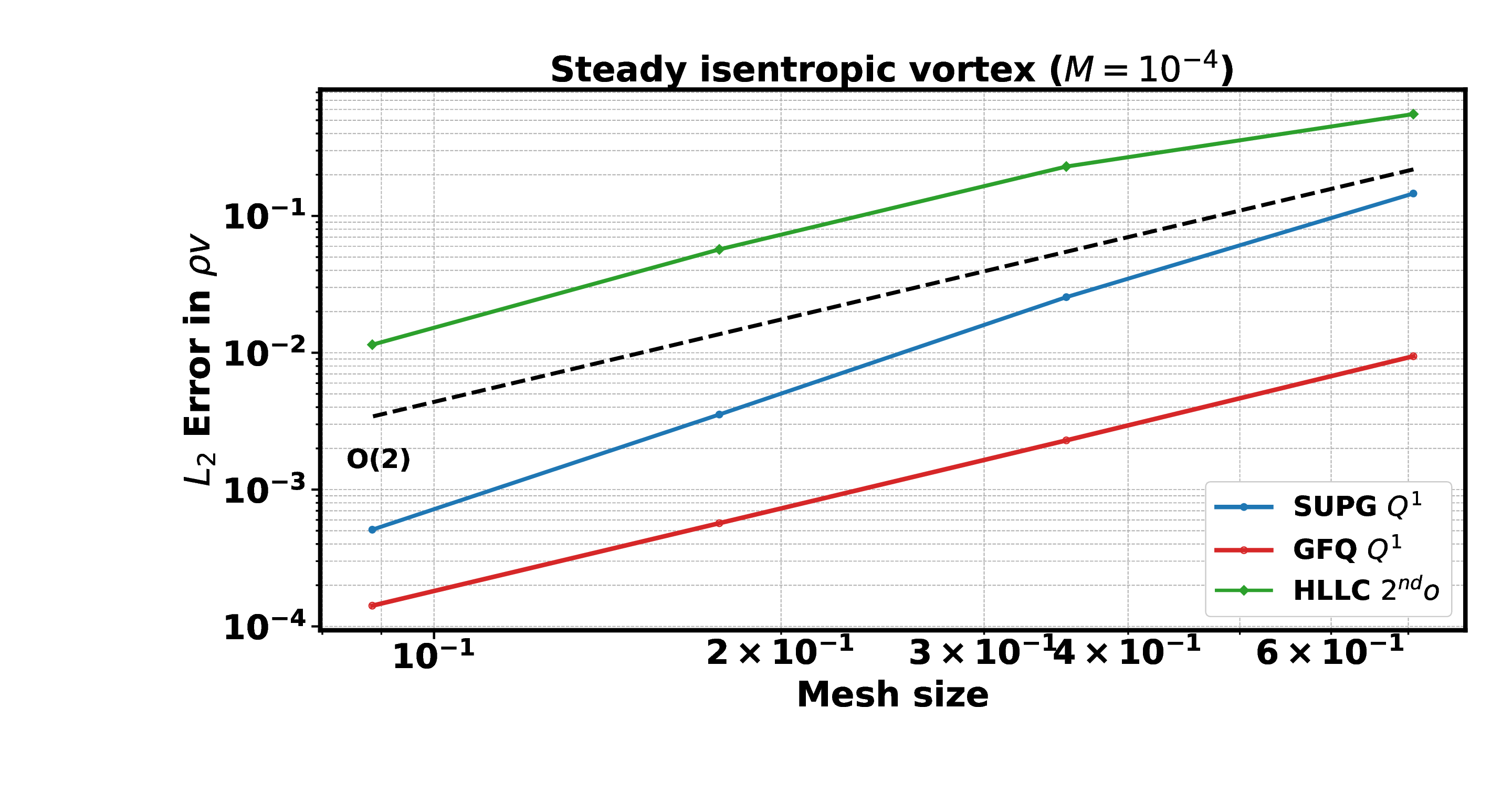}\\
		Mach = $10^{-6}$\\
		\includegraphics[width=\textwidth, trim={142 50 0 0},clip]{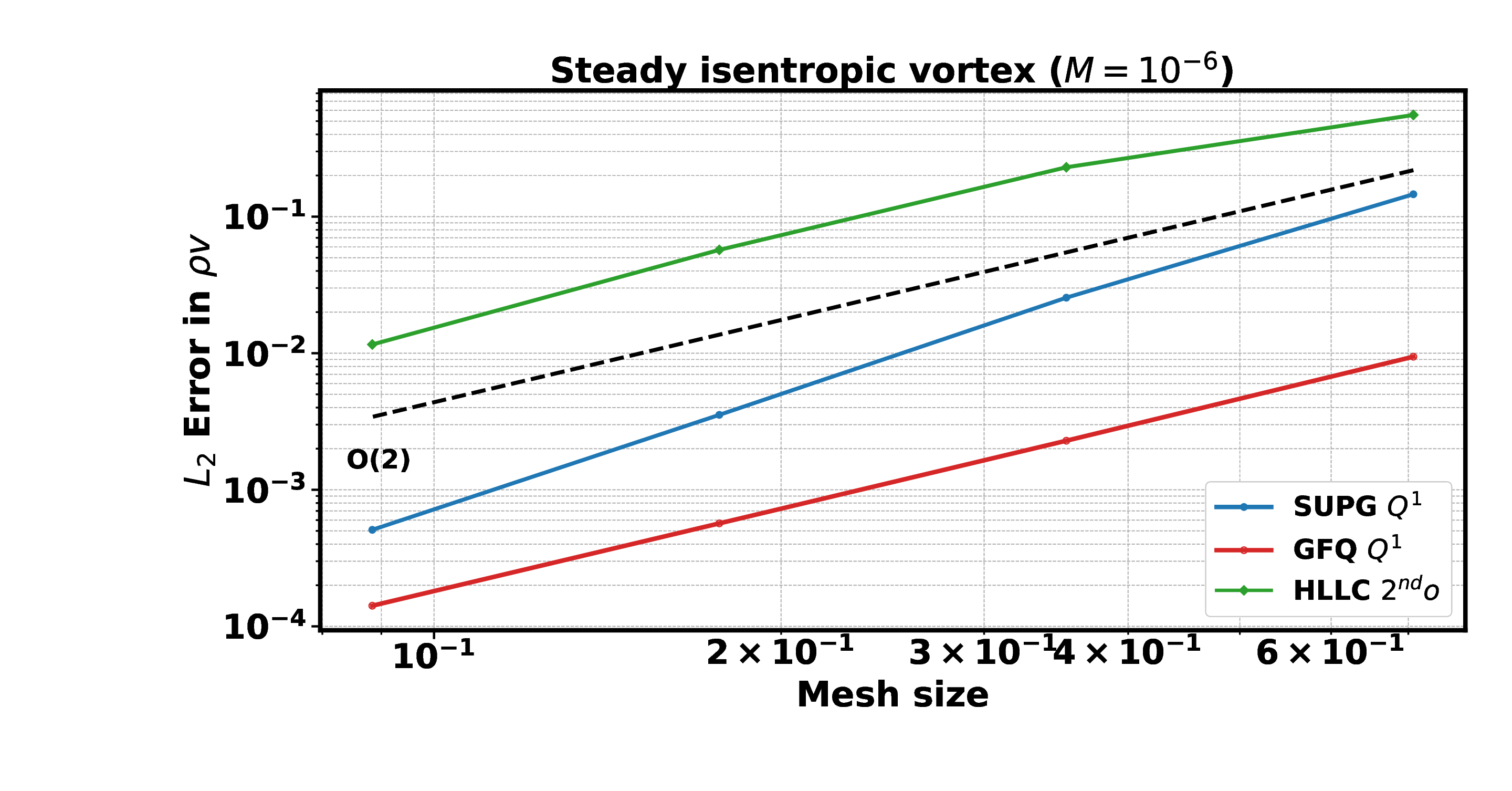}
	\end{minipage}\hfill
	\begin{minipage}{0.49\textwidth}
		\centering
		High-order methods\\
		Mach = $10^{-2}$\\	
		\includegraphics[width=\textwidth, trim={142 50 0 0},clip]{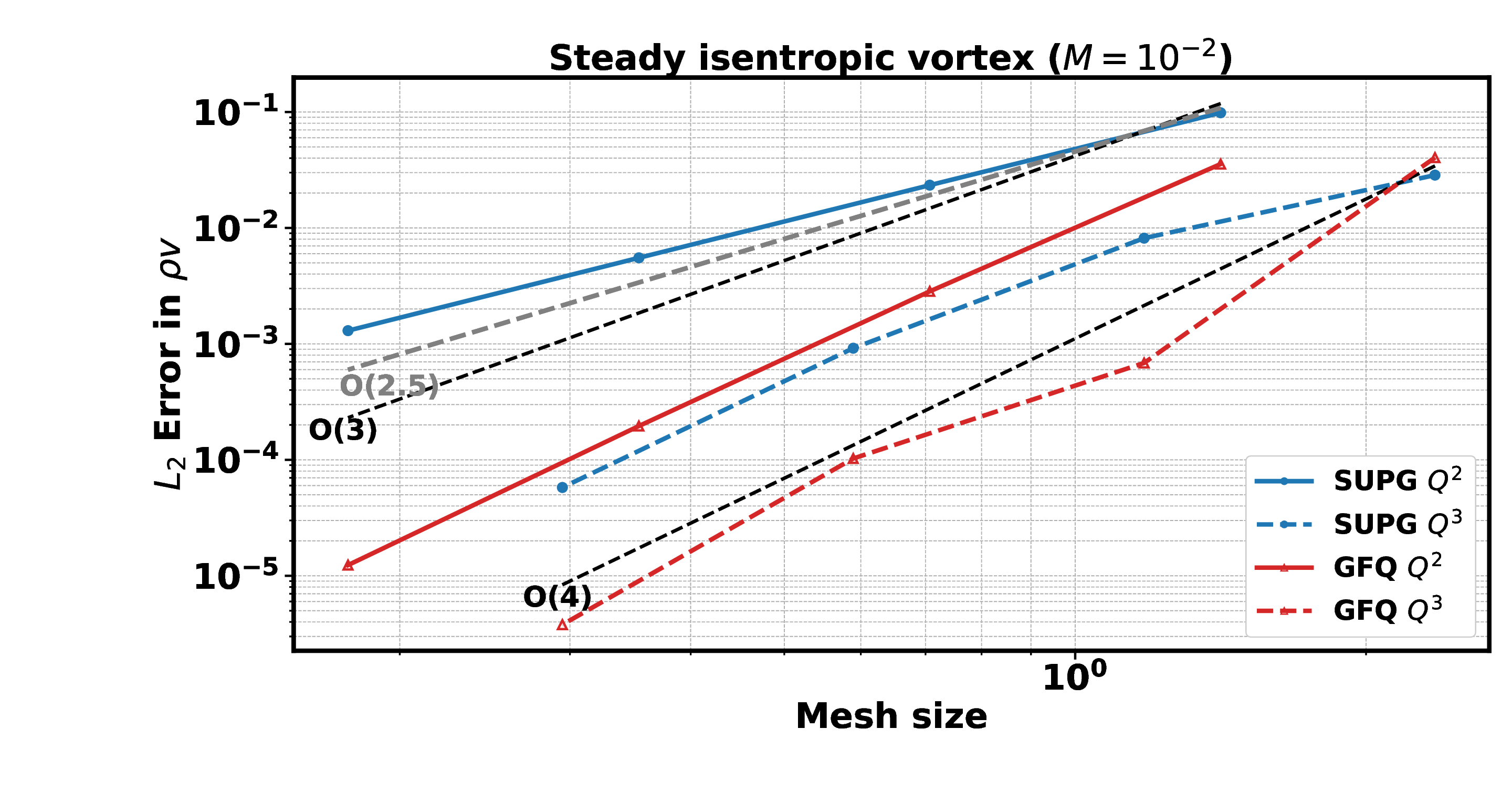}\\
		Mach = $10^{-4}$\\
		\includegraphics[width=\textwidth, trim={142 50 0 0},clip]{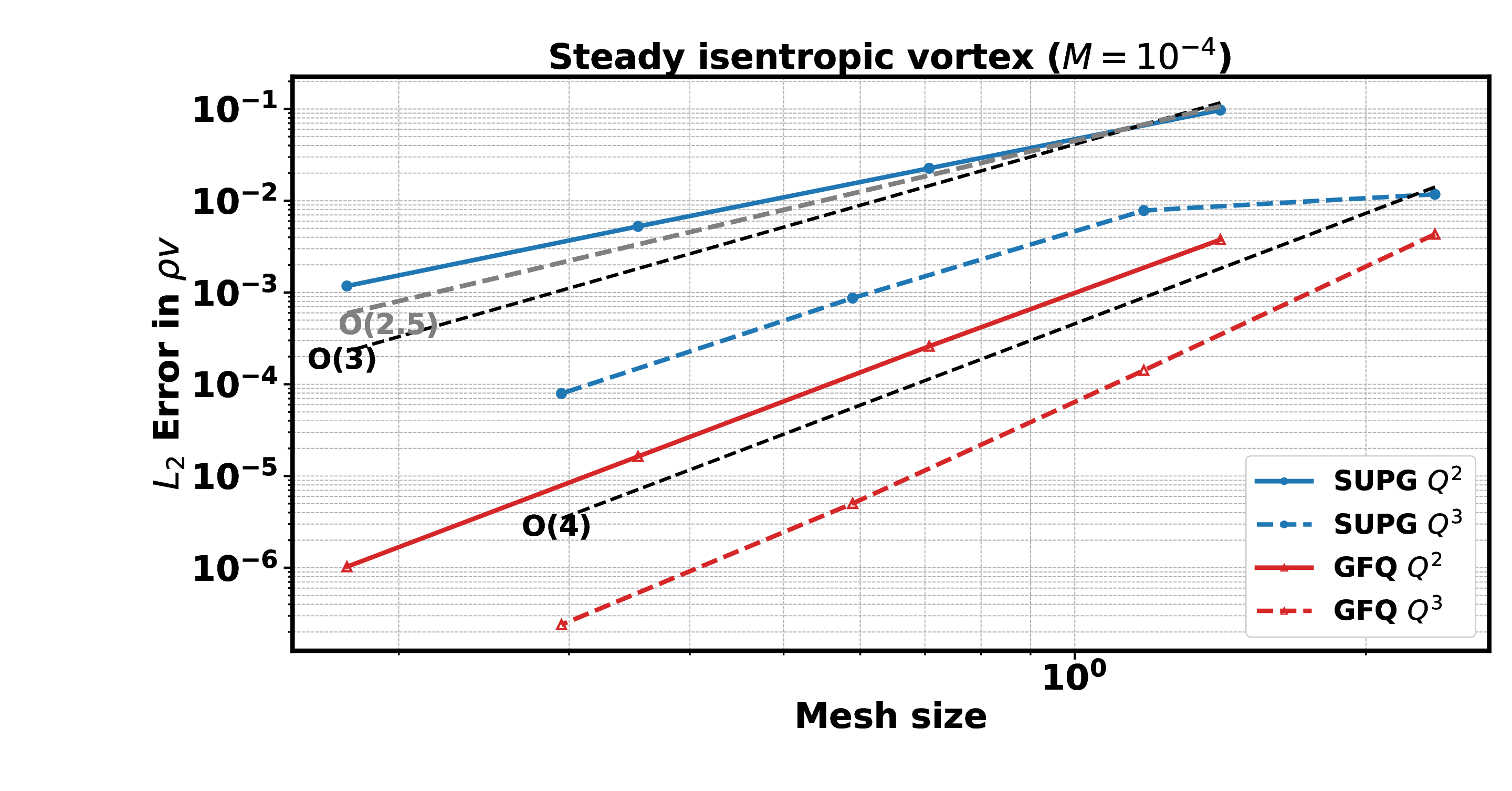}\\
		Mach = $10^{-6}$\\
		\includegraphics[width=\textwidth, trim={142 50 0 0},clip]{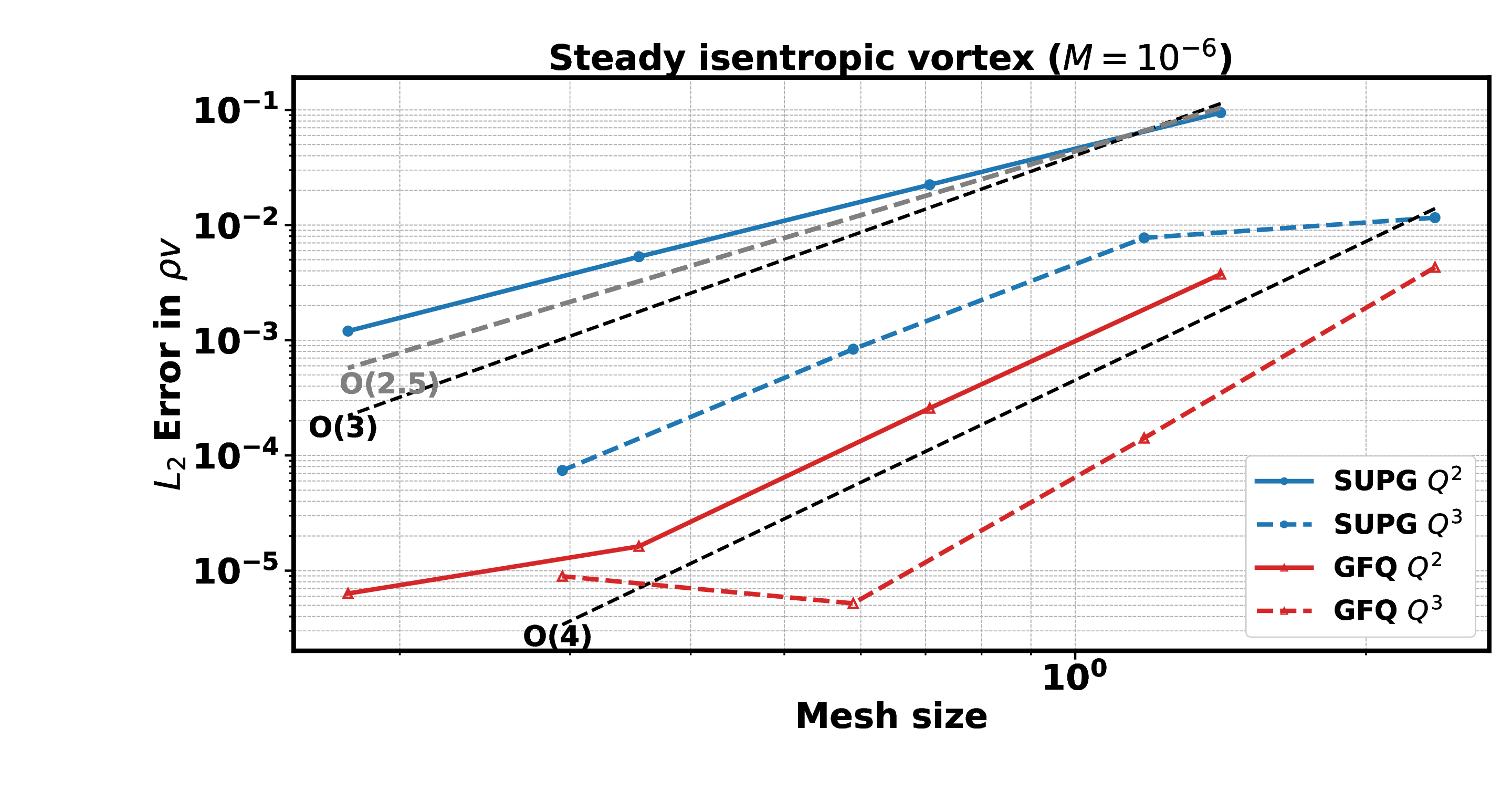}
		
	\end{minipage}

\caption{Convergence of the relative error on  $\rho v$ for $\mathbb Q^1$ SUPG-Std and SUPG-GFQ and for FV-HLLC (left) and $\mathbb Q^2$ and $\mathbb Q^3$ (right) at time 50 for the steady isentropic vortex test case with Mach numbers $M = 10^{-2}$ (\textbf{top}), $M = 10^{-4}$ (\textbf{center}) and $M = 10^{-6}$ (\textbf{bottom}).}
\label{fig:vortexeuler_all}

\end{figure}

		\begin{figure}
			
			\centering
			
			\noindent
			Mach = $10^{-2}$ \\
			\begin{minipage}[t]{0.9\textwidth}
				\centering
				\begin{subfigure}[b]{0.328\textwidth}
					\centering
					\includegraphics[width=\textwidth]{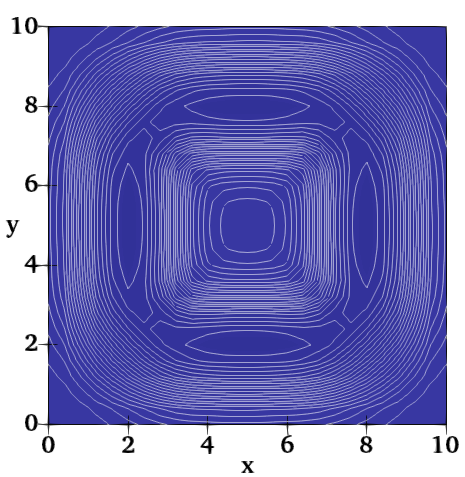}
					\caption{HLLC (2nd)}
				\end{subfigure}
				\hfill
				\begin{subfigure}[b]{0.328\textwidth}
					\centering
					\includegraphics[width=\textwidth]{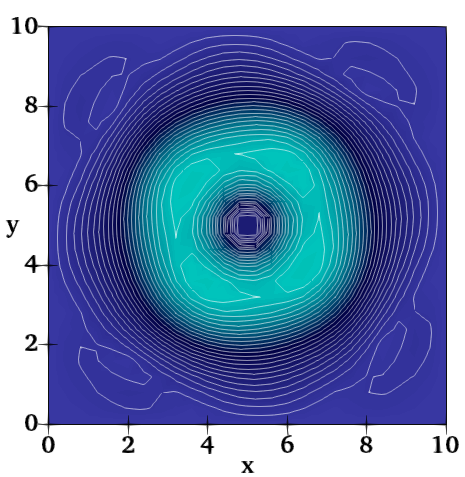}
					\caption{SUPG-Std}
				\end{subfigure}
				\hfill
				\begin{subfigure}[b]{0.328\textwidth}
					\centering
					\includegraphics[width=\textwidth]{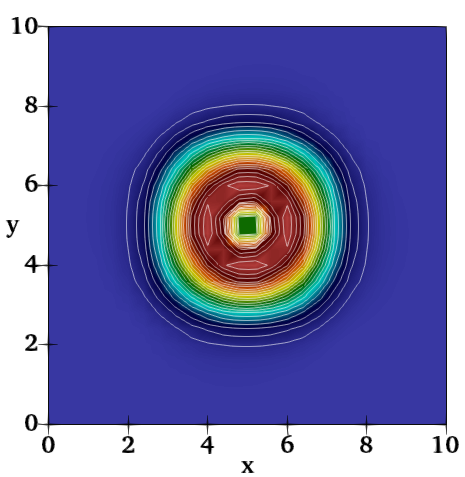}
					\caption{SUPG-GFQ}
				\end{subfigure}
			\end{minipage}
			\hfill
			\begin{minipage}[t]{0.09\textwidth}
				\vspace{-4.87cm}
				\includegraphics[width=\textwidth]{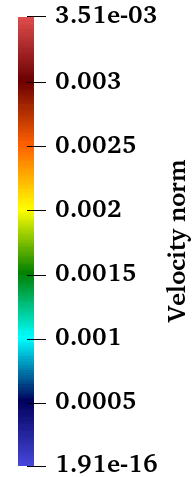}
			\end{minipage}
			
			\vspace{0.2cm}
			
			\noindent
			Mach = $10^{-4}$ \\
			\begin{minipage}[t]{0.9\textwidth}
				\centering
				\begin{subfigure}[b]{0.328\textwidth}
					\centering
					\includegraphics[width=\textwidth]{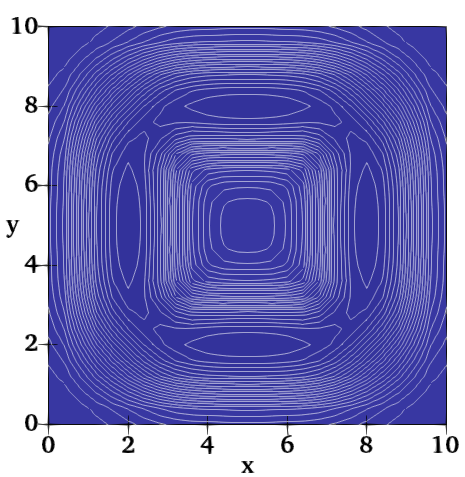}
					\caption{HLLC (2nd)}
				\end{subfigure}
				\hfill
				\begin{subfigure}[b]{0.328\textwidth}
					\centering
					\includegraphics[width=\textwidth]{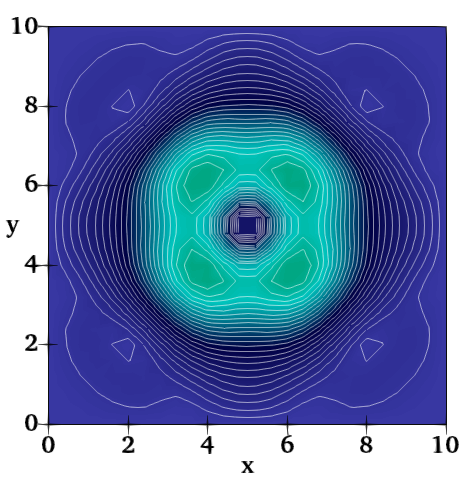}
					\caption{SUPG-Std}
				\end{subfigure}
				\hfill
				\begin{subfigure}[b]{0.328\textwidth}
					\centering
					\includegraphics[width=\textwidth]{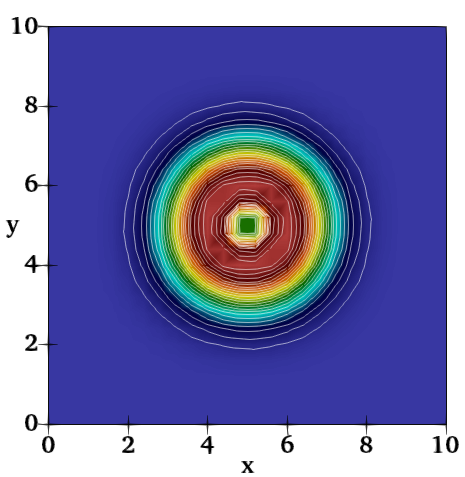}
					\caption{SUPG-GFQ}
				\end{subfigure}
			\end{minipage}
			\hfill
			\begin{minipage}[t]{0.09\textwidth}
				\vspace{-4.87cm}
				\includegraphics[width=\textwidth]{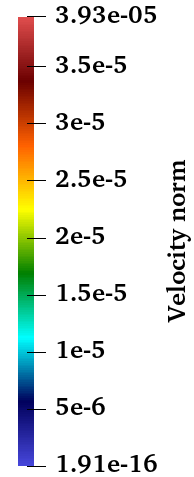}
			\end{minipage}
			
			\vspace{0.2cm}
			
			\noindent
			Mach = $10^{-6}$ \\
			\begin{minipage}[t]{0.9\textwidth}
				\centering
				\begin{subfigure}[b]{0.328\textwidth}
					\centering
					\includegraphics[width=\textwidth]{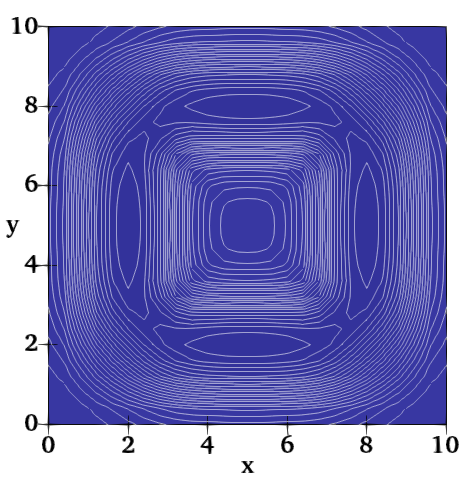}
					\caption{HLLC (2nd)}
				\end{subfigure}
				\hfill
				\begin{subfigure}[b]{0.328\textwidth}
					\centering
					\includegraphics[width=\textwidth]{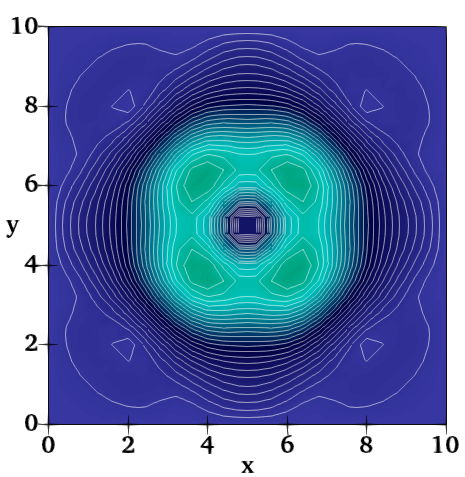}
					\caption{SUPG-Std}
				\end{subfigure}
				\hfill
				\begin{subfigure}[b]{0.328\textwidth}
					\centering
					\includegraphics[width=\textwidth]{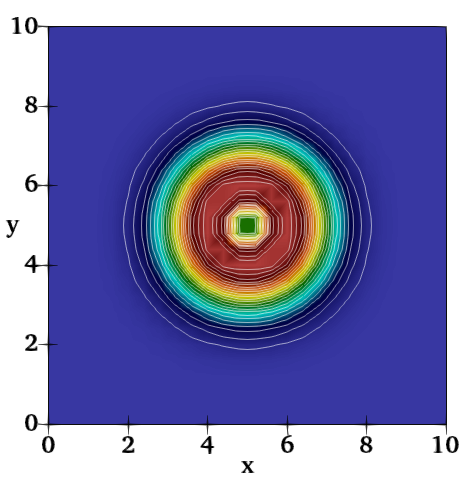}
					\caption{SUPG-GFQ}
				\end{subfigure}
			\end{minipage}
			\hfill
			\begin{minipage}[t]{0.09\textwidth}
				\vspace{-4.87cm}
				\includegraphics[width=\textwidth]{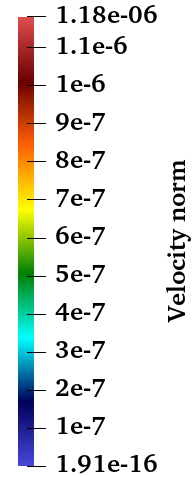}
			\end{minipage}
			
			\caption{
				Low-Mach-number stationary isentropic vortex.
				Isocontours of the velocity magnitude at the long-time limit 
				($t_{\mathrm{end}} = 2000\,\mathrm{s}$), 
				obtained with the second-order HLLC scheme (\textbf{left}), 
				SUPG-Std with $\mathbb{Q}^1$ elements (\textbf{middle}), 
				and SUPG-GFQ (\textbf{right}).
				\textbf{Top}: $\mathrm{Ma} = 10^{-2}$.
				\textbf{Middle}: $\mathrm{Ma} = 10^{-4}$.
				\textbf{Bottom}: $\mathrm{Ma} = 10^{-6}$.
				Computations performed on a $25 \times 25$ mesh.
			}
			
			\label{fig:long_time_norm_vortex}
		\end{figure}
		
		Finally, we run a long time simulation for the isentropic vortex test case, up to $t_{\mathrm{end}} = 2000\,\mathrm{s}$, to investigate the long-term behaviour of the numerical schemes. The results are presented in Figure \ref{fig:long_time_norm_vortex}, where we observe that the second-order HLLC scheme introduces significant numerical diffusion, leading to a constant density and zero velocity. In contrast, the SUPG-Std keeps more structures even at long times. This is an improvement brought by the  nice structure of the streamline upwind
		dissipation which has, in the linear acoustic limit, the correct grad-div form. However, the lack of a discrete kernel leads to the dissipation of the 
		initial condition as one can see from the figure. 
		The SUPG-GFQ  preserves nicely the initial structure, with perfectly stable and converging  results, with no spurious modes.
		
		\newpage
{\color{black}
\subsection{Low Mach Behavior:  Gresho vortex}
\label{subsec:smooth_gresho}
We consider the stationary Gresho vortex \cite{GreshoChan1990}, a standard
benchmark for nearly incompressible flows \cite{Barsukow2021allspeed}. Unlike the
isentropic vortex, it features a non-trivial azimuthal velocity balanced by a
radial pressure gradient.  We refer to the   references for the 
precise definition of the velocity components and pressure (the exact density field is constant).    
We just remark that   in the initialization of this vortex, the low Mach
regime is embedded in the definition only of the ambient pressure which scales as $p_0\sim 1/\mathrm{Ma}^2$,
while the velocities remain of order one. This   allows to see much stronger effects for shorter  final times.

Firstly we consider the standard definition which involves   velocity fields   with very low regularity:
only $C^0$ continuity, with discontinuous derivatives at the vortex foot and in correspondence of the maximum speeds.
For these solutions we compare the $\mathbb Q^1$ results for the standard method, and for the new one 
for $\mathrm{Ma} = 10^{-2}$ and  $\mathrm{Ma} = 10^{-4}$, at the final time $t_{\mathrm{end}}=1$ which is the one usually taken in literature.
On figure \ref{fig:gresho_vtk_q1} we report the Mach number contours of the two methods, and the exact ones for comparison.
We can clearly see the dissipative behaviour of the classical method, and the superiority of the new one.
To further confirm this, as it is customary we report on figure \ref{fig:gresho_vy_q1} a line plot of the velocity which
show the great advantage of the new formulation.


\begin{figure}[!htbp]
	\centering
	\setlength{\tabcolsep}{0pt}%
	\newcommand{\GqPanel}[1]{\raisebox{-0.5\height}{%
			\includegraphics[width=\dimexpr(\textwidth-30pt)*100/346\relax]{./#1}}}%
	\newcommand{\GqBar}[1]{\raisebox{-0.5\height}{%
			\includegraphics[height=\dimexpr(\textwidth-30pt)*88/346\relax]{./#1}}}%
	\newcommand{\GqRow}[1]{\raisebox{-0.5\height}{\rotatebox{90}{#1}}}%
	\newcommand{\GvPanel}[1]{\includegraphics[width=\dimexpr(\textwidth-10pt)/2\relax]{./#1}}%
	\begin{tabular}{@{}c@{\hspace{6pt}}c@{\hspace{3pt}}c@{\hspace{3pt}}c@{\hspace{5pt}}c@{}}
		& {\small SUPG-Std} & {\small SUPG-GFQ} & {\small Exact} & \\[-9pt]
		\GqRow{$\mathrm{Ma}=10^{-2}$}
		& \GqPanel{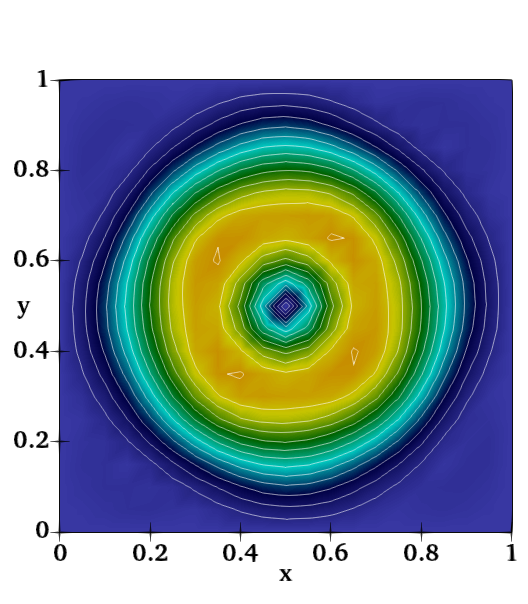} & \GqPanel{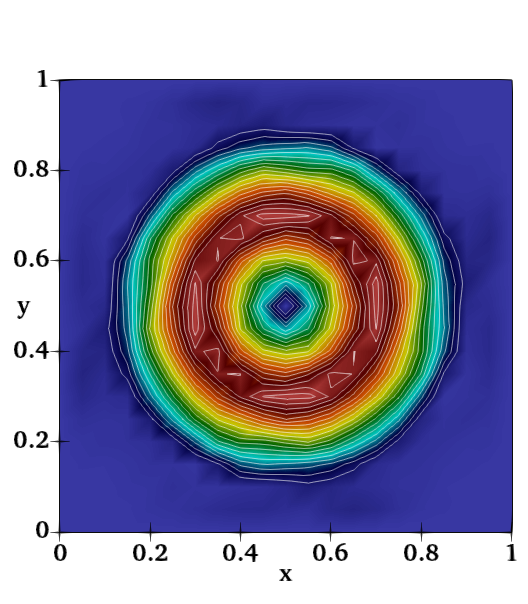} & \GqPanel{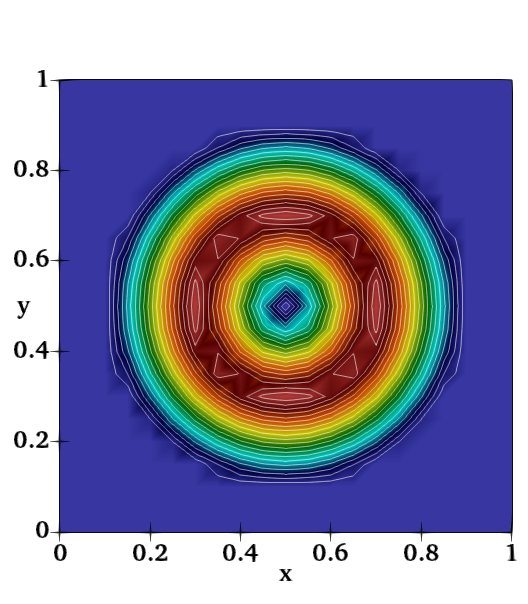}
		& \GqBar{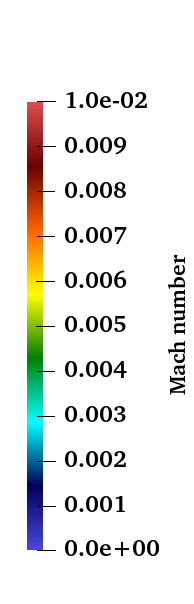} \\[1pt]
		& {\footnotesize (a)} & {\footnotesize (b)} & {\footnotesize (c)} & \\[-9pt]
		\GqRow{$\mathrm{Ma}=10^{-4}$}
		& \GqPanel{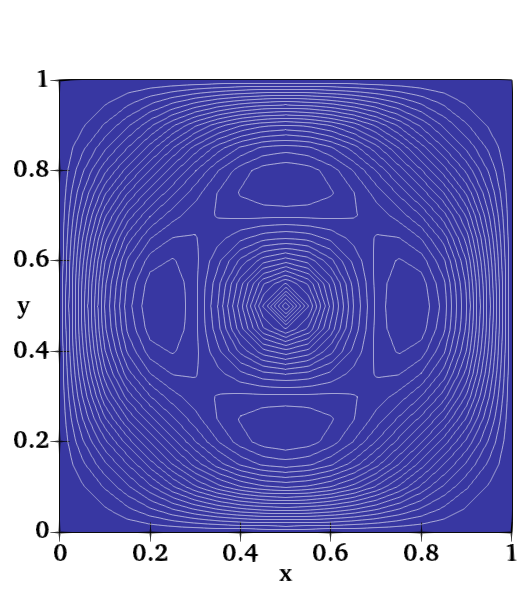} & \GqPanel{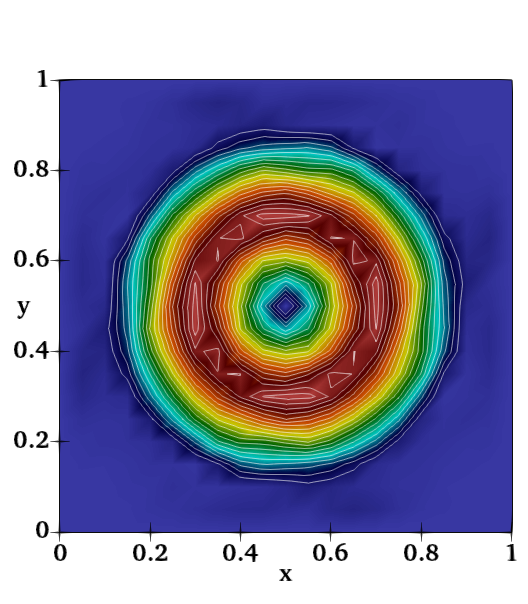} & \GqPanel{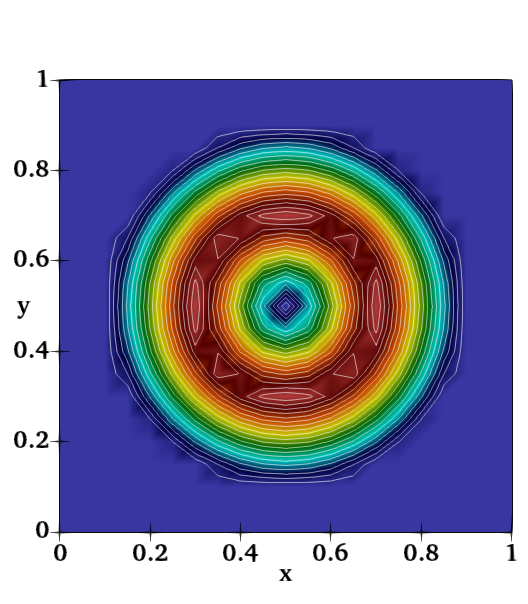}
		& \GqBar{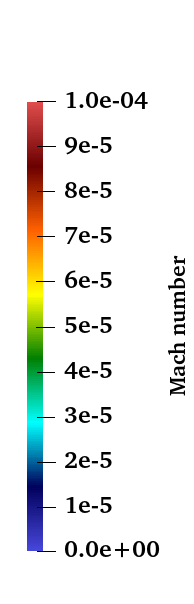} \\[1pt]
		& {\footnotesize (d)} & {\footnotesize (e)} & {\footnotesize (f)} & \\
	\end{tabular}
	\caption{Gresho vortex at
		$t_{\mathrm{end}}=1$, computed with $\mathbb{Q}^1$ elements on a $20\times20$ mesh:
		SUPG-Std (a,\,d), SUPG-GFQ (b,\,e) and exact solution (c,\,f), for
		$\mathrm{Ma}=10^{-2}$ (top) and $\mathrm{Ma}=10^{-4}$ (bottom).}
	\label{fig:gresho_vtk_q1}
	\end{figure}

\begin{figure}[!htbp]
	\centering
	\setlength{\tabcolsep}{0pt}%
	\newcommand{\GqPanel}[1]{\raisebox{-0.5\height}{%
			\includegraphics[width=\dimexpr(\textwidth-30pt)*100/346\relax]{./#1}}}%
	\newcommand{\GqBar}[1]{\raisebox{-0.5\height}{%
			\includegraphics[height=\dimexpr(\textwidth-30pt)*88/346\relax]{./#1}}}%
	\newcommand{\GqRow}[1]{\raisebox{-0.5\height}{\rotatebox{90}{#1}}}%
	\newcommand{\GvPanel}[1]{\includegraphics[width=\dimexpr(\textwidth-10pt)/2\relax]{./#1}}%
	\begin{tabular}{@{}c@{\hspace{9pt}}c@{}}
		\GvPanel{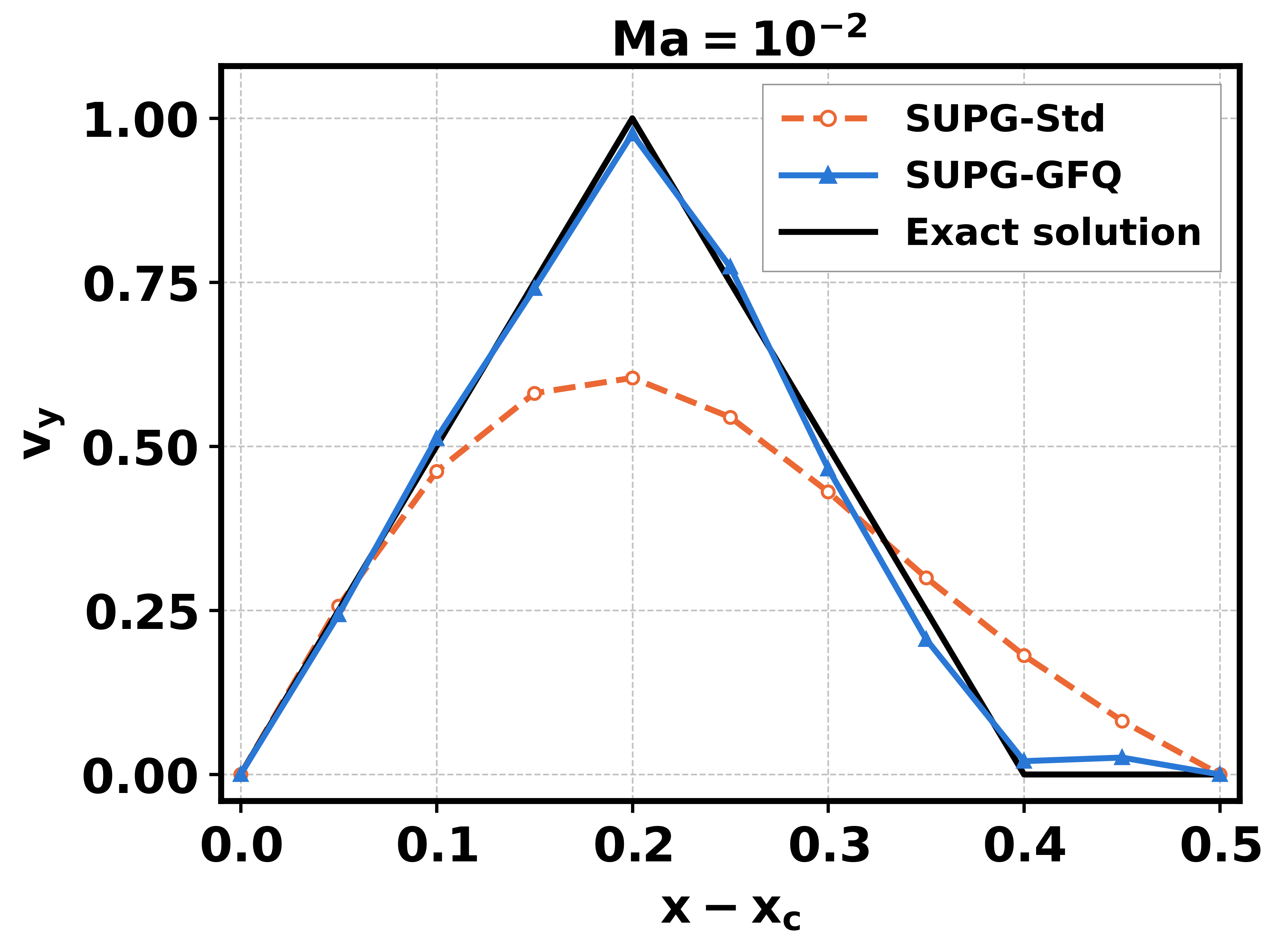} & \GvPanel{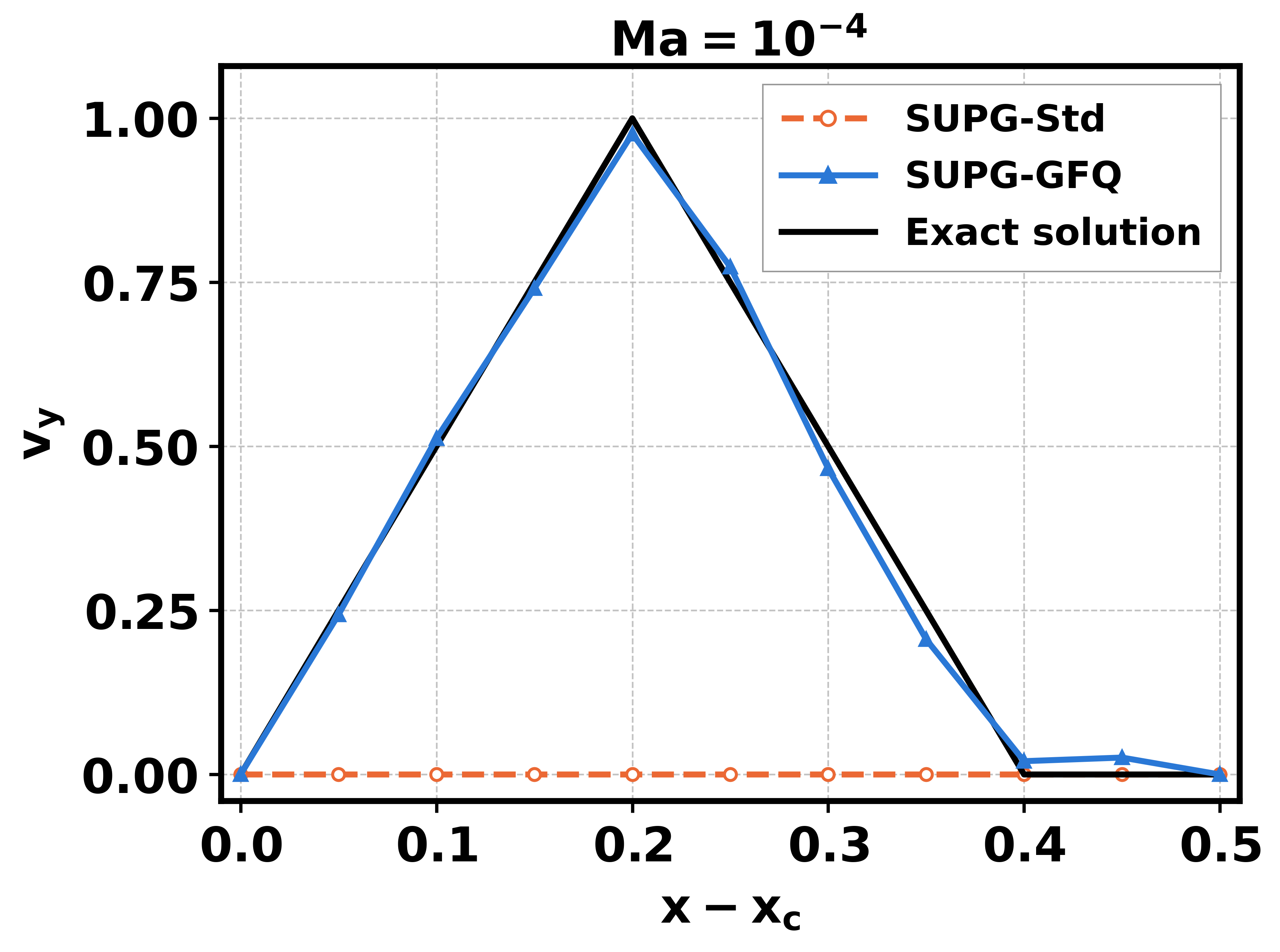} \\[-2pt]
		{\footnotesize (a)} & {\footnotesize (b)} \\
	\end{tabular}
	\caption{Gresho vortex: $y$-component of the velocity along the positive
		$x$-axis through the vortex centre at $t_{\mathrm{end}}=1$, computed with
		$\mathbb{Q}^1$ elements on a $20\times20$ mesh with SUPG-Std and SUPG-GFQ.}
	\label{fig:gresho_vy_q1}
\end{figure}
\begin{table}[htbp]
	\centering
	\caption{Convergence of the \textbf{smooth steady Gresho vortex} at $t_{\mathrm{end}}=1$\,s for $\mathrm{Ma}=10^{-2}$ and $10^{-4}$.
		$L^2$ errors of the density $\rho$, $x$-momentum $\rho u$, and total energy $\rho E$, together with experimental orders of
		accuracy (EOA), using $\mathbb{Q}^1$, $\mathbb{Q}^2$, and $\mathbb{Q}^3$ elements with SUPG-Std and SUPG-GFQ.}
	\label{tab:gresho_smooth_low_mach}
	\begingroup
	\small
	\setlength{\tabcolsep}{4pt}
	\renewcommand{\arraystretch}{1.1}
	\resizebox{\linewidth}{!}{%
		\begin{tabular}{c|c|cc|cc|cc|cc|cc|cc}
			\hline
			\multicolumn{2}{c|}{Mach} & \multicolumn{6}{c|}{$\mathrm{Ma}=10^{-2}$} & \multicolumn{6}{c}{$\mathrm{Ma}=10^{-4}$} \\
			\hline
			\multicolumn{2}{c|}{Variable} & \multicolumn{2}{c|}{$\rho$} & \multicolumn{2}{c|}{$\rho u$} & \multicolumn{2}{c|}{$\rho E$} & \multicolumn{2}{c|}{$\rho$} & \multicolumn{2}{c|}{$\rho u$} & \multicolumn{2}{c}{$\rho E$} \\
			\hline
			Element & $N_x$ & $L^2$ error & EOA & $L^2$ error & EOA & $L^2$ error & EOA & $L^2$ error & EOA & $L^2$ error & EOA & $L^2$ error & EOA \\
			\hline
			\multicolumn{14}{c}{\textbf{SUPG-Std}} \\
			\hline
			\multirow{4}{*}{$\mathbb{Q}^{1}$} & 20 & 1.04E-05 & -- & 3.27E-01 & -- & 1.16E-05 & -- & 2.46E-09 & -- & 1.00E+00 & -- & 2.65E-09 & -- \\
			& 40 & 3.20E-06 & 1.70 & 1.40E-01 & 1.22 & 3.53E-06 & 1.72 & 2.35E-09 & 0.06 & 8.60E-01 & 0.22 & 2.59E-09 & 0.03 \\
			& 80 & 6.05E-07 & 2.40 & 3.26E-02 & 2.11 & 6.48E-07 & 2.45 & 1.32E-09 & 0.83 & 4.02E-01 & 1.09 & 1.48E-09 & 0.80 \\
			& 160 & 7.94E-08 & 2.93 & 7.04E-03 & 2.21 & 8.82E-08 & 2.88 & 4.17E-10 & 1.66 & 1.73E-01 & 1.22 & 4.63E-10 & 1.68 \\
			\hline
			\multirow{4}{*}{$\mathbb{Q}^{2}$} & 10 & 8.09E-06 & -- & 1.50E-01 & -- & 3.23E-06 & -- & 1.48E-09 & -- & 2.85E-01 & -- & 5.40E-10 & -- \\
			& 20 & 2.05E-06 & 1.98 & 2.16E-02 & 2.79 & 7.28E-07 & 2.15 & 2.05E-10 & 2.85 & 2.33E-02 & 3.61 & 8.50E-11 & 2.67 \\
			& 40 & 5.83E-07 & 1.81 & 4.79E-03 & 2.17 & 1.79E-07 & 2.02 & 6.16E-11 & 1.73 & 4.87E-03 & 2.26 & 1.79E-11 & 2.25 \\
			& 80 & 1.65E-07 & 1.82 & 1.15E-03 & 2.06 & 4.61E-08 & 1.96 & 2.05E-11 & 1.59 & 1.18E-03 & 2.05 & 4.41E-12 & 2.02 \\
			\hline
			\multirow{4}{*}{$\mathbb{Q}^{3}$} & 6 & 8.12E-06 & -- & 8.24E-02 & -- & 2.57E-06 & -- & 1.56E-09 & -- & 1.41E-01 & -- & 4.91E-10 & -- \\
			& 12 & 1.49E-06 & 2.45 & 1.13E-02 & 2.87 & 4.55E-07 & 2.50 & 1.74E-10 & 3.17 & 2.98E-02 & 2.24 & 1.56E-10 & 1.65 \\
			& 24 & 8.75E-08 & 4.09 & 1.76E-03 & 2.68 & 6.42E-08 & 2.82 & 4.30E-11 & 2.01 & 3.83E-03 & 2.96 & 5.87E-11 & 1.41 \\
			& 48 & 5.00E-09 & 4.13 & 2.11E-04 & 3.06 & 5.57E-09 & 3.53 & 1.05E-11 & 2.03 & 9.51E-04 & 2.01 & 1.13E-11 & 2.38 \\
			\hline
			\multicolumn{14}{c}{\textbf{SUPG-GFQ}} \\
			\hline
			\multirow{4}{*}{$\mathbb{Q}^{1}$} & 20 & 1.87E-06 & -- & 9.13E-02 & -- & 1.75E-06 & -- & 1.87E-10 & -- & 9.14E-02 & -- & 1.75E-10 & -- \\
			& 40 & 3.02E-07 & 2.63 & 1.94E-02 & 2.24 & 2.54E-07 & 2.78 & 3.02E-11 & 2.63 & 1.94E-02 & 2.24 & 2.55E-11 & 2.78 \\
			& 80 & 7.59E-08 & 1.99 & 4.98E-03 & 1.96 & 7.02E-08 & 1.86 & 7.60E-12 & 1.99 & 4.99E-03 & 1.96 & 7.03E-12 & 1.86 \\
			& 160 & 1.81E-08 & 2.07 & 1.25E-03 & 1.99 & 1.65E-08 & 2.09 & 1.82E-12 & 2.07 & 1.25E-03 & 1.99 & 1.66E-12 & 2.09 \\
			\hline
			\multirow{4}{*}{$\mathbb{Q}^{2}$} & 10 & 1.70E-06 & -- & 6.24E-02 & -- & 9.69E-07 & -- & 1.71E-10 & -- & 6.25E-02 & -- & 9.69E-11 & -- \\
			& 20 & 1.39E-07 & 3.61 & 7.33E-03 & 3.09 & 1.07E-07 & 3.18 & 1.40E-11 & 3.60 & 7.34E-03 & 3.09 & 1.07E-11 & 3.18 \\
			& 40 & 7.13E-09 & 4.29 & 7.36E-04 & 3.32 & 4.53E-09 & 4.56 & 7.33E-13 & 4.26 & 7.41E-04 & 3.31 & 4.57E-13 & 4.55 \\
			& 80 & 4.59E-10 & 3.96 & 5.97E-05 & 3.62 & 2.76E-10 & 4.03 & 1.56E-13 & 2.23 & 7.25E-05 & 3.35 & 3.96E-14 & 3.53 \\
			\hline
			\multirow{4}{*}{$\mathbb{Q}^{3}$} & 6 & 2.04E-06 & -- & 8.44E-02 & -- & 1.63E-06 & -- & 2.05E-10 & -- & 8.50E-02 & -- & 1.64E-10 & -- \\
			& 12 & 1.25E-07 & 4.02 & 8.33E-03 & 3.34 & 8.46E-08 & 4.27 & 1.28E-11 & 4.01 & 8.41E-03 & 3.34 & 8.49E-12 & 4.27 \\
			& 24 & 4.53E-09 & 4.79 & 5.21E-04 & 4.00 & 2.63E-09 & 5.00 & 4.78E-13 & 4.74 & 5.25E-04 & 4.00 & 2.67E-13 & 4.99 \\
			& 48 & 3.16E-10 & 3.84 & 4.85E-05 & 3.43 & 1.84E-10 & 3.84 & 1.60E-13 & 1.58 & 7.94E-05 & 2.73 & 3.49E-14 & 2.94 \\
			\hline
		\end{tabular}%
	}
	\endgroup
\end{table}
\paragraph{Grid convergence for a smooth   Gresho vortex.} To perform a grid convergence, we now consider  
the smooth piecewise-cubic profile proposed by Berthon, Klingenberg and Zenk \cite{BerthonKlingenbergZenk2020}, also used in
\cite{ThomannZenkPuppoKlingenberg2020}, which is more regular than the
classical piecewise-linear profile and therefore better suited to mesh-convergence studies.
In particular note that this version allows to have $C^1$ velocity profiles which in principle allows to measure accuracy orders 
up to three, due to the  presence of discontinuities of the second derivatives. We refer to the last two references for the 
precise definition of the velocity components and pressure. Note that as for the previous case, the initial density is constant 
and the ambient pressure verifies $p_0 \sim 1/\mathrm{Ma}^2$. We report on table \ref{tab:gresho_smooth_low_mach}
the convergence study and error values for the  standard and stationarity preserving formulations
using the spaces   $\mathbb Q^1$ to $\mathbb Q^3$, 
for $\mathrm{Ma} = 10^{-2}$ and  $\mathrm{Ma} = 10^{-4}$,  at the final time $t_{\mathrm{end}}=1$.
The results of the table  show  that at lowest order and for $\mathrm{Ma} = 10^{-2}$
both approaches give the expected second order convergence rates, the standard method even slightly super-converging.
However,  the SUPG-GFQ still handles this solution much better, with  errors which   are four times smaller at least. 
For $\mathrm{Ma} = 10^{-4}$  SUPG-Std  fails to reach the second order slope, while the stationarity preserving does, 
roviding error reduction up to two orders of magnitudes.

For the third and fourth order computations the conclusions are similar. The SUPG-GFQ results 
show generally slopes  which are  higher than the standard method, and which at least on the coarsest meshes keep some
degree of super-convergence above the third order (fourth order for the $\mathbb{Q}^{2}$ and almost fifth order for the $\mathbb{Q}^{3}$), despite the lack  of regularity. 
The finer mesh slopes seem to  indicate a reduction toward the expected limit of order 3.
The error reduction with  the stationarity preserving formulations  are of a factor in between 10 and 50. 
For completeness, we  also provide convergence curves corresponding to the same study on figure \ref{fig:gresho_smooth_convergence}.
The pictures show visually the considerable error reduction obtained at fixed  number of unknonws.


\begin{figure}[!htbp]
	\centering
	
	\includegraphics[width=1.02\textwidth]{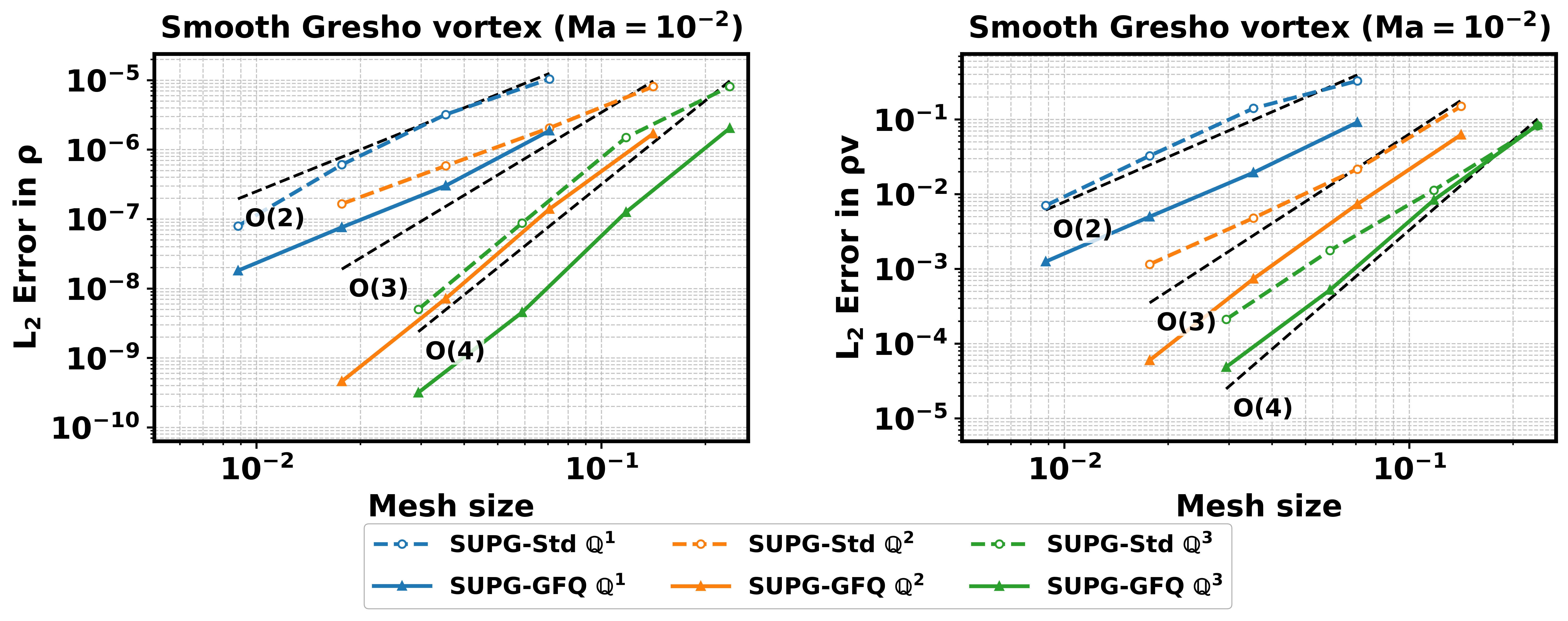}
	
	\vspace{0.2cm}
	
	\includegraphics[width=1.02\textwidth]{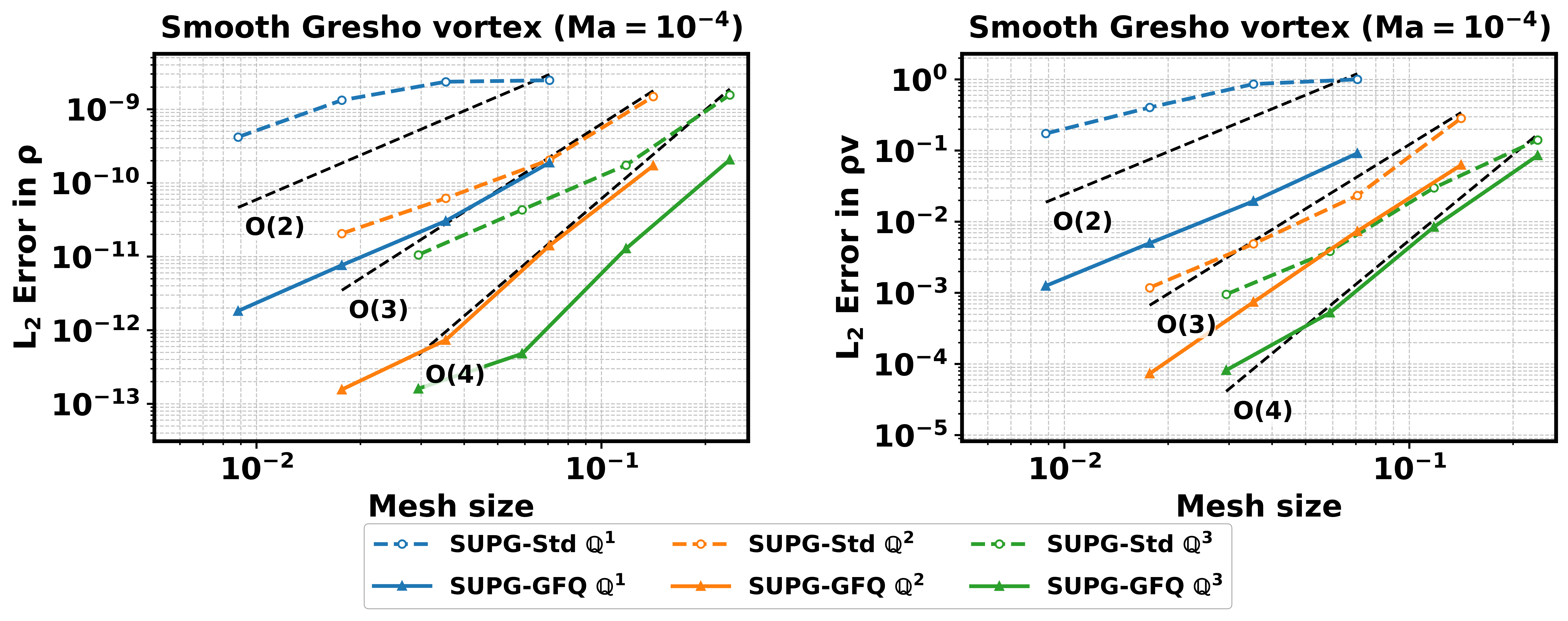}
	
	\caption{
		Convergence of the relative $L^2$ errors in
		$\rho$ (\textbf{left}) and $\rho v$ (\textbf{right})
		for the smooth steady Gresho vortex at
		$t_{\mathrm{end}}=1\,\mathrm{s}$.
		Comparison of SUPG-Std and SUPG-GFQ using
		$\mathbb{Q}^1$, $\mathbb{Q}^2$, and $\mathbb{Q}^3$
		elements for Mach numbers
		$\mathrm{Ma}=10^{-2}$ (\textbf{top}) and
		$\mathrm{Ma}=10^{-4}$ (\textbf{bottom}).
	}
	
	\label{fig:gresho_smooth_convergence}
\end{figure}

}


\subsection{Kelvin--Helmholtz instability}
We consider now a first case involving a hydrodynamic instability. As we will see, 
the stationarity preserving quadrature developed in this work brings significant enhancements in the resolution of 
such instabilities when they evolve starting from a quasi-stationary state.  In this section, in particular we consider the formation and evolution
of a  Kelvin–Helmholtz instability. We consider a setup  often used to evaluate the behaviour of the numerical dissipation in the low Mach regime \cite{leidi2024performance}.

\begin{figure}[t]

\begin{minipage}[t]{0.86\textwidth}
\centering

\begin{subfigure}[t]{0.32\textwidth}
	\centering
	\includegraphics[width=\textwidth, trim={15 40 35 40}, clip]{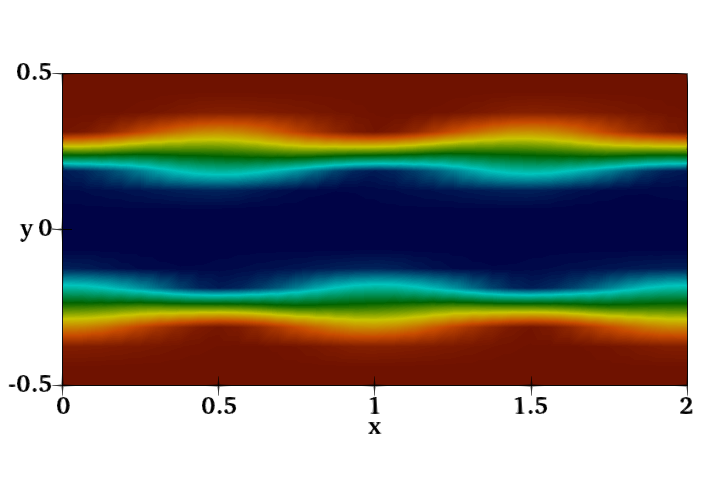}
	\vspace{-0.6cm}
	\caption{HLLC (2nd), $32 \times 16$}
	\includegraphics[width=\textwidth, trim={15 40 35 40}, clip]{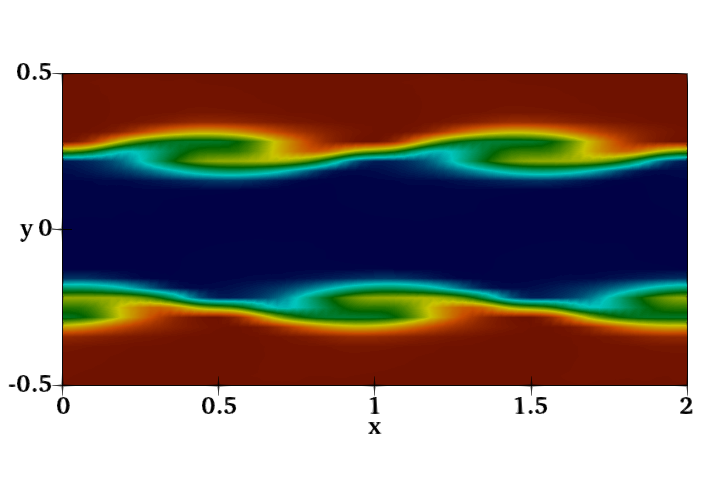}
	\vspace{-0.6cm}
	\caption{HLLC (2nd), $64 \times 32$}
	\includegraphics[width=\textwidth, trim={15 40 35 40}, clip]{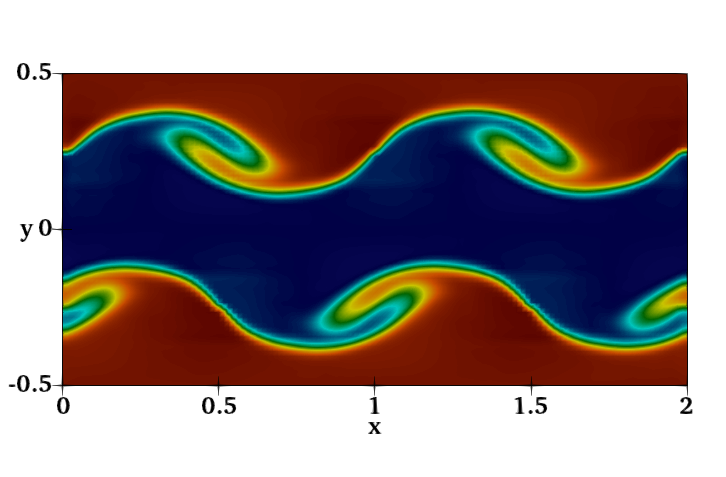}
	\vspace{-0.6cm}
	\caption{HLLC (2nd), $128 \times 64$}
	\includegraphics[width=\textwidth, trim={15 40 35 40}, clip]{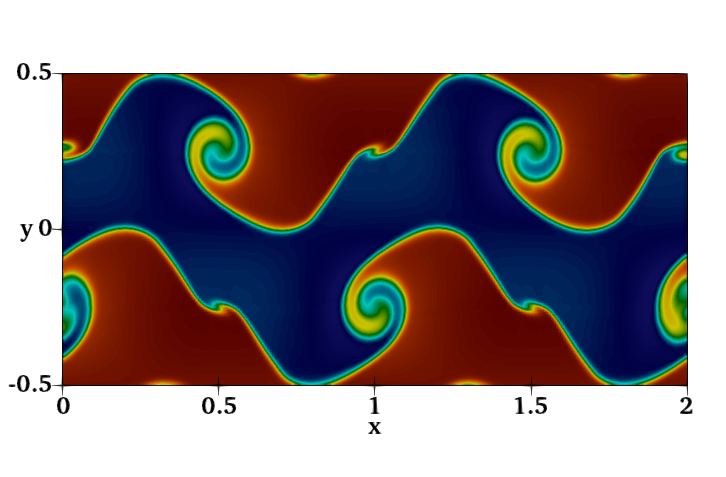}
	\vspace{-0.6cm}
	\caption{HLLC (2nd), $256 \times 128$}
\end{subfigure}
\hfill
\begin{subfigure}[t]{0.32\textwidth}
	\centering
	
	\includegraphics[width=\textwidth, trim={15 40 35 40}, clip]{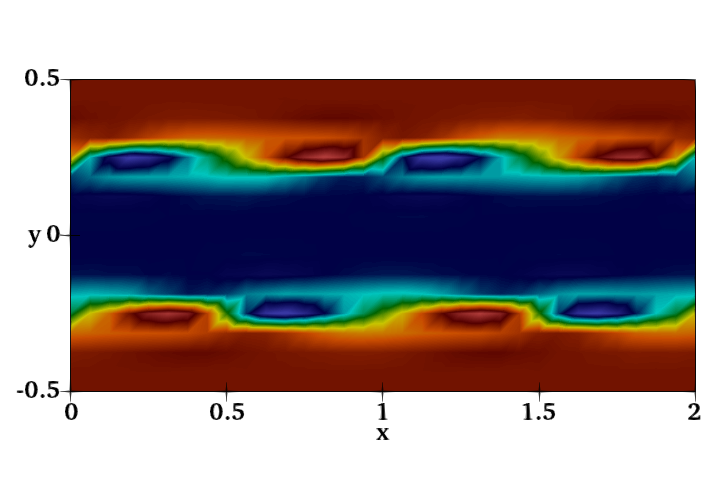}
	\vspace{-0.6cm}
	\caption{SUPG-Std, $32 \times 16$}
	
	\includegraphics[width=\textwidth, trim={15 40 35 40}, clip]{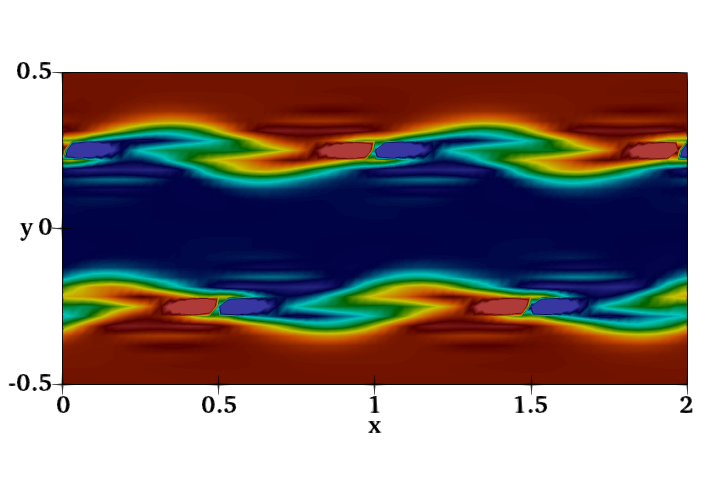}
	\vspace{-0.6cm}
	\caption{SUPG-Std, $64 \times 32$}
	
	\includegraphics[width=\textwidth, trim={15 40 35 40}, clip]{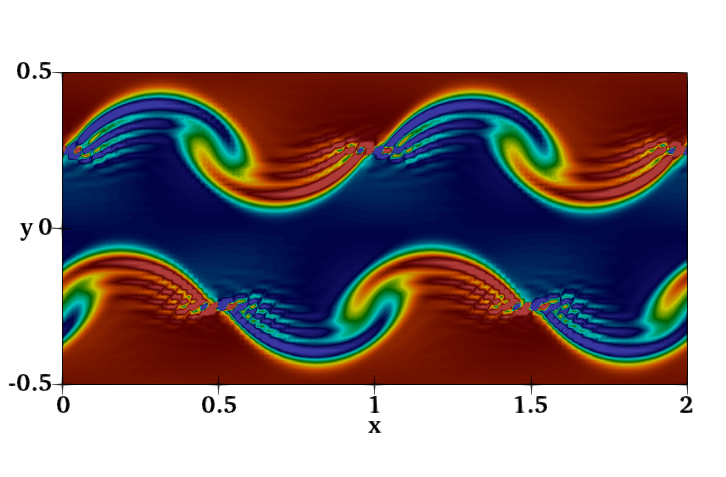}
	\vspace{-0.6cm}
	\caption{SUPG-Std, $128 \times 64$}
	
	\includegraphics[width=\textwidth, trim={15 40 35 40}, clip]{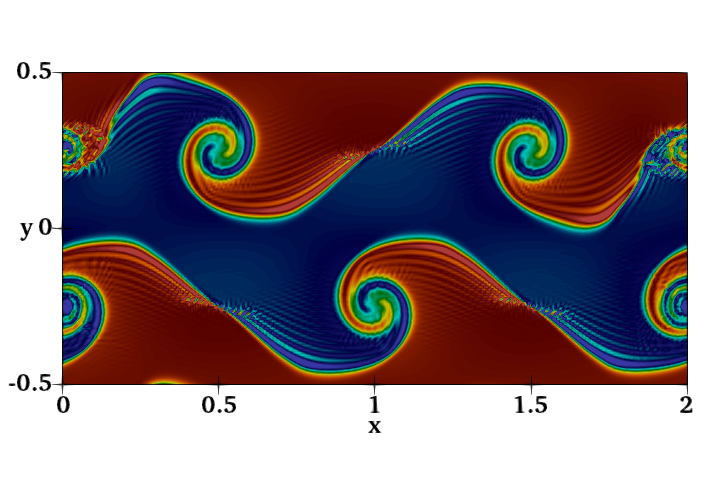}
	\vspace{-0.6cm}
	\caption{SUPG-Std, $256 \times 128$}
	
\end{subfigure}
\hfill
\begin{subfigure}[t]{0.32\textwidth}
	\centering
	
	\includegraphics[width=\textwidth, trim={15 40 35 40}, clip]{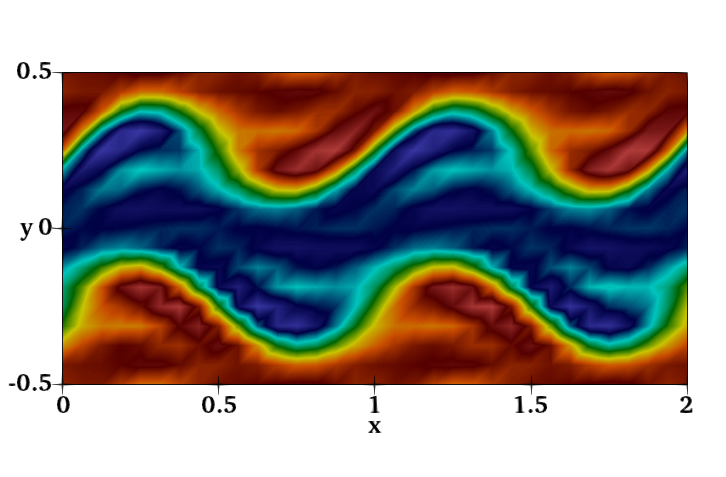}
	\vspace{-0.6cm}
	\caption{SUPG-GFQ, $32 \times 16$}
	
	\includegraphics[width=\textwidth, trim={15 40 35 40}, clip]{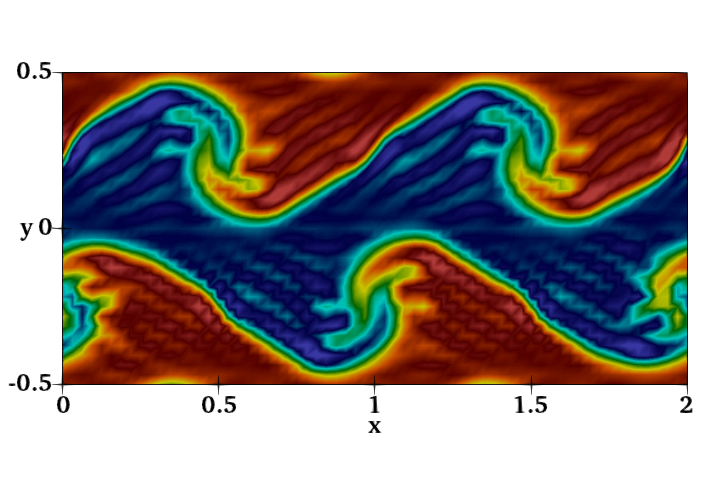}
	\vspace{-0.6cm}
	\caption{SUPG-GFQ, $64 \times 32$}
	
	\includegraphics[width=\textwidth, trim={15 40 35 40}, clip]{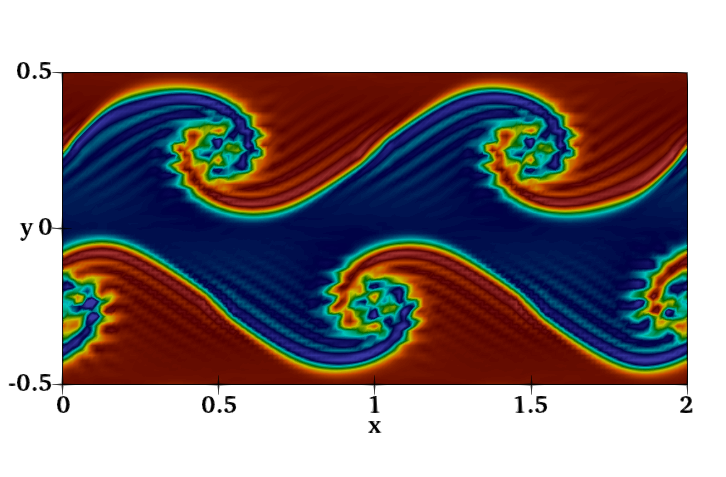}
	\vspace{-0.6cm}
	\caption{SUPG-GFQ, $128 \times 64$}
	
	\includegraphics[width=\textwidth, trim={15 40 35 40}, clip]{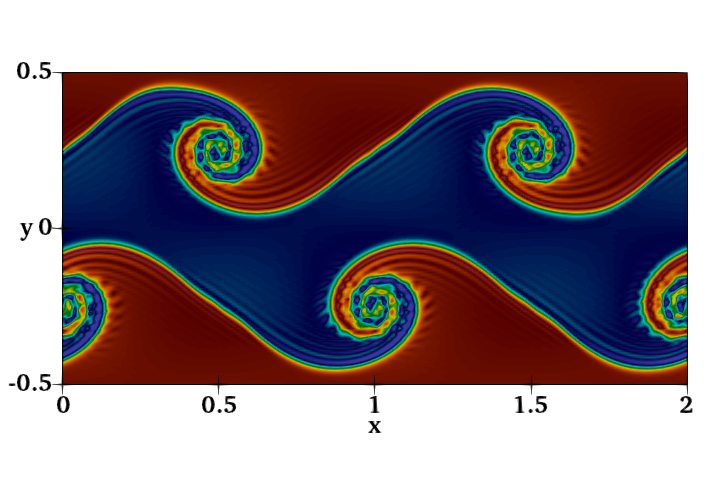}
	\vspace{-0.6cm}
	\caption{SUPG-GFQ, $256 \times 128$}
	
\end{subfigure}

\end{minipage}
\captionof{figure}{Kelvin–Helmholtz instability. Density field at final time 
computed with $\mathbb{Q}^1$ elements on successively refined meshes. 
\textbf{Left column:} second-order HLLC scheme. 
\textbf{Middle column:} standard SUPG method. 
\textbf{Right column:} SUPG-GFQ method.
}
\label{fig:all_KHSUPG_GFQ}
\hfill
\begin{minipage}[t][-16.4cm][c]{0.125\textwidth}
\includegraphics[width=1.15\textwidth, , trim={20 0 60 0},clip]{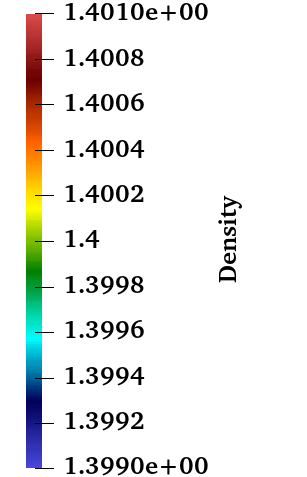}
\end{minipage}
\end{figure}

\begin{figure}[t]
\begin{minipage}{0.86\textwidth}
	\centering $t=60$s\\[5pt]
	\begin{minipage}{0.325\textwidth}
		\centering HLLC-FV\\
		\includegraphics[width=\textwidth, trim={0 40 0 40}, clip]{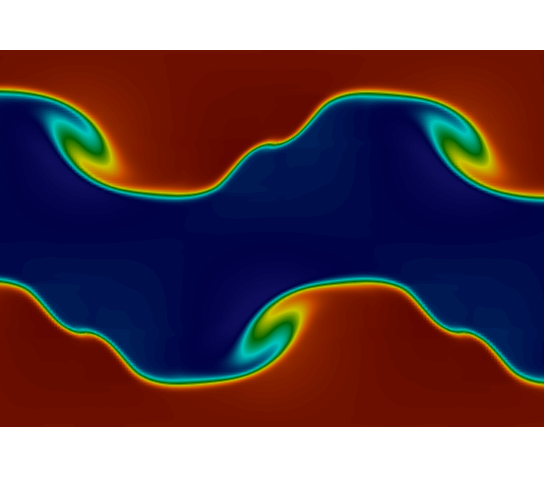}
	\end{minipage}\,
	\begin{minipage}{0.325\textwidth}
		\centering SUPG-Std\\
		\includegraphics[width=\textwidth, trim={0 40 0 40}, clip]{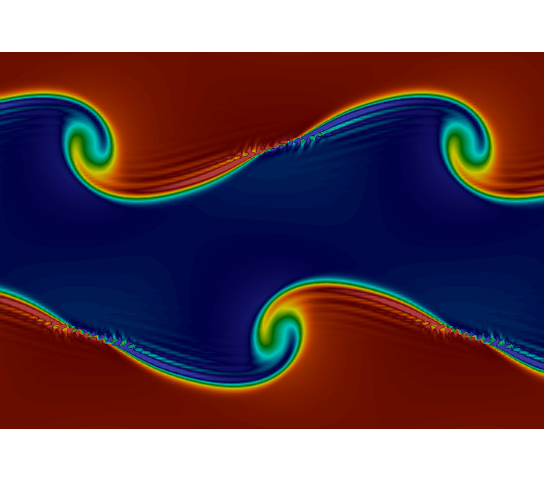}
	\end{minipage}\,
	\begin{minipage}{0.325\textwidth}
		\centering SUPG-GFQ\\
		\includegraphics[width=\textwidth, trim={0 40 0 40}, clip]{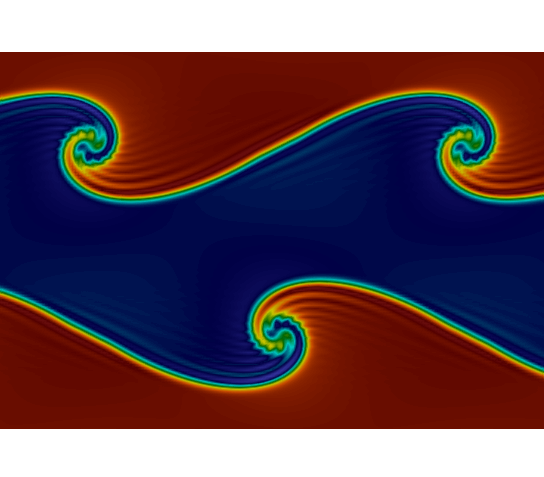}
	\end{minipage}\\[2mm]
	\centering $t=80$s\\[5pt]
	\begin{minipage}{0.325\textwidth}
		\centering HLLC-FV\\
		\includegraphics[width=\textwidth, trim={0 40 0 40}, clip]{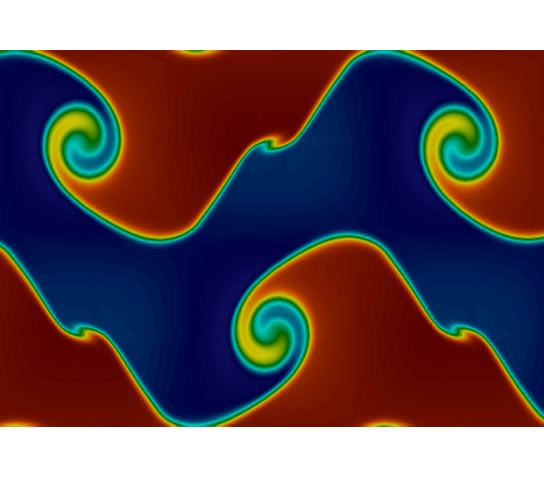}
	\end{minipage}\,
	\begin{minipage}{0.325\textwidth}
		\centering SUPG-Std\\
		\includegraphics[width=\textwidth, trim={0 40 0 40}, clip]{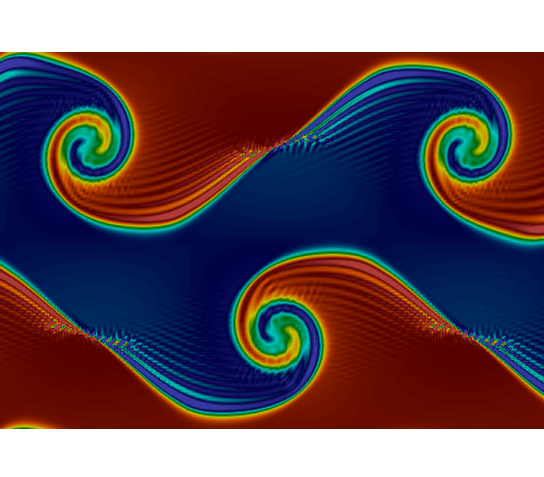}
	\end{minipage}\,
	\begin{minipage}{0.325\textwidth}
		\centering SUPG-GFQ\\
		\includegraphics[width=\textwidth, trim={0 40 0 40}, clip]{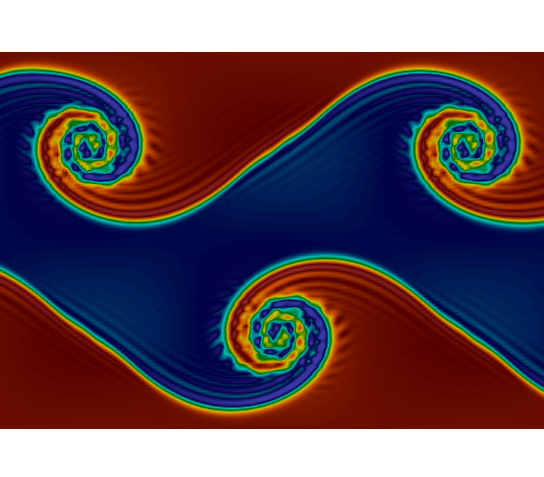}
	\end{minipage}
\end{minipage}
\hfill
\begin{minipage}{0.125\textwidth}
	\centering
	\vspace{0.9cm}	
	\includegraphics[width=1.15\textwidth, trim={20 0 60 0},clip]{./Color_bar_2.png}
\end{minipage}

\caption{Evolution of the Kelvin-Helmholtz instability at different time instants ($t=60$ s top, $t=80$ s bottom), computed with HLLC-FV (\textbf{left}), the SUPG-Std method (\textbf{center}), and the SUPG-GFQ method (\textbf{right}) on a $256 \times 128$ grid.}
\label{fig:KH_evolution_GF}
\end{figure}

The simulations are performed in the domain $[0,2]\times[-1/2,1/2]$ up to a final time $t_{\mathrm{end}} = 80\mathrm{s}$ with periodic boundary conditions. The initial condition is defined as follows:
\begin{equation}
	\begin{aligned}
		\rho &= \gamma + K(y)\,r, \\
		u &= M\,K(y), \\
		v &= \delta\,M\,\sin(2\pi x), \\
		p &= 1.
	\end{aligned}
\end{equation}
where the parameters are set to $M=10^{-2}$, $r=10^{-3}$, and $\delta=0.1$.  The function $K(y)$ introduces a layered shear profile and is defined as:
\begin{equation}
	K(y)=
	\begin{cases}
		-\sin\left(\dfrac{\pi}{\omega}\left(y+\dfrac{1}{4}\right)\right), 
		& \text{if } -\dfrac{1}{4}-\dfrac{\omega}{2}\le y< -\dfrac{1}{4}+\dfrac{\omega}{2}, \\[1.5ex]
		-1, 
		& \text{if } -\dfrac{1}{4}+\dfrac{\omega}{2}\le y< \dfrac{1}{4}-\dfrac{\omega}{2}, \\[1.5ex]
		\sin\left(\dfrac{\pi}{\omega}\left(y-\dfrac{1}{4}\right)\right), 
		& \text{if } \dfrac{1}{4}-\dfrac{\omega}{2}\le y< \dfrac{1}{4}+\dfrac{\omega}{2}, \\[1.5ex]
		1, 
		& \text{otherwise},
	\end{cases}
\end{equation}
with $\omega=1/16$.  This is an interesting case because the  shear flow developing from  this initial state remains smooth  and stable for
relatively short times so that numerical schemes converge under mesh refinement, and a reference solution exists, to which we refer for comparison \cite{leidi2024performance}.

Figure~\ref{fig:all_KHSUPG_GFQ} shows the numerical results obtained with three second order methods: the HLLC-FV, the $\mathbb{Q}^1$ SUPG-Std and the $\mathbb{Q}^1$ SUPG-GFQ. The computations are carried out on the meshes $32\times 16$, $64 \times 32$, $128 \times 64$ and $256 \times 128$.

The HLLC-FV and SUPG-Std fail to accurately capture the flow structures generated by the instability up to the mesh $64\times 32$. 
Vortex formation becomes noticeable only on a moderately refined mesh ($128\times 64$). In this case, the higher-order accuracy partially mitigates 
excessive diffusion at this Mach number and simulation time. 
However, the resulting structures remain diffused and would require additional resolution to resolve the vortical features in detail.
Also, both the HLLC-FV and SUPG-Std show  the appearance of a spurious glitch  in between two subsequent vortices, in correspondence of   which the contact aligns itself to the mesh.

In contrast, the GFQ-based method successfully reproduces the main flow structures with high accuracy and extremely small mesh imprinting.
 Even on the coarsest mesh, fluid patterns begin to emerge and develop. A few spurious vortices are observed, which are well-known artefacts present in most numerical methods; see \cite{brown1995performance}. As the resolution increases, the flow features converge toward the reference solution reported in previous works~\cite{leidi2024performance,barsukow2025genuinely}. A comparison with low Mach-compliant schemes studied in~\cite{leidi2024performance} indicates that the GFQ method achieves comparable, if not superior, performance.

In Figure~\ref{fig:KH_evolution_GF}, we show the temporal evolution of the instability for the SUPG-GFQ, SUPG-Std, and HLLC-FV methods on the finest mesh. The GFQ method captures the flow structures from the early stages of the instability, while the other two methods fail to do so until later times. The HLLC-FV method exhibits significant diffusion, which is evident in the smoothed-out vortical structures at $t=80\ \text{s}$.

These results confirm  the enhancements at low Mach number  discussed in Remark 13.

\subsection{Isothermal hydrostatic equilibria and their perturbations} 

We now  consider the effects of the gravitational source.
We start by considering the approximation of an isothermal hydrostatic equilibrium,
and compare  the  SUPG-Std method, the SUPG-GFQ, and the exactly well-balanced SUPG-GFQ method,  obtained with the modification  proposed in Section~\ref{sec:hydro_equi}.
To this end,  on   the domain $\Omega = [0,1]\times[0,1]$ we consider the exact  equilibrium \cite{berberich2016general,xing2013high,chandrashekar2015second}. 
\begin{equation}
\rho(x,y)=\bar{\rho}\,\exp\!\left(-\frac{\bar{\rho}\,\phi(x,y)}{\bar{p}}\right),\qquad
p(x,y)=\bar{p}\,\exp\!\left(-\frac{\bar{\rho}\,\phi(x,y)}{\bar{p}}\right),\qand
\mathbf{u}(x,y)=\mathbf{0},
\end{equation}
with $\phi(x,y)=x+y$, $\bar{\rho}=1.21$, and $\bar{p}=1$. We set Dirichlet boundary data coincident with the exact equilibrium. 
The solution is initialized interpolating the analytical steady state and evolved until $t_{\mathrm{end}}=1\,\mathrm{s}$.  In Table~\ref{tab:isothermal_equilibrium_q1style}, we report the $\mathrm{L}^1$ errors with 
respect to the exact equilibrium of the SUPG-GFQ (WB) method for different polynomial degrees and mesh sizes.
We can see that the stationarity preserving method with the correction   preserves the equilibrium up to machine precision, as expected.
Interestingly, for $K\ge 2$  the stationarity preserving approach without any correction provides  error reductions  of more than 2 orders of magnitude.  

	\begin{table}[htbp]
	\centering
	\begin{threeparttable}
		\caption{$L^1$ errors for the isothermal hydrostatic equilibrium at final time for the three variants of the SUPG scheme: GFQ (well-balanced), GF (non well-balanced), and Std (non well-balanced).}
		\label{tab:isothermal_equilibrium_q1style}
		\small
		\begin{tabular}{c|c|c| c|c|c|c}
			\toprule
			Element & $N_x=N_y$ & Method & $\rho$ & $u$ & $v$ & $p$ \\
			\midrule
			
			\multirow{6}{*}{$\mathbb{Q}^1$}
			& \multirow{3}{*}{40}
			& GFQ (WB)          & 2.528E-15 & 2.629E-14 & 2.673E-14 & 2.262E-15 \\
			& & GFQ (non-WB)       & 9.628E-06 & 3.549E-05 & 3.549E-05 & 7.651E-06 \\
			& & Std (non-WB) & 1.356E-05 & 6.716E-05 & 6.716E-05 & 1.507E-05 \\
			\cmidrule(lr){2-7}
			& \multirow{3}{*}{80}
			& GFQ (WB)          & 6.091E-15 & 8.920E-14 & 9.936E-14 & 4.340E-15 \\
			& & GFQ (non-WB)       & 2.341E-06 & 8.963E-06 & 8.963E-06 & 1.923E-06 \\
			& & Std (non-WB) & 3.384E-06 & 1.690E-05 & 1.690E-05 & 3.895E-06 \\
			\midrule
			
			\multirow{6}{*}{$\mathbb{Q}^2$}
			& \multirow{3}{*}{20}
			& GFQ (WB)          & 4.319E-14 & 1.566E-13 & 1.563E-13 & 1.482E-14 \\
			& & GFQ (non-WB)       & 2.138E-08 & 1.990E-08 & 1.990E-08 & 2.649E-08 \\
			& & Std (non-WB) & 5.860E-06 & 2.844E-05 & 2.844E-05 & 4.534E-06 \\
			\cmidrule(lr){2-7}
			& \multirow{3}{*}{40}
			& GFQ (WB)          & 1.172E-13 & 4.896E-13 & 4.970E-13 & 3.387E-14 \\
			& & GFQ (non-WB)       & 1.415E-09 & 1.290E-09 & 1.290E-09 & 1.747E-09 \\
			& & Std (non-WB) & 8.126E-07 & 3.671E-06 & 3.671E-06 & 6.577E-07 \\
			\midrule
			
			\multirow{6}{*}{$\mathbb{Q}^3$}
			& \multirow{3}{*}{10}
			& GFQ (WB)          & 3.396E-13 & 4.427E-13 & 4.501E-13 & 3.665E-14 \\
			& & GFQ (non-WB)       & 4.689E-09 & 6.746E-09 & 6.746E-09 & 5.298E-09 \\
			& & Std (non-WB) & 5.420E-07 & 8.294E-07 & 8.294E-07 & 5.417E-07 \\
			\cmidrule(lr){2-7}
			& \multirow{3}{*}{20}
			& GFQ (WB)          & 1.342E-12 & 1.713E-12 & 1.754E-12 & 9.255E-14 \\
			& & GFQ (non-WB)       & 1.453E-10 & 2.535E-10 & 2.535E-10 & 1.641E-10 \\
			& & Std (non-WB) & 3.407E-08 & 4.724E-08 & 4.724E-08 & 3.428E-08 \\
			\midrule
			
			\multirow{6}{*}{$\mathbb{Q}^4$}
			& \multirow{3}{*}{5}
			& GFQ (WB)          & 7.022E-13 & 7.275E-13 & 6.985E-13 & 7.397E-14 \\
			& & GFQ (non-WB)       & 2.338E-09 & 8.282E-09 & 8.282E-09 & 2.532E-09 \\
			& & Std (non-WB) & 8.685E-08 & 3.863E-07 & 3.863E-07 & 8.024E-08 \\
			\cmidrule(lr){2-7}
			& \multirow{3}{*}{10}
			& GFQ (WB)          & 2.648E-12 & 2.674E-12 & 2.689E-12 & 1.864E-13 \\
			& & GFQ (non-WB)       & 3.901E-11 & 1.472E-10 & 1.472E-10 & 4.110E-11 \\
			& & Std (non-WB) & 2.985E-09 & 1.236E-08 & 1.236E-08 & 2.678E-09 \\
			\bottomrule
		\end{tabular}
	\end{threeparttable}
\end{table}

\begin{figure}[t]
\begin{minipage}[t]{0.88\textwidth}
		\centering
	\begin{subfigure}[b]{0.325\textwidth}
		\centering
		\includegraphics[width=1.035\textwidth]{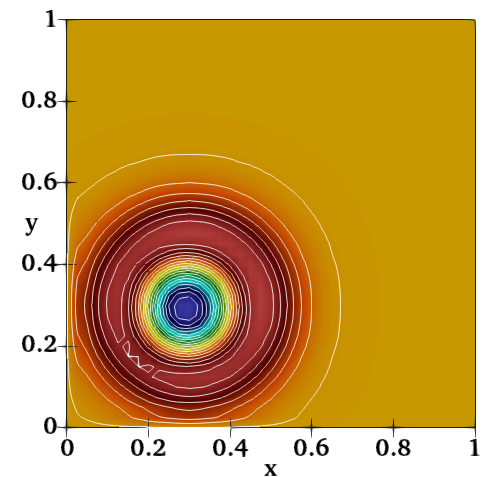}
		\caption{ GFQ (WB) }
	\end{subfigure}
	\begin{subfigure}[b]{0.325\textwidth}
		\includegraphics[width=1.035\textwidth]{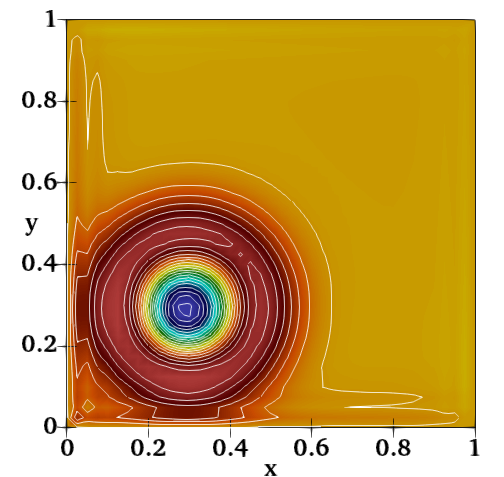}
		\caption{ GFQ (non-WB)}
	\end{subfigure}
	\begin{subfigure}[b]{0.325\textwidth}
		\includegraphics[width=1.035\textwidth]{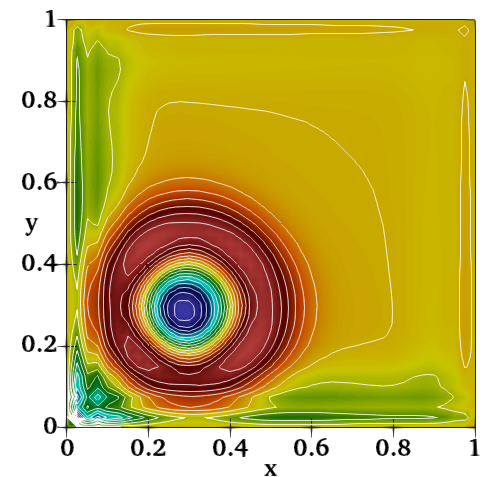}
		\caption{ SUPG-Std (non-WB)}
	\end{subfigure}	

\end{minipage}
\caption{Evolution of a small pressure perturbation on an isothermal equilibrium with $\mathbb Q^1$ elements. Pressure coloured plots  
overlaid with 20 equally spaced contour lines on a $40\times40$ mesh.
SUPG-GFQ (WB) (\textbf{left}), SUPG-GFQ (non-WB) (\textbf{middle}), and SUPG-Std (non-WB) (\textbf{right}) methods.}

\label{fig:perturb_euler}

\hfill
\begin{minipage}[t]{0.11\textwidth}
	\centering
	\vspace{-7. cm}
	\includegraphics[width=\textwidth, trim={18 0 8 0},clip]{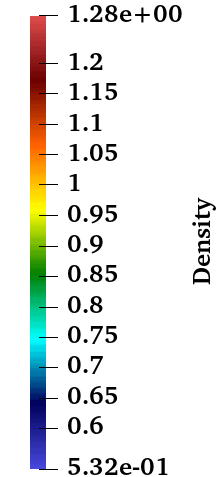}
\end{minipage}
\end{figure}

\begin{figure}[t]
\begin{minipage}[t]{0.88\textwidth}
	\centering
	\begin{subfigure}[b]{0.325\textwidth}
		\centering
		\includegraphics[width=1.02\textwidth]{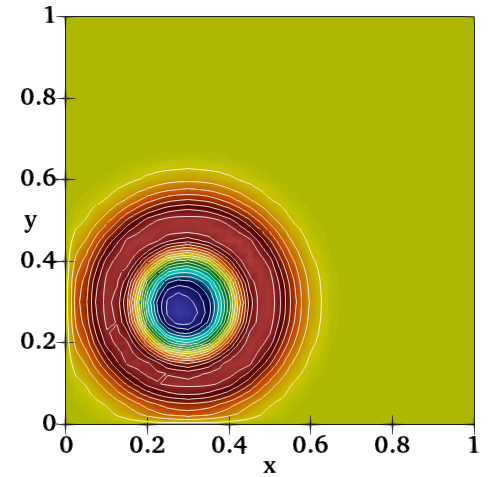}
		\caption{GFQ (WB) }
	\end{subfigure}
	\begin{subfigure}[b]{0.325\textwidth}
		\includegraphics[width=1.02\textwidth]{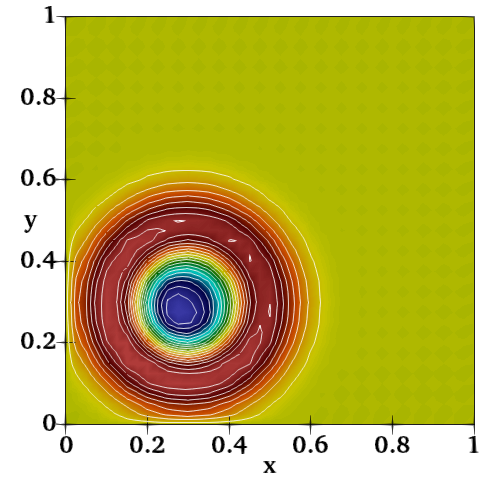}
		\caption{GFQ (non-WB)}
	\end{subfigure}
	\begin{subfigure}[b]{0.325\textwidth}
		\includegraphics[width=1.02\textwidth]{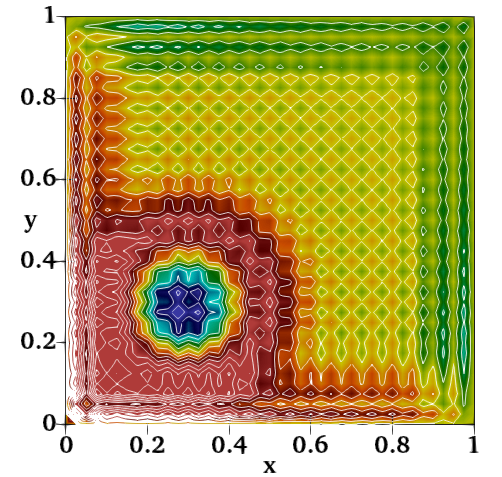}
		\caption{SUPG-Std (non-WB)}
	\end{subfigure}

	\centering
	\begin{subfigure}[b]{0.325\textwidth}
		\centering
		\includegraphics[width=1.02\textwidth]{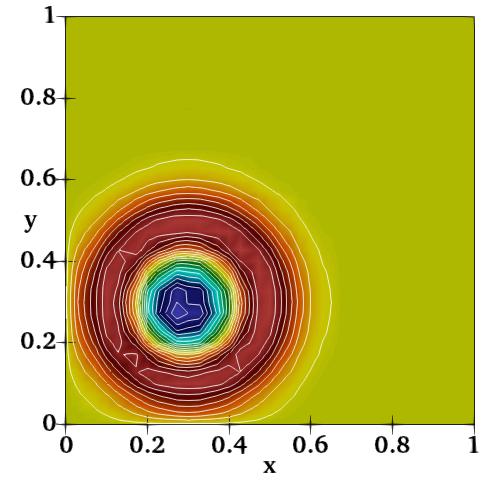}
		\caption{GFQ (WB) }
	\end{subfigure}
	\begin{subfigure}[b]{0.325\textwidth}
		\includegraphics[width=1.02\textwidth]{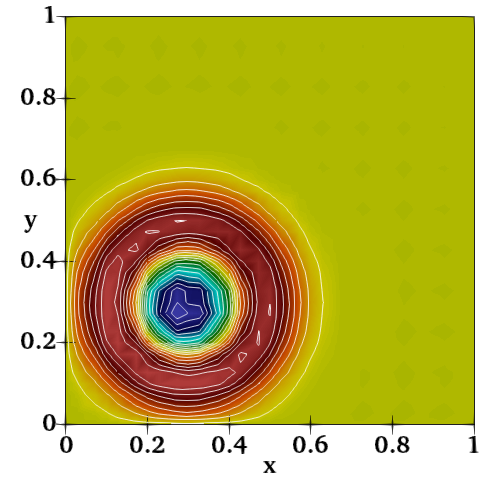}
		\caption{GFQ (non-WB)}
	\end{subfigure}
	\begin{subfigure}[b]{0.325\textwidth}
		\includegraphics[width=1.02\textwidth]{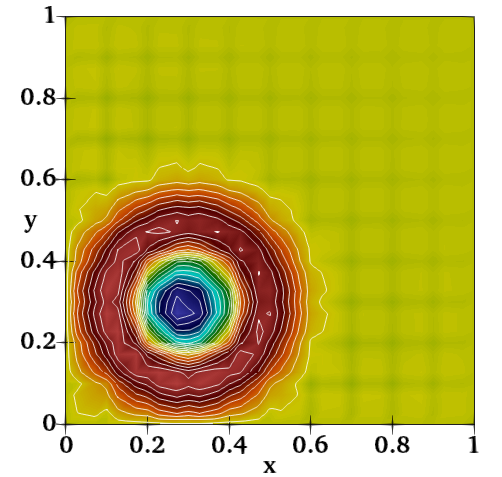}
		\caption{SUPG-Std (non-WB)}
	\end{subfigure}	
	
\end{minipage}

\caption{Evolution of a small pressure perturbation (with amplitude $A=2\times10^{-5}$) about an isothermal equilibrium with  $20\times20$ $\mathbb Q^2$ (\textbf{top}), $10\times10$ $\mathbb Q^3$ (\textbf{center}).
The panels show the pressure-perturbation field (colormap) overlaid with 20 equally spaced contour lines. SUPG-GFQ (WB) (\textbf{left}), SUPG-GFQ (non-WB) (\textbf{middle}), and SUPG-Std (non-WB) (\textbf{right}) methods.}

\label{fig:perturb_eulerP234}

\hfill
\begin{minipage}[t]{0.11\textwidth}
	\centering
	\vspace{-10.5 cm}
	\includegraphics[width=\textwidth, trim={18 0 17 0},clip]{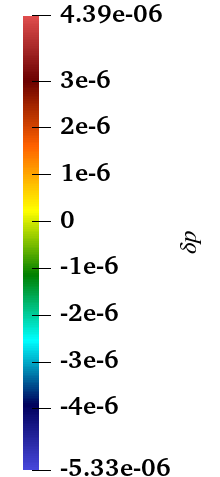}
\end{minipage}
\end{figure}

We now examine the evolution of a small pressure perturbation superposed on the isothermal hydrostatic equilibrium. 
Density and velocity are initialized at equilibrium, while the pressure is perturbed as
\begin{equation}
\label{eq:init-perturb}
p(x,y,0)=\bar p\,\exp\!\Big(-\tfrac{\bar\rho}{\bar p}\,\phi(x,y)\Big)
+ A\,\exp\!\Big(-100\,\tfrac{\bar\rho}{\bar p}\big[(x-0.3)^2+(y-0.3)^2\big]\Big),
\end{equation}
with amplitude $A=2\times10^{-5}$. 
The solution is advanced to $t=0.15\,\mathrm{s}$ under Dirichlet boundary conditions consistent with the equilibrium. 
In Figure~\ref{fig:perturb_euler}, we compare the pressure perturbations, defined by
\[
\delta p(x,y,t)=p(x,y,t)-\bar p\,\exp\!\Big(-\tfrac{\bar\rho}{\bar p}\,\phi(x,y)\Big),
\]
of the SUPG-GFQ (WB), SUPG-GFQ (non-WB), and SUPG-Std (non-WB) methods   on a $40\times40$ mesh and  $\mathbb Q^1$ elements.
The well-balanced SUPG-GFQ method captures perfectly, while the GFQ method shows some small spurious effects close to the lower boundaries,
The latter solution is however  way more accurate than the one obtained with the standard SUPG which shows unphysical waves of significant amplitude 
in a large part of the domain.

In Figure~\ref{fig:perturb_eulerP234}, we show the same comparison for $\mathbb Q^2$ and $\mathbb Q^3$  
elements on varying grid resolutions to have more or less the same number of total degrees of freedom. 
As the order increases, all methods increase their accuracy, but the SUPG-GFQ method still   outperforms the 
standard one.


.

		\subsection{Gravity driven Rayleigh–Taylor instability}
		\begin{figure}[t]
			
			\begin{minipage}{0.87\textwidth}
				\centering
				
				\includegraphics[width=0.325\textwidth]{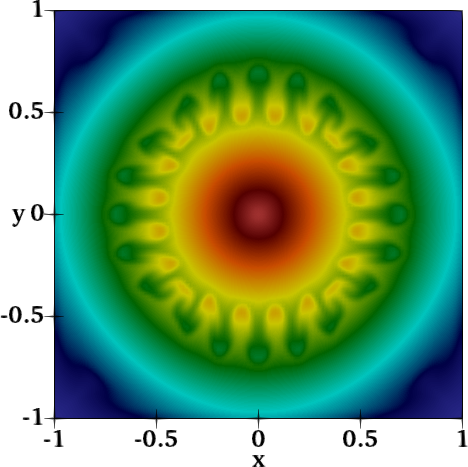}%
				\includegraphics[width=0.325\textwidth]{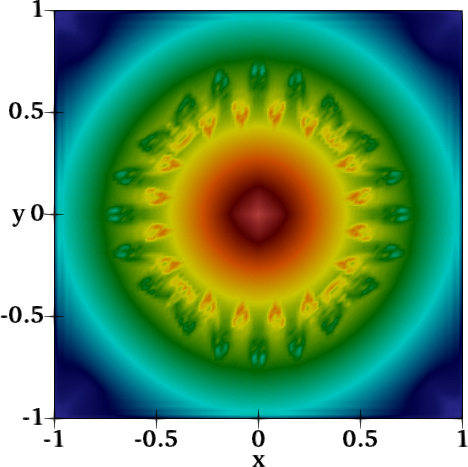}%
				\includegraphics[width=0.325\textwidth]{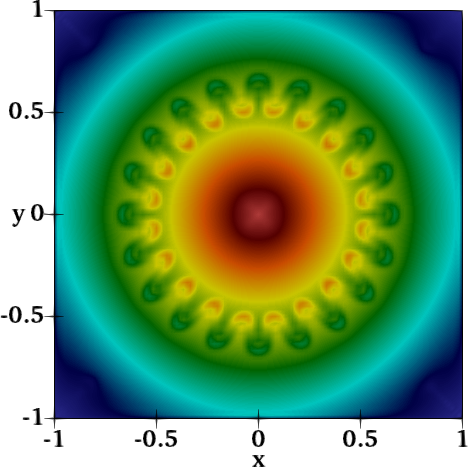}

			\end{minipage}
			\hfill
			\begin{minipage}{0.12\textwidth}
				\centering
				\hspace{-0.3cm}\includegraphics[width=\textwidth]{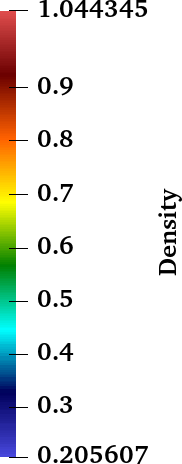}
			\end{minipage}

			\caption{Rayleigh-Taylor instability. Density field at 
				$t=4.1\,\mathrm{s}$ computed with $\mathbb{Q}^1$ elements on $120\times 120$ meshes. 
				\textbf{Left:} second-order HLLC; 
				\textbf{middle:} standard SUPG; 
				\textbf{right:} SUPG-GFQ.
			}
			
			\label{fig:all_RT_SUPG_GFQ}
			
		\end{figure}
		
		\begin{figure}[t]
			\begin{minipage}{0.86\textwidth}
				\begin{minipage}{0.495\textwidth}
					\centering  $120\times 120$\\
					\begin{minipage}{0.49\textwidth}
						\centering SUPG-Std			
					\end{minipage}\begin{minipage}{0.49\textwidth}
						\centering  SUPG-GFQ			
					\end{minipage}\\
					\half{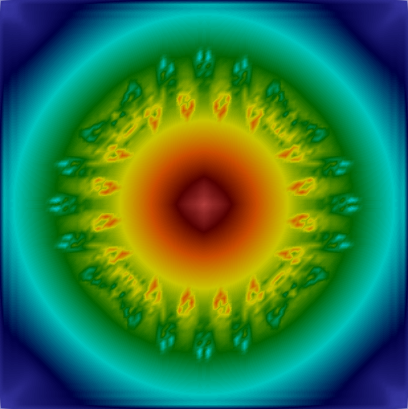}{l}{0.49}\,\half{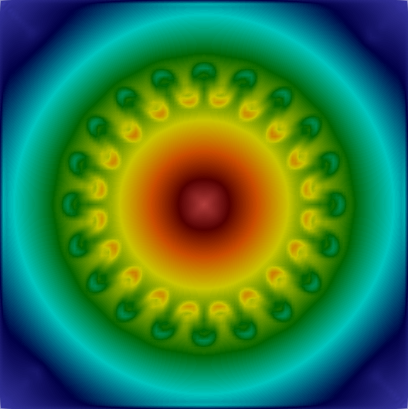}{r}{0.49}	\\	
				\end{minipage}
				\hfill
				\begin{minipage}{0.495\textwidth}
					\centering  $240\times 240$\\
					\begin{minipage}{0.49\textwidth}
						\centering SUPG-Std			
					\end{minipage}\begin{minipage}{0.49\textwidth}
						\centering  SUPG-GFQ			
					\end{minipage}\\
					\half{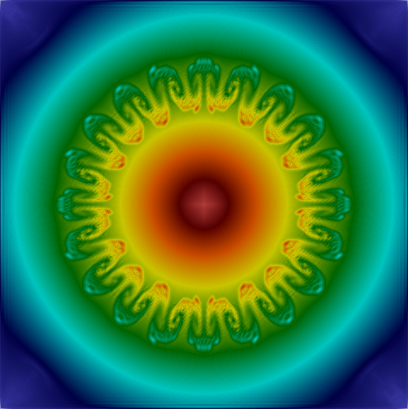}{l}{0.49}\,\half{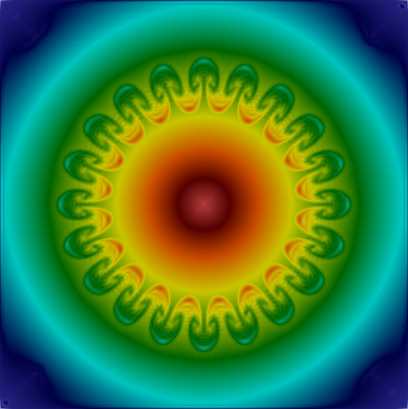}{r}{0.49}\\	
				\end{minipage}		
			\end{minipage}
			\begin{minipage}{0.13\textwidth}
				\vspace{0.65cm}
				\includegraphics[width=\textwidth]{./Color_bar_SUPG_cropped.png}
			\end{minipage}

			\caption{%
				Rayleigh-Taylor instability with gravity directed radially inward. 
				$t_{end}$= \(4.1\,\mathrm{s}\)) for mesh $120 \times 120$ (\textbf{left})  and $240 \times 240$ (\textbf{right}).  In each panel, the left half corresponds to the SUPG-Std method, while the right half corresponds to the SUPG-GFQ method.
			}
			\label{fig:RT_SUPG_GFQ1}
		\end{figure}

		We now consider a gravity-driven Rayleigh--Taylor instability under the radial gravitational potential $\phi(r)=r$. In the unperturbed state, the system admits the isothermal equilibrium
		\begin{equation}
			p(r) = \rho(r) = e^{-r}.
		\end{equation}
		A perturbation is introduced at a reference radius $r_0$ by prescribing the following initial conditions:
		\begin{equation}
			p(r,\theta) = 
			\begin{cases}
				e^{-r}, & r \leq r_0, \\[0.3em]
				e^{-\tfrac{r}{\alpha} + r_0 \tfrac{1-\alpha}{\alpha}}, & r > r_0,
			\end{cases}
			\qquad
			\rho(r,\theta) =
			\begin{cases}
				e^{-r}, & r \leq r_I(\theta), \\[0.3em]
				\dfrac{1}{\alpha}\, e^{-\tfrac{r}{\alpha} + r_0 \tfrac{1-\alpha}{\alpha}}, & r > r_I(\theta),
			\end{cases}
		\end{equation}
		where the perturbed interface is given by
		\begin{equation}
			r_I(\theta) = r_0 \left(1 + \eta \cos(k\theta)\right),
			\qquad
			\alpha = \frac{e^{-r_0}}{e^{-r_0} + \Delta \rho}.
		\end{equation}
		
		With this construction, the pressure is continuous while the density exhibits a jump of amplitude 
		$\Delta \rho$ at $r = r_I(\theta)$. Following \cite{leveque1999wave}, the parameters are chosen as:
		\begin{equation}
			r_0 = 0.5, 
			\qquad \Delta \rho = 0.1, 
			\qquad k = 20, 
			\qquad \eta = 0.02.
		\end{equation}
		The computational domain is $[-1,1]\times[-1,1]$.  This configuration gives rise to the Rayleigh-Taylor instability along the perturbed interface, where plume-like 
		structures develop over time. Away from the interface, the well-balanced scheme preserves the equilibrium state and prevents spurious oscillations near the boundaries.
		
		In Figure~\ref{fig:all_RT_SUPG_GFQ}, we compare the density fields obtained with the second-order HLLC-FV, the $\mathbb{Q}^1$ SUPG-Std, and the $\mathbb{Q}^1$ SUPG-GFQ methods at $t=4.1\,\mathrm{s}$ on a 120x120 mesh. It is clear that the SUPG-GFQ is the only method that keeps the radial symmetry already for  coarse meshes. 
		A one to one comparison with the  SUPG-Std method is reported on figure \ref{fig:RT_SUPG_GFQ1}. We can clearly see that the 
		classical method  suffers from mesh imprinting which is   absent in the 
		SUPG-GFQ  solutions. The latter is able to capture the flow structures without spurious oscillations and with almost no mesh imprinting. 
		The HLLC-FV method,  is the one showing the most numerical diffusion, even when refining the mesh. 

		\subsection{Low Mach Behavior: Rising thermal bubble}
		We consider a standard benchmark in atmospheric flows: the evolution of a warm, localized potential temperature anomaly embedded in a neutrally stratified, hydrostatic environment \cite{robert1993bubble,giraldo2008study}.
		
		Since the perturbed parcel is warmer (and therefore lighter) than its surroundings, buoyancy accelerates it upward; shear then deforms the bubble, eventually producing a characteristic mushroom-shaped structure.
		
		The quiescent reference state is hydrostatic and uniform in potential temperature:
		\begin{equation}
			\theta_0 = 300~\mathrm{K}, 
			\qquad p_0 = 10^{5}~\mathrm{Pa}, 
			\qquad \rho_0 = 1.1612055~\mathrm{kg\,m^{-3}}.
		\end{equation}
		With the gravitational acceleration constant $g=9.8~\mathrm{m\,s^{-2}}$ acting in the negative $y$–direction and $c_p=\frac{\gamma}{\gamma-1}R$, the Exner pressure associated with the base state is
		\begin{equation}
			\Pi(y) \;=\; 1 - \frac{g\,y}{c_p\,T_0},
		\end{equation}
		and the hydrostatic density follows from the ideal gas law in terms of the potential temperature:
		\begin{equation}
			\rho(x,y) \;=\; \frac{p_0}{R\,\theta(x,y)}\,\Pi(y)^{\tfrac{1}{\gamma-1}} .
		\end{equation}
		
		The warm bubble is introduced as a cosine perturbation in potential temperature,
		\begin{equation}
			\Delta\theta(x,y) \;=\;
			\begin{cases}
				\displaystyle \frac{\theta_c}{2}\left(1+\cos\!\Big(\frac{\pi r}{r_c}\Big)\right), & r \le r_c,\\[0.7em]
				0, & r>r_c,
			\end{cases}
			\qquad 
			r=\sqrt{(x-x_c)^2+(y-y_c)^2},
		\end{equation}
		with center $(x_c,y_c)=(500,350)$, radius $r_c=250\,\mathrm{m}$, and amplitude $\theta_c = \SI{0.5}{\celsius}$. 
		The full initial potential temperature is $\theta(x,y)=\theta_0+\Delta\theta(x,y)$, and the velocity field is initially at rest:
		\begin{equation}
			u(x,y)=v(x,y)=0,
		\end{equation}
		while the total energy is consistent with the thermodynamic state,
		\begin{equation}
			\rho E(x,y)=\rho(x,y)\,\frac{R}{\gamma-1}\,\theta(x,y)\,\Pi(y).
		\end{equation}
		
		The computational domain is $(x,y)\in[0,1000]\times[0,1000]\,\mathrm{m}$ with slip (no-flux) boundary conditions on all sides, and the simulation is run up to $t_{\mathrm{end}}=700~\mathrm{s}$. 
		This test is widely used to assess the ability of numerical schemes to preserve the hydrostatic equilibrium away from the bubble and to accurately capture the buoyancy-driven ascent and roll-up of the thermal without generating spurious imbalances \cite{marras2015stabilized,giraldo2008study}.
		
		In figure~\ref{fig:bubble}, we show the time evolution of the potential temperature field computed with the SUPG-GFQ method on a $150\times 150$ mesh, as well as a comparison between the SUPG-Std and SUPG-GFQ methods at $t=700\,\mathrm{s}$ for $60\times 60$ and $120\times 120$ meshes. In the time evolution, we see how the bubble rises and rolls up, forming the characteristic mushroom shape, as expected. 
		The comparison between the two methods at $t=700\,\mathrm{s}$ shows that the SUPG-Std method produces a less regular shape of the mushroom, even for the fine meshgrid, where the shape is characterized by weird wiggles.  
		This highlights the improved accuracy of the SUPG-GFQ method in low Mach number flows with buoyancy effects.
		
		\begin{figure}[t]
			\begin{minipage}{\textwidth}
				\centering
				$150\times 150$\\
				\begin{minipage}{0.24\textwidth}
					\centering
					$t=0\,\mathrm{s}$
				\end{minipage}
				\begin{minipage}{0.24\textwidth}
					\centering
					$t=233.3\,\mathrm{s}$
				\end{minipage}
				\begin{minipage}{0.24\textwidth}
					\centering
					$t=466.7\,\mathrm{s}$
				\end{minipage}
				\begin{minipage}{0.24\textwidth}
					\centering
					$t=700\,\mathrm{s}$
				\end{minipage}\\
				\includegraphics[width=0.24\textwidth]{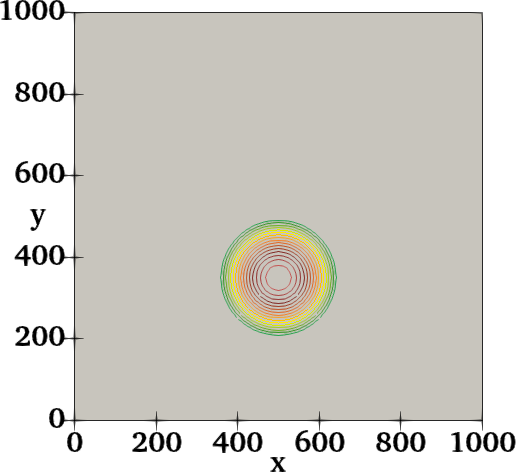}
				\includegraphics[width=0.24\textwidth]{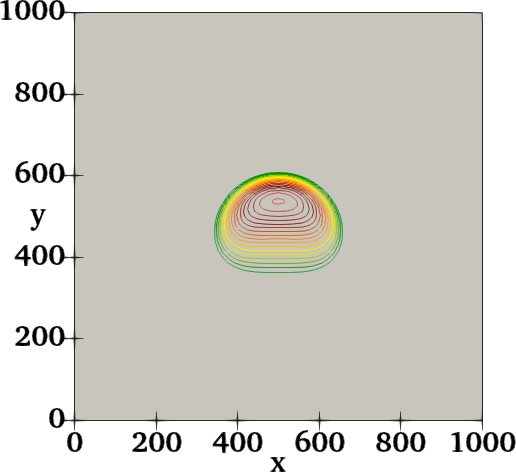}
				\includegraphics[width=0.24\textwidth]{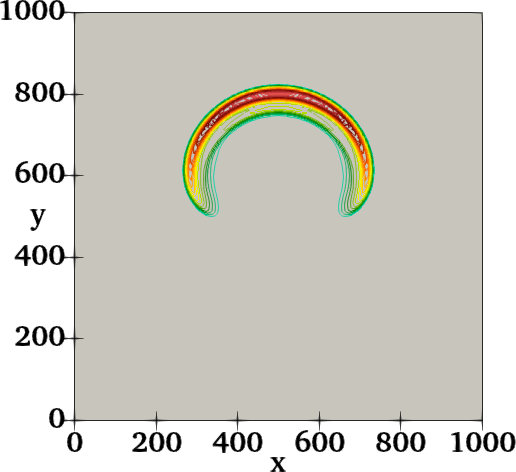}
				\includegraphics[width=0.24\textwidth]{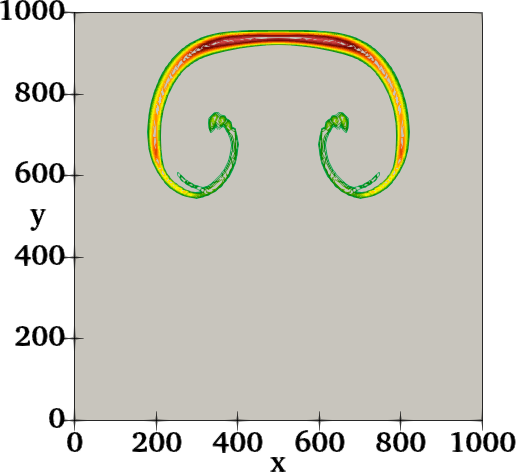}
			\end{minipage}\\[2mm]
			\begin{minipage}{0.87\textwidth}
				\centering
				\begin{minipage}{0.49\textwidth}
					\centering
					$60\times 60$\\
					\begin{minipage}{0.5181818\textwidth}
						\centering SUPG-Std
					\end{minipage}
					\begin{minipage}{0.4618181\textwidth}
						\centering SUPG-GFQ
					\end{minipage}\\
					\includegraphics[width=0.5181818\textwidth,trim={0 20 254 0},clip]{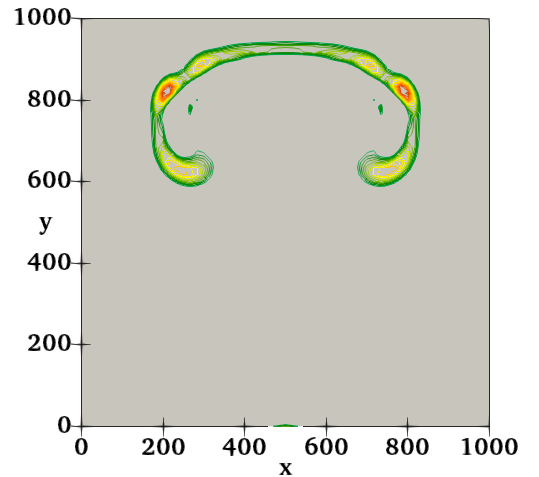}
					\includegraphics[width=0.4618181\textwidth,trim={285 20 0 0},clip]{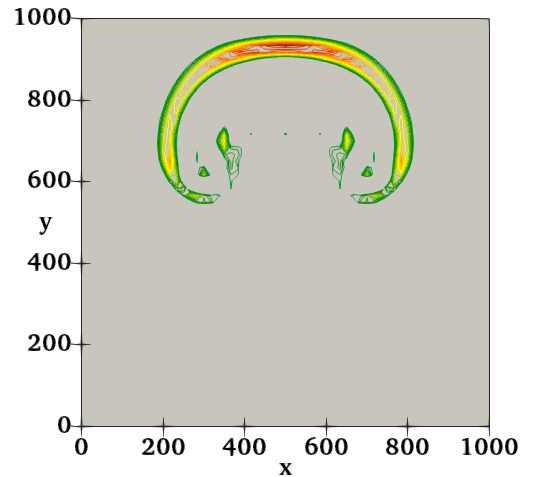}
				\end{minipage}
				\begin{minipage}{0.49\textwidth}
					\centering
					$120\times 120$\\
					\begin{minipage}{0.5181818\textwidth}
						\centering SUPG-Std
					\end{minipage}
					\begin{minipage}{0.4618181\textwidth}
						\centering SUPG-GFQ
					\end{minipage}\\
					\includegraphics[width=0.5181818\textwidth,trim={0 20 254 0},clip]{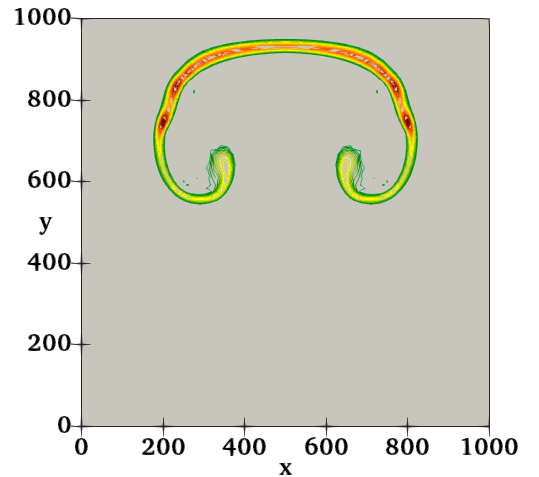}
					\includegraphics[width=0.4618181\textwidth,trim={285 20 0 0},clip]{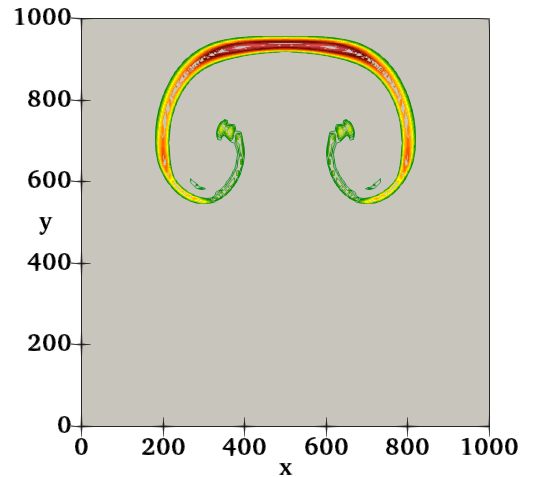}
				\end{minipage}
			\end{minipage}
			\hfill
			\begin{minipage}{0.12\textwidth}
				\centering
				\vspace{0.55cm}
				\includegraphics[width=\textwidth]{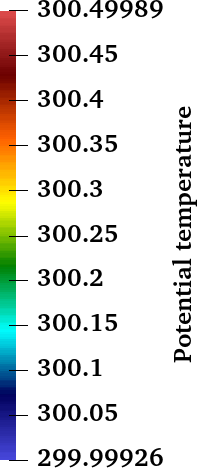}
			\end{minipage}
			\caption{
				Rising thermal bubble, potential temperature. \textbf{Top} line, time evolution for SUPG-GFQ on a $150\times 150$ mesh for times $t\in \lbrace 0, 700/3, 2\cdot 700/3, 700 \rbrace$.
				\textbf{Bottom} line comparison at $t = 700\,\text{s}$ 
				using $\mathbb{Q}^1$ elements on $60\times 60$ (\textbf{left}) and $120\times 120$ (\textbf{right}) meshes of SUPG-Std (\textbf{left columns}) and SUPG-GFQ (\textbf{right columns}).
			}\label{fig:bubble}
		\end{figure}

		\subsection{Shock-vortex interaction}

To show the applicability of our method to flows with higher Mach numbers, we now consider   a  
 classical benchmark  involving the interaction between a localized vortex and a stationary normal shock.
 We consider in particular  the composite-vortex setting considered in \cite{rault2003shock}.
The computational domain is $[0,2]\times[0,1]$. A stationary shock is located at $x_s=0.5$, with upstream state
\[
\rho_u=1,\qquad u_u=M_s\sqrt{\gamma},\qquad v_u=0,\qquad p_u=1.
\]
Since we do not add any oscillation control method to the scheme, we limit ourself to the case of a mild shock obtained for  $M_s=1.1$. The downstream state is determined from the Rankine-Hugoniot relations associated with the stationary shock. A counterclockwise vortex is centered at $(x_c,y_c)=(0.25,0.5)$ and is characterized by the inner and outer radii $a=0.075$ and $b=0.175$, respectively. The angular velocity is given by $v_\theta = v_m \frac{r}{a}$ for $r\leq a$, $v_\theta = v_m \frac{a}{a^2-b^2}\frac{r^2-b^2}{r}$ for $a<r\leq b$, and $v_\theta=0$ for $r>b$, where $v_m$ is the maximum angular velocity.
Its strength is measured by the parameter $M_v=v_m/\sqrt{\gamma}$, and we set $M_v=0.9$.  The corresponding temperature field is obtained from radial equilibrium, and the density and pressure inside the vortex are then recovered from the isentropic relations, consistently with the construction proposed in \cite{rault2003shock}. At the boundary level, the upstream state is imposed at the left boundary, the downstream state at the right boundary, and slip-wall boundary conditions are applied at the top and bottom boundaries.

Figure~\ref{fig:shock_vortex} shows the numerical solution at time $t=0.4\,\mathrm{s}$ for $\mathbb{Q}^1$, $\mathbb{Q}^2$, $\mathbb{Q}^3$, and $\mathbb{Q}^4$ approximations for around $200\times 100$ degrees of freedom. All polynomial degrees reproduce the main qualitative features of the shock--vortex interaction. The visual differences between the various approximations remain relatively modest, indicating that the method already captures the overall flow structure well at low order.  Some oscillations are of course present, since we do not have any artificial dissipation to control them, but the simulations remain perfectly stable.
Moreover, higher-order approximations   provide a slightly smoother description of some localized features and further reduce  mesh imprinting.

\begin{figure}[t]
	\begin{minipage}[t]{0.87\textwidth}
		\centering
		
		\begin{subfigure}[b]{0.49\textwidth}
		\centering $\mathbb{Q}^1$, $200\times100$\\[-0.2cm]
			\includegraphics[width=0.87\textwidth]{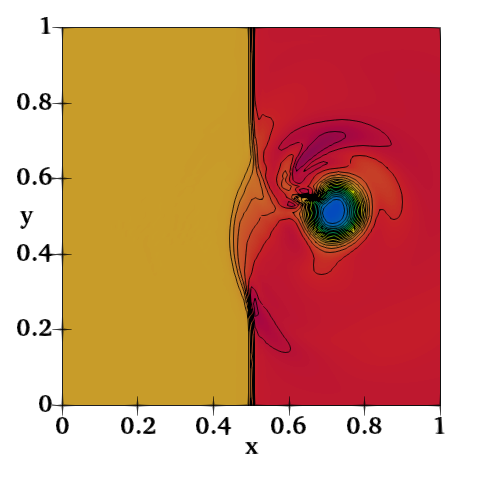}
				\vspace{0.4cm}
		\end{subfigure}
		\begin{subfigure}[b]{0.49\textwidth}
			\centering $\mathbb{Q}^2$, $100\times50$\\[-0.2cm]
			\includegraphics[width=0.87\textwidth]{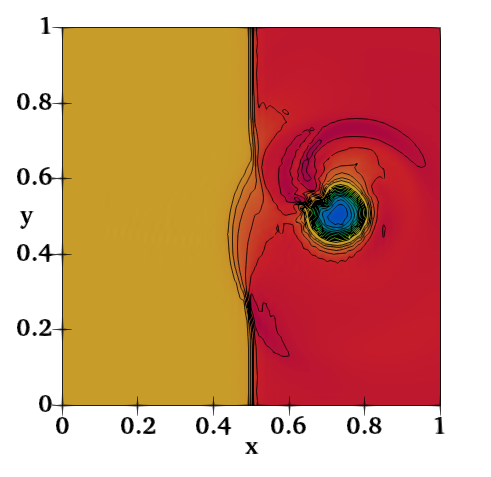}
				\vspace{0.4cm}
		\end{subfigure}
		
		\vspace{-0.3cm}
		
		\begin{subfigure}[b]{0.49\textwidth}
			\centering $\mathbb{Q}^3$, $66\times33$\\[-0.2cm]
			\includegraphics[width=0.87\textwidth]{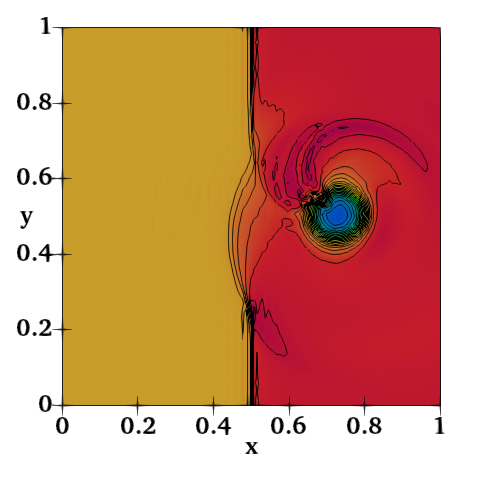}
		\end{subfigure}
		\begin{subfigure}[b]{0.49\textwidth}
			\centering $\mathbb{Q}^4$, $50\times25$\\[-0.2cm]
			\includegraphics[width=0.87\textwidth]{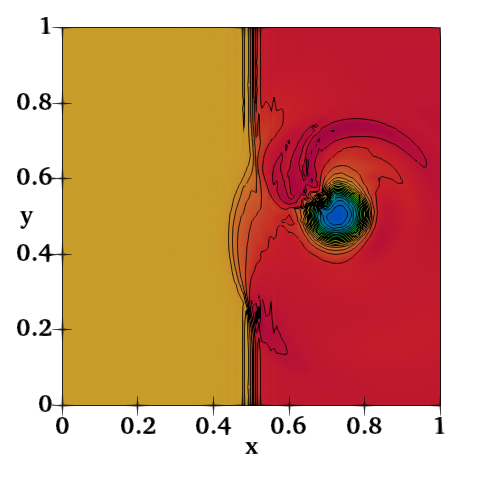}
		\end{subfigure}
		
		
	\end{minipage}
	\hfill
	\begin{minipage}[t]{0.123\textwidth}
		\vspace{-3.3cm}
		\hspace{-0.7cm}\includegraphics[width=1.05\textwidth, trim={10 0 7 0},clip]{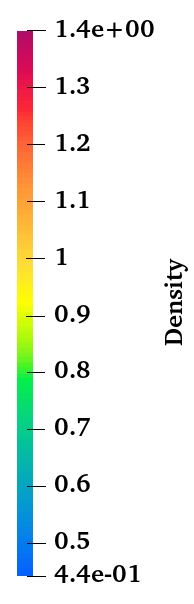}
	\end{minipage}
	
	\caption{Shock-vortex interaction test at time $t=0.4\,\mathrm{s}$. Density field (colormap) overlaid with contour lines, obtained with the SUPG-GFQ method using $\mathbb{Q}^1$ elements ({top left}), and $\mathbb{Q}^2$, $\mathbb{Q}^3$, and $\mathbb{Q}^4$ elements ({top right}, {bottom left}, and {bottom right}), respectively. $200\times 100$ degrees of freedom in all cases. }
	\label{fig:shock_vortex}
\end{figure}

	\section{Conclusions}\label{sec:conclusion}
	
We have presented a framework to obtain very high order stationarity-preserving finite element discretizations for balance laws.
The paper extends the ideas of \cite{barsukow2025structure,barsukow2025stationaritypreservingnodalfinite,barsukow2025genuinely} in several ways. 
First, we provide a  precise definition of  stationarity preservation 
and some general guidelines to check whether a scheme fits into this category or not. We have provided
a more detailed discussion on how  the underlying  method is obtained and what its properties are.
In particular, we have clarified that the global flux quadrature can be seen as a pre-processing
of the fluxes, mapping the continuous space of  $\mathbb{Q}_K$ vectors onto a local space, which has much in common with the usual Raviart-Thomas $\mathbb{RT}_{K-1}$ space.
Several implications of this projection are discussed. 

We have then applied this procedure to nonlinear systems of balance laws, extending the stationarity preserving SUPG method proposed for linear systems in \cite{barsukow2025structure,barsukow2025stationaritypreservingnodalfinite}. 
We have discussed and demonstrated many properties 
of the scheme, starting from local conservation, to consistency, to the relations with low Mach compliant methods.
Finally, we have applied the scheme to the Euler equations with gravitational potential, for which we also propose a simple modification
of the source terms to preserve hydrostatic  isothermal equilibria at machine precision.
The numerical results leave no doubt on the  enhancements obtained with the approach proposed,
and on the superiority of the stationarity preserving SUPG method compared to the standard one.

Ongoing works involve the application to other systems and to three-dimensional flows, the
combination with well balancing techniques as those proposed  e.g. in \cite{GOMEZBUENO2021125820}, and extensions to complex geometries. 
A tricky aspect is the design of non-linear methods allowing to capture strong discontinuities, and allowing the preservation
of the discrete kernel. Preliminary work in this direction is discussed in \cite{AbgLiuRic25}, based on 
  a blending strategy which accounts for the stationary states.
The extension to genuinely unstructured meshes is another exciting challenge left for forthcoming papers.

\section*{Acknowledgements} 

Work started when MZ was postdoctoral associate in the team CARDAMOM at the  Inria research center at University of Bordeaux. 
MR is member of the team CARDAMOM at the  Inria research center at University of Bordeaux. 
DT is a member of GNCS group of INdAM.
The authors would like to acknowledge many insightful discussions with W. Barsukow, M. Ciallella, and V. Perrier.

\appendix
\section{Consistency estimate proof}\label{app:consistency}
In this appendix, we provide a proof of proposition \ref{prop:consistency_equilibria}. The proof is obtained  following developments similar to those of section~\ref{sec:local_conservation},
and classical for residual distribution methods \cite{AR:17,amr2025}. First, we note that for a smooth stationary solution $\mathbf{W}^e$, we can  define continuous counterparts of the 
potentials \eqref{eq:GF1} and \eqref{eq:GF10} and  trivially
$$
\partial_{x^1x^2}(\mathcal{F}_{1}^e  +  \mathcal{F}_{2}^e +  \mathcal{S}^e ) = \nabla\cdot\mathbf{F}^e-\mathbf{S}^e =0.
$$
In particular, by analogy with \eqref{eq:Phi_gf} we define the array of integrated divergence minus source :  
$$
[\Phi^e]_{pq} + \mathcal{S}^e_{pq}  = \int_{x^2_{j,0} }^{x^2_{j,q} } (\mathcal{F}_{1}^e-\mathcal{F}_{1}^e|_{x^1=x^1_{i,0}})\mathrm{d}x^2
+ \int_{x^1_{i,0} }^{x^1_{i,p} } (\mathcal{F}_{2}^e-\mathcal{F}_{2}^e|_{x^2=x^2_{j,0}})\mathrm{d}x^1
-\int_{x^2_{j,0} }^{x^2_{j,q} }\int_{x^1_{i,0} }^{x^1_{i,p} }\mathbf{S}^e\,d\mathbf{x} = 0.
$$
Using now  the expression \eqref{eq:fluct_su_gf} of the steady fluctuations, we  manipulate the truncation error \eqref{eq:TE0} as follows:
$$
\begin{aligned}
\mathsf{E}_h =& \int_{\Omega}\partial_{x^1x^2}\psi_{h}  \;( \mathcal{F}_{1,h}  +  \mathcal{F}_{2,h} +  \mathcal{S}_{h}   )\,d\mathbf{x} 
+ \sum_\alpha\sum_{E_{ij}\in\Omega}  (\psi_{\alpha} -\bar \psi_{E_{ij}}) \int_{E_{ij}} \nabla\varphi_\alpha\cdot \mathbf{A}\,\tau\, 
\partial_{x^1x^2}\Phi_h \,d\mathbf{x} \\
=&  \int_{\Omega}\partial_{x^1x^2}\psi_{h}  \;( \mathcal{F}_{1,h}  -\mathcal{F}_{1}^e +  \mathcal{F}_{2,h} -\mathcal{F}_{2}^e+  \mathcal{S}_{h} - \mathcal{S}^e   )\,d\mathbf{x}  \\
+& \sum_\alpha\sum_{E_{ij}\in\Omega}  (\psi_{\alpha} -\bar \psi_{E_{ij}}) \int_{E_{ij}} \nabla\varphi_\alpha\cdot \mathbf{A}\,\tau\, 
\partial_{x^1x^2}(\Phi_h - \Phi^e) \,d\mathbf{x} 
\end{aligned}
$$ 
having used the compactness of the support of $\psi$  in the first term.
We now  recall that as stated by  Proposition~\ref{prop:Lobatto_integration},  both the $\mathcal{F}_{1,h}$ and $\mathcal{F}_{2,h}$ potentials are
the result of integrating  \eqref{eq:bigF7} with the  LobattoIIIA collocation method \cite{Hairer1993}.
For the source term, the same arguments used in proposition \ref{prop:Lobatto_integration} allow to state that $\mathcal{S}(\mathbf{x^*})$ can be seen as the repeated application of the LobattoIIIA method 
in the $x^1$  and the $x^2$ direction:
$$
\dfrac{d}{dx^2}\partial_{x^1}\mathcal{S} = \mathbf{S} \;\; \overset{\text{LobattoIIIA}}{\longrightarrow}    \;\;  \dfrac{d}{dx^1} \mathcal{S}  =  \partial_{x^1}\mathcal{S}    \;\; \overset{\text{LobattoIIIA}}{\longrightarrow}    \;\;  \mathcal{S}.
$$
We then recall that the consistency of the LobattoIIIA method is   of order $h^{K+2}$ for internal stages, and order $h^{2K}$ for the endpoints of each integration interval. 
Thus,  for any point $\mathbf{x}^\star = (x^{1,\star},x^{2,\star})$ in the Lobatto nodal grid,
 we can write 
$$
\begin{aligned}
\mathcal{F}_{1,h}(t^n,x^1_{\pi}, x^{2,\star}) -\mathcal{F}_{1}^e(t^n,x^{1}_{\pi}, x^{2,\star}) = &\; \mathcal{O}(h^{K+2}),\quad\text{for any }x^{1}_{\pi},\\
\mathcal{F}_{2,h}(t^n,x^{1,\star},x^2_{\kappa}) -\mathcal{F}_{2}^e(t^n, x^{1,\star}, x^{2}_{\kappa}) =  &\; \mathcal{O}(h^{K+2}),\quad\text{for any }x^{2}_{\kappa},\\
\mathcal{S}_{h}(t^n,x^{1,\star},x^{2,\star}) -\mathcal{S}^e(t^n,x^{1,\star}  , x^{2,\star})  =  &\; \mathcal{O}(h^{K+2}).
\end{aligned}
$$
With this result we can immediately estimate the first term, simply by  noting that for $\psi$ a $C^1$ function of each of its arguments, then  $\partial_{x^1x^2}\psi_h$ is bounded, so 
$$
\left\lvert \int_{\Omega}\partial_{x^1x^2}\psi_{h}  \;( \mathcal{F}_{1,h}  -\mathcal{F}_{1}^e +  \mathcal{F}_{2,h} -\mathcal{F}_{2}^e+  \mathcal{S}_{h} - \mathcal{S}^e   )\,d\mathbf{x} \right\rvert \le C_1 h^{K+2}.
$$
For the second term, first we note  that due to the double integral  $ \Phi_h -\Phi^e$   contains terms of the type
$$
\int_{x^2_{j,0} }^{x^2_{j,q} } (\mathcal{F}_{1,h}^e-\mathcal{F}_{1}^e)\mathrm{d}x^2
$$
so we can estimate it as $| \Phi_h -\Phi^e|=\mathcal{O}(h^{K+3})$.  Proceeding classically \cite{AR:17,amr2025} we use the fact that there are $h^{-2}$ elements in the domain,
the smoothness of $\psi$ and the definition of $\tau$, proportional to $h$, to conclude also that 
$$
\left\lvert  \sum_\alpha\sum_{E_{ij}\in\Omega}  (\psi_{\alpha} -\bar \psi_{E_{ij}}) \int_{E_{ij}} \nabla\varphi_\alpha\cdot \mathbf{A}\,\tau\, 
\partial_{x^1x^2}(\Phi_h - \Phi^e) \,\mathrm{d}\mathbf{x},
\right\rvert \le C_2 h^{K+2}
$$
which achieves the proof.

\bibliographystyle{siam}
\bibliography{Ref1}
	
\end{document}